\documentclass[12pt]{amsart}

\usepackage{geometry}
\newcommand{\paper}[1]{#1}
\newcommand{\thesis}[1]{}
\newcommand{\primarydivision}[1]{\section{#1}}
\newcommand{\secondarydivision}[1]{\subsection{#1}}
\newcommand{\primarydivname}{Section}
\newcommand{\primarydivnamepl}{Sections}
\newcommand{\secondarydivname}{Subsection}

\usepackage{amsmath, amsthm, amssymb, amsfonts, tikz, amscd}

\usepackage{url}

\newtheorem{proposition}{Proposition}
\newtheorem{lemma}{Lemma}
\newtheorem{corollary}{Corollary}
\newtheorem{theorem}{Theorem}

\newcommand{\VP}{\mathcal{V}}
\newcommand{\WP}{\mathcal{W}}

\newcommand{\PP}{\mathcal{P}_2}
\newcommand{\BPP}{\mathcal{BP}_2}

\newcommand{\IBPP}{\mathcal{BP}_2^\IndexSet}
\newcommand{\IBPPi}{\IBPP(\infty)}
\newcommand{\IBPPn}{\IBPP(\bn,\bzero)}
\newcommand{\IPP}{\mathcal{P}_2^\IndexSet}
\newcommand{\IPPi}{\IPP(\infty)}
\newcommand{\PPi}{\PP(\infty)}
\newcommand{\BC}{\mathbb{C}}
\newcommand{\BZ}{\mathbb{Z}}
\newcommand{\BN}{\mathbb{N}}
\newcommand{\BR}{\mathbb{R}}
\newcommand{\SN}{\mathcal{N}}
\newcommand{\SX}{\mathcal{X}}

\newcommand{\BO}{\mathcal B}

\newcommand{\vf}{\mathbf{f}}
\newcommand{\Pt}{\mathbf{t}}

\newcommand{\FS}{\mathcal F}

\newcommand{\HH}{\mathcal H}
\newcommand{\HK}{\mathcal K}

\newcommand{\COA}{\mathcal C}
\newcommand{\FOA}{\mathcal A}
\newcommand{\COAI}{\mathcal C^{\IndexSet}}
\newcommand{\FOAI}{\mathcal A^{\IndexSet}}

\newcommand{\IndexSet}{\mathcal{I}}

\DeclareMathOperator{\wlim}{w-lim}
\DeclareMathOperator{\crossings}{cr}
\DeclareMathOperator{\linspan}{span}
\newcommand{\bn}{\mathbf{n}}
\newcommand{\bzero}{\mathbf{0}}
\newcommand{\bpd}{\mathbf{d}}
\newcommand{\bw}{\mathbf{w}}
\newcommand{\bone}{\mathbf{1}}

\newcommand{\basis}{\mathcal{B}}

\newcommand{\TensorIndices}{\mathcal J}
\newcommand{\elem}{\mathcal{E}}

\newcommand{\NO}{\mathbf{N}}

\newcommand{\SP}{\mathfrak{P}}

\newcommand{\nc}{\mathfrak{c}}
\newcommand{\na}{\mathfrak{a}}

\newcommand{\vl}{\mathfrak{p}}
\newcommand{\vls}{\mathfrak{r}}
\newcommand{\Comp}{\mathfrak{C}}
\newcommand{\ZB}{Z}
\newcommand{\MS}{\mathcal{Y}}

\newcommand{\vcequiv}{\overset{\VP,c}{\sim}}

\theoremstyle{definition}
\newtheorem{notation}{Notation}
\newtheorem{definition}{Definition}
\newtheorem{example}{Example}

\theoremstyle{remark}
\newtheorem{remark}{Remark}

\newcommand{\lprimarydivname}{\MakeLowercase{\primarydivname}}
\newcommand{\lprimarydivnamepl}{\MakeLowercase{\primarydivnamepl}}
\newcommand{\lsecondarydivname}{\MakeLowercase{\secondarydivname}}

\newcommand{\dom}{\mathcal{D}}
\newcommand{\subo}{\mathcal{S}}

\title{Remarks on multi-dimensional noncommutative generalized Brownian motions}
\author{Adam Merberg}
\address{Department of Mathematics,
  University of California,
  Berkeley, CA, USA 94720}
\email{amerberg@math.berkeley.edu}
\begin{document}
\begin{abstract}

We consider certain questions pertaining to noncommutative generalized Brownian motions with multiple processes.
We establish a framework for generalized Brownian motion with multiple processes similar to that defined by Gu{\c{t}}{\u{a}} and prove multi-dimensional analogs of some results of Gu{\c{t}}{\u{a}} and Maassen.
We then consider examples of processes indexed by a two-element set and characterize the function on $\IndexSet$-indexed pair partitions associated via the $\IndexSet$-indexed generalized Brownian motion construction to certain pairs of representations connected to certain spherical representations of infinite symmetric groups.
In doing so, we generalize the notion (introduced by Bo{\.z}ejko and Gu{\c{t}}{\u{a}}) of the cycle decomposition of a pair partition.
We then generalize Gu{\c{t}}{\u{a}}'s $q$-product of generalized Brownian motions to a product corresponding to a (possibly infinite) matrix $(q_{ij})$ and show that this $q_{ij}$-product satisfies a central limit theorem.
\end{abstract}
\maketitle
\primarydivision{Introduction}
\label{pri:intro}
Bo{\.z}ejko and Speicher initiated the study of noncommutative generalized Brownian motions, introducing operators satisfying an interpolation between Fermionic and Bosonic commutation relations \cite{BS1991}.
Specifically, for $q\in [-1,1]$ and a complex Hilbert space $\HH$, they constructed a $q$-twisted Fock space $\FS_q(\HH)$ with creation operators
$c^*(f)$ and annihilation operators $c(f)$ for $f\in\HH$ satisfying the relations
\begin{equation}
 c(f)c^*(g)-qc^*(g)c(f)=\left<f,g\right>\cdot 1.
\end{equation}
Subsequently, they developed a broader framework of generalized Brownian motion which incorporated this example \cite{BS1996}.
In this general framework, one considers the algebra obtained by applying the GNS construction to the  free tensor algebra of a real Hilbert space $\HH$ with certain states, called Gaussian states, associated to a class of functions, called positive definite, on pair partitions via a pairing prescription.

{Gu{\c{t}}{\u{a}} and Maassen further explored this notion of generalized Brownian motion \cite{GM1,GM2}.
They showed that Gaussian states $\rho_\Pt$ can be alternatively characterized by sequences of complex Hilbert spaces $(V_n)_{n=1}^\infty$ with densely defined maps $j_n:V_n\to V_{n+1}$ and representations $U_n$ of the symmetric group $S_n$ on $V_n$ satisfying $j_n\cdot U(\pi)=U(i_n(\pi))\cdot j_n$ where $i_n:S_n\to S_{n+1}$ is the inclusion arising from the map $\{1,\ldots, n\}\hookrightarrow \{1,\ldots, n+1\}$, data which give rise to a symmetric Fock space with creation and annihilation operators.
They also provided an algebraic characterization of the notion of positive definiteness for a function $\Pt$ on pair partitions and characterized the functions $\Pt$ which give rise to analogs of the Gaussian functor.
Separately \cite{GM2}, they examined a class of Brownian motions arising from the combinatorial notion of species of structure.

Bo{\.z}ejko and {Gu{\c{t}}{\u{a}} \cite{BG} considered a special case of the  generalized Brownian motion of {Gu{\c{t}}{\u{a}} and Maassen arising from $II_1$-factor representations of the infinite symmetric group $S_\infty$ constructed by Vershik and Kerov \cite{VK}.
Lehner \cite{Lehner} considered  these  generalized Brownian motions in the context of exchangeability systems, which generalize various notions of independence and give rise to cumulants analogous to the well-known free and classical cumulants.
Recent work of Avsec and Junge \cite{AJ} offers another point of view on the subject of noncommutative Brownian motion.

In \cite{Guta} {Gu{\c{t}}{\u{a}} extended the notion of generalized Brownian motion to multiple processes indexed by some set $\IndexSet$.
He went on to define for $-1\le q\le 1$  a $q$-product of generalized Brownian motions interpolating between the graded tensor product previously considered by Mingo and Nica \cite{MN} ($q=-1$), the reduced free product \cite{V1985} ($q=0$) and the usual tensor product ($q=1$).
He also showed that this $q$-product obeys a central limit theorem as the size of the index set $\IndexSet$ grows.

In this paper, we explore certain additional questions pertaining to the $\IndexSet$-indexed generalized Brownian motions.
As a warmup, we begin with the very simple case of a generalized Brownian motion arising from a tensor product of representations of the infinite symmetric group $S_\infty$.
We compute the functions on pair partitions associated to the Gaussian states in this context.
We then proceed to consider the generalized Brownian motions associated to spherical representations of the Gelfand pair $(S_\infty\times S_\infty,S_\infty)$.
Here again we give a combinatorial formula for the function on pair partitions arising from the associated Gaussian states, and in the course of doing so we generalize the notion of a cycle decomposition of a pair partition introduced by Bo{\.z}ejko and {Gu{\c{t}}{\u{a}} \cite{BG}. 
We also generalize {Gu{\c{t}}{\u{a}}'s $q$-product of generalized Brownian motions to a $q_{ij}$ product, where $i,j\in\IndexSet$ and show that a central limit theorem holds when $q_{ij}=q_{ji}$ and the $q_{ij}$ are periodic in both $i$ and $j$.

The paper has four sections, excluding this introduction.
In \primarydivname\ \ref{pri:prelim}, we expand upon the notion of generalized Brownian motion with multiple processes established by {Gu{\c{t}}{\u{a}} \cite{Guta}, proving analogs of some results of {Gu{\c{t}}{\u{a}} and Maassen \cite{GM1}.
We also review Vershik and Kerov's factor representations of symmetric groups \cite{VK} and Bo{\.z}ejko and {Gu{\c{t}}{\u{a}}'s work on generalized Brownian motions with one process associated to the infinite symmetric group \cite{BG}.
In \primarydivname\ \ref{pri:tp}, we move on to consider generalized Brownian motions indexed by a two-element set associated to tensor products of factor representations of the infinite symmetric group $S_\infty$.
In \primarydivname\ \ref{pri:spherical}, we consider generalized Brownian motions associated to spherical representations of $(S_\infty\times S_\infty,S_\infty)$.
In \primarydivname\ \ref{pri:qij}, we generalize {Gu{\c{t}}{\u{a}}'s $q$-product to a $q_{ij}$ product, where $i,j\in\IndexSet$.
\subsection*{Acknowledgments}
While working on this paper, the author was supported in part by a National Science Foundation (NSF) Graduate Research Fellowship. He was also supported in part by funds from NSF grant DMS-1001881.
The author also benefited from attending the Focus Program on Noncommutative Distributions in Free Probability Theory at the Fields Institute at the University of Toronto in July of 2013.
The author's travel expenses for this conference were funded by NSF grant DMS-1302713.

The author would like to thank Dan-Virgil Voiculescu for suggesting the problems and Stephen Avsec and Natasha Blitvi{\'c} for a number of thoughtful discussions.

\primarydivision{Preliminaries}
\label{pri:prelim}
\secondarydivision{Generalized Brownian motion}
We begin by establishing the notion of a noncommutative generalized Brownian motion with multiple processes.
Our framework for generalized Brownian motion with multiple processes is slightly more general than that defined in \cite{Guta}.
However everything in this \lprimarydivname\ is in the spirit of results found in \cite{Guta} and \cite{GM1}.

Throughout the \lprimarydivname, we assume that $\IndexSet$ is some fixed index set.
In later \lprimarydivnamepl\, we will specialize to specific index sets.
\begin{notation}
 We will make extensive use of notations for integer intervals, which appear frequently in the combinatorial literature:
\begin{equation}
 \begin{split}
  [m,n]:=\{m,m+1,\ldots,n-1,n\};\\
    [n]:=[1,n]=\{1,2,\ldots,n-1,n\}.
 \end{split}
\end{equation}
for $m,n\in\BZ$.
\end{notation}

\begin{notation}
For a real Hilbert space $\HK$,  $\FOA^{\IndexSet}(\HK)$ will denote the quotient of the free unital $*$-algebra with generators $\omega_{i}(h)$ for $h\in\HK$  and $i\in\IndexSet$ by the relations
\begin{equation}
\omega_i(cf+dg)=c\omega_i(f)+d\omega_i(g),\quad\omega_i(f)=\omega_i(f)^*
\end{equation}
for all $f,g\in\HK$, $i\in \IndexSet$ and $c,d\in \BR$.
\end{notation}

\begin{notation}
If $\HH$ is a complex Hilbert space, $\COA^{\IndexSet}(\HH)$ denotes the free unital $*$-algebra with generators $a_i(h)$ and $a^*_i(h)$ for all $h\in\HH$ and $i\in\IndexSet$ divided by the relations
\begin{equation}
a^*_i(\lambda f+\mu g)=\lambda a^*_i(f)+\mu a_i^*(g),\quad a_i^*(f)=a_i(f)^*,
\end{equation}
for all $f,g\in\HH$, $i\in\IndexSet$, and $\lambda,\mu\in\BC$.
We will also use the notations $a_i^1(h):=a_i(h)$ and $a_i^2(h):=a_i^*(h)$.

Consistent with the notation used in \cite{GM1}, we will assume that the inner product on a complex Hilbert space is linear in the second variable and conjugate linear in the first.
\end{notation}

\begin{definition}\label{def:pp}
 If $P$ is a finite ordered set, let $\PP(P)$ be the set of pair partitions of $P$.
That is,
\begin{equation}
 \PP(P):=\left\{\left\{(l_1,r_1),\ldots,(l_m,r_m)\right\}:  l_k<r_k, \bigcup_{k=1}^n\{l_k,r_k\}=P, \{l_p,r_p\}\cap \{l_{q},r_{q}\}=\emptyset\text{ if }p\ne q\right\}. 
\end{equation}
The set of  $\IndexSet$-indexed pair partitions, $\PP^\IndexSet(P)$ is the set of pairs $(\VP,c)$ with $\VP\in\PP(P)$ and $c:\VP\to\IndexSet$.
We will sometimes refer to the elements of $\IndexSet$ as colors and the function $c$ as the coloring function.
If $P'$ is another finite ordered set and $\alpha:P\to P'$ is an order-preserving bijection, then $\alpha$ induces a bijection $\PP^\IndexSet(P)\to\PP^\IndexSet(P')$.
Considering all order-preserving bijections gives an equivalence relation on the union of $\PP^\IndexSet(P)$ over sets of cardinality $2m$.
Let $\IPP(2m)$ be the set of equivalence classes under this relation, and let $\IPPi:=\bigcup_{n=1}^\infty \IPP(2m)$.
\end{definition}
\begin{figure}
\begin{center}\begin{tikzpicture}
 \draw (0,-1) node[anchor=south]{1} -- ++(0,-0.5) node[anchor=east]{-1} -- ++(3,0) -- ++(0,0.5) node[anchor=south]{4};
 \draw[dotted] (1,-1) node[anchor=south,color=black]{2} -- ++(0,-1) node[anchor=east,color=black]{1} -- ++(3,0) -- ++(0,1) node[anchor=south,color=black]{5};
\draw (2,-1) node[anchor=south]{3} -- ++(0,-1.5) node[anchor=east]{-1} -- ++(3,0) -- ++(0,1.5) node[anchor=south]{6};
\end{tikzpicture}\end{center}
\caption[Example of an indexed pair partition]{The $\{-1,1\}$-indexed pair partition $(\VP,c)$ where $\VP$ is the pair partition $\{(1,4),(2,5), (3,6)\}$ and $c(1,4)=-1$, $c(3,6)=-1$ and $c(2,5)=1$. Solid lines represent the color $-1$ and dotted lines denote the color $1$.}
\label{fig:pp}
\end{figure}
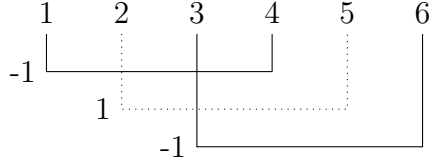
We can create a visual representation of an $\IndexSet$-colored pair partition by connecting the pairs $(l_j,r_j)$ by a path and labeling that path with the color $c(l_j,r_j)$.
When the number of colors is small, we may find it convenient to use different line styles to indicate colors, instead of an explicit label.
Figure \ref{fig:pp} gives the diagram for a simple example with $\IndexSet=\{-1,1\}$.

It is clear that the coloring function $c:\VP\to\IndexSet$ defines a function $c:[2m]\to\IndexSet$.
It should not create confusion to refer to this function by the same name $c$.
Note that $c(l)=c(r)$ when $(l,r)\in\VP$.
We will use these two descriptions of the coloring function $c$ interchangeably.

\begin{definition}
A Fock state on the algebra $\COAI(\HH)$ is a positive unital linear functional $\rho_\Pt:\COAI(\HH)\to\BC$ which for some  $\Pt:\IPPi\to\BC$ satisfies 
\begin{equation}
 \rho_\Pt\left(\prod_{k=1}^m a_{i_k}^{e_k}(f_k)\right)=\sum_{(\VP,c)\in\IPP(m)}\Pt(\VP,c)\prod_{(l,r)\in\VP}\left<f_l,f_r\right>\delta_{i_lc(l,r)}\delta_{i_r,c(l,r)}B_{e_le_r},\label{eqn:fock}
\end{equation}
where for $k\in [m]$, $f_k\in\HH$ the $e_k$ are chosen from $\{1,2\}$ and $i_k\in\IndexSet$.
\begin{equation}
B:= \begin{pmatrix}
     0&1\\
     0&0
    \end{pmatrix}.
\end{equation}
\end{definition}

\begin{definition}
A Gaussian state on $\FOAI(\HK)$  is a  positive unital linear function $\tilde \rho_\Pt:\FOAI(\HH)\to\BC$ which for some $\Pt:\IPPi\to\BC$ satisfies 
\begin{equation}\label{eqn:gauss}
 \tilde \rho_\Pt\left(\prod_{k=1}^m \omega_{i_k}(f_k)\right)=\sum_{(\VP,c)\in\IPP(m)}\Pt(\VP,c)\prod_{(l,r)\in\VP}\left<f_l,f_r\right>\delta_{i_l,c(l,r)}\delta_{i_r,c(l,r)},
\end{equation}
 any $f_k\in \HH$ ($k\in [m]$) and $i_k\in \IndexSet$.
\end{definition}
\begin{remark}
If $\HK$ is a real Hilbert space, then there is a canonical injection $\FOAI(\HK)\to \COAI(\HK_\BC)$  given by
\begin{equation}
\omega_i(h)\mapsto a_i(h)+a_i^*(h),
\end{equation}
where $\HK_\BC$ denotes the complexification of $\HK$.
Considering $\COAI(\HK_\BC)$ as a subalgebra of $\FOAI(\HK)$, the restriction of a Fock state $\rho_\Pt$ on $\COAI(\HK_\BC)$ to $\FOAI(\HK)$ is a Gaussian state.
\end{remark}
While we can use \eqref{eqn:gauss} and \eqref{eqn:fock} to define linear functionals $\tilde\rho_\Pt$ and $\rho_\Pt$ for any choice of $\Pt:\IPPi\to\BC$, these linear functionals are  not always positive.
This leads to the following definition.
\begin{definition}
A function $\Pt:\IPPi\to\BC$ is positive definite if $\rho_\Pt$ is a positive linear functional on $\COAI(\HK)$ for any complex Hilbert space $\HK$.
\end{definition}
\begin{remark}
Our definition of a positive definite function $\Pt:\IPPi\to\BC$ is modeled after the definition made for a single process in \cite{GM1} and is not obviously the same as the definition in \cite{Guta}.
However, we will see in Theorem \ref{thm:pd} that the definitions are equivalent.
\end{remark}

Suppose that for each $\bn:\IndexSet\to\BN\cup\{0\}$, $V_\bn$ is a complex Hilbert space with unitary representation $U_\bn$ of the direct product group
\begin{equation}
 S_\bn:=\prod_{b\in\IndexSet}S_{\bn(b)},
\end{equation}
If $\HH$ is a complex Hilbert space, then the Fock-like space is given by
\begin{equation}\label{eqn:focklike}
 \FS_{V}(\HH):=\bigoplus_{\bn}\frac{1}{\bn!}V_\bn\otimes_s\HH^{\otimes \bn}.
\end{equation}
Here $\HH^{\otimes \bn}$ means
\begin{equation}
\HH^{\otimes \bn}:=\bigotimes_{b\in\IndexSet}\HH^{\otimes\bn(b)},
\end{equation}
and $\bn!$ means $\prod_{b\in\IndexSet}\bn(b)!$, and the factor $\frac{1}{\bn!}$ modifies the inner product.
Also, we set $\HH^{\otimes 0}=\BC\Omega$ for some distinguished unit vector $\Omega$.
The notation $\otimes_s$ denotes the subspace of vectors which are invariant under the action of $S_{\bn}$ given by $U_\bn\otimes \tilde U_\bn$, where $\tilde U_\bn$ is the natural action of $S_{\bn}$ on $\HH^{\otimes\bn}$.
That is, $\tilde U_\bn(\pi)$ permutes the vectors according to $\pi$.
The projection onto the subspace $\frac{1}{\bn!}V_\bn\otimes_s\HH^{\otimes\bn}$ is given by
\begin{equation}
 P_\bn:=\frac{1}{\bn!}\sum_{\sigma\in S_\bn}U(\sigma)\otimes\tilde U(\sigma).
\end{equation}
For $v\in V_\bn$ and $\vf\in \HH^{\bn}$ we denote by $v\otimes_s \vf$ the vector $P_\bn(v\otimes \vf)$

To define creation and annihilation operators on the Fock-like space, we will also require densely defined operators $j_b: V_\bn\to V_{\bn+\delta_b}$ (where $\delta_b(b')=\delta_{b,b'}$) satisfying the following intertwining relations:
\begin{equation}
 j_b U_{\bn}(\sigma)=U_{\bn+\delta_b}(\iota_\bn^{(b)}(\sigma)) j_b\label{eqn:intertwine},
\end{equation}
where $\iota_{\bn}^{(b)}$ is the natural embedding of $S_{\bn}$ into $S_{\bn+\delta_b}$.
Note that we have used the same notation for the maps $V_\bn\to V_{\bn+\delta_b}$ for different $\bn$, but the choice of $\bn$ should be clear from context so confusion should not result.
We will call the maps $j_b$ the \textit{transition maps} for our Hilbert spaces $V_\bn$.

Given transition maps, we can define creation and annihilation operators on the Fock-like space $\FS_V(\HH)$ for each $b\in \IndexSet$ and each $h\in\HH$.
Let $\left(r_b^{(\bn)}\right)^*(h)$ be the operator
\begin{equation}
 \left(r_b^{(\bn)}\right)^*(h):\HH^{\otimes\bn}\to \HH^{\otimes\bn+\delta_{b}}
\end{equation}
which acts as right creation operator on $\HH^{\otimes \bn(b)}$ and identity on $\HH^{\otimes \bn(b')}$ for $b'\ne b$.
Let $r_b^{(\bn)}(h)$ be the adjoint of $\left(r_b^{(\bn)}\right)^*(h)$.
If $\bn(b)\ne 0$, then the annihilation operator $a_b^{V,j}(f)$ is defined on the level $\bn$ component of the Fock-like space by
\begin{equation}
\begin{split}
 a_b^{V,j}(f)&:V_\bn\otimes_s\HH^{\otimes\bn}\to V_{\bn-\delta_b}\otimes_s\HH^{\otimes(\bn-\delta_{b})}\\
a_b^{V,j}(f)&:v\otimes_s \bigotimes_{b\in\IndexSet}h_1^{(b)}\otimes\cdots\otimes h_{\bn(b)}^{(b)}
\\ &\quad \mapsto \sum_{k=1}^{\bn(b)}\frac{\left<f,h_k\right>}{\bn(b)} j_b^*v\otimes_s \bigotimes_{b'\in\IndexSet\setminus\{b\}}\left(h_1^{(b')}\otimes\cdots\otimes h_{\bn(b')}^{(b')}\right)\otimes h_1^{(b)}\otimes\cdots\otimes\hat h_{k}^{(b)}\otimes\cdots\otimes h_{\bn(b)}^{(b)}
\end{split}.
\end{equation}
If $\bn(b)\ne 0$, then $a_b^{V,j}(f)\left(V_\bn\otimes_s\HH^{\otimes\bn}\right)=0$.

The creation operator $(a_b^{V,j})^*(h)$ is the adjoint of $a_b^{V,j}(h)$, and its action on a vector $v\otimes_s \vf$ is given by
\begin{equation}
 (a_b^{V,j})^*(h)\left(v_\bn\otimes_s\vf\right) = (\bn(b)+1)(j_bv_\bn)\otimes_s \left(r_b^{\bn}\right)^*(h)\vf.
\end{equation}
We denote by $\COA_{V,j}(\HH)$ the $*$-algebra generated by the operators $a_b^{V,j}(f)$ and $(a_b^{V,j})^*(f)$ for $f\in\HH$, and $b\in\IndexSet$.

There is a representation $\nu_{V,j}$ of $\COAI(\HH)$ on the Fock-like space $\FS_{V}(\HH)$ satisfying
\begin{equation}
 \nu_{V,j}: a_b(f)\mapsto a_b^{V,j}(f)\quad\text{and}\quad a_b^*(f)\mapsto (a_b^{V,j})^*(f)
\end{equation}
for all $b\in\IndexSet$ and $f\in\HH$.
We will usually identify $X\in \COAI(\HH)$ with its image $ \nu_{V,j}(X)$.

We will sometimes write $a_b(f)$ and $a_b^*(f)$ for the annihilation and creation operators $a_b^{V,j}(f)$ and $\left(a_b^{V,j}\right)^*(f)$.
We will also use the notation $a_b^{V,j,1}(f)$ or simply $a_b^{1}(f)$ for $a_b^{V,j}(f)$ and likewise $a_b^{V,j,2}(f)$ or simply $a_b^{2}(f)$ for $(a_b^{V,j})^*(f)$.

The following is an $\IndexSet$-indexed generalization of Theorem 2.6 of \cite{GM1}.
\begin{theorem}\label{thm:pf}
 Let $\IndexSet$ be an index set and $\HH$ a complex Hilbert space.
 Let $(U_\bn,V_\bn)$ be representations of $S_\bn$ with maps $j_b:V_{\bn}\to V_{\bn+\delta_b}$ satisfying the intertwining relation \eqref{eqn:intertwine}.
 Let $\xi_V\in V_0$ be a unit vector.
  The state $\rho_{V,j}$ on $\COAI(\HH)$ defined by 
\begin{equation}
\rho_{V,j}(X)=\left< \xi_V\otimes_s\Omega,X(\xi_V\otimes_s\Omega) \right>
\end{equation}
is a Fock state.
That is, there is a positive definite function $\Pt$ such that $\rho_{V,j}=\rho_{\Pt}$.
\end{theorem}
The proof is very similar to the proof of Theorem 2.6 of \cite{GM1}, but we include it for completeness.
\begin{proof}
Let $\HH$ be an infinite-dimensional  complex Hilbert space, and let $\{f_k\}_{k=1}^\infty$ be an orthonormal basis for $\HH$. 
Also let $\VP=\{(l_k,r_k):k\in[n]\}$ with $l_k<r_k$ for $1\le k\le n$ and $l_k<l_{k'}$ for $k<k'$.

Define 
\begin{equation}
\begin{split}
 \Pt((\VP,c))&=\rho_{V,j}\left(\prod_{k=1}^n a_{b_k}^{e_k}(f_k)\right)\\
&=\left< \xi_V\otimes_s\Omega,\left(\prod_{k=1}^n a_{b_k}^{e_k}(f_k)\right)(\xi_V\otimes_s\Omega) \right>
\end{split}
\end{equation}
where $b_k$ and $e_k$ are chosen as follows.
Each $k$ is an element of one pair $(l_i,r_i)\in \VP$ for some $i$.
If $k=l_i$, we let $e_k=1$, and if $k=r_i$ we let $e_k=2$.
In either case, we let $b_k=c(l_i,r_i)$.

For $A\in\BO(\HH)$ and $b\in\IndexSet$, define the operator $d\Gamma_V^{b}(A)$ on $\FS_{V}(\HH)$ by
\begin{equation}
\begin{split}
 d\Gamma_V^{b}(A):& v\otimes_s\bigotimes_{b'\in \IndexSet} h_{b',1}\otimes\cdots\otimes h_{b',m_b'}\mapsto\\  &\sum_{k=1}^{m_b}v\otimes_s\left(\bigotimes_{b'\ne b} h_{b',1}\otimes\cdots\otimes h_{b',m_{b'}}\right)\otimes \left(h_{b,1}\otimes\cdots\otimes A h_{b,k} \otimes\cdots\otimes h_{b,m_{b}}\right).
 \end{split}
\end{equation}
The operators $ d\Gamma_V^b(A)$ satisfies the commutation relations
\begin{equation}
 \left[a_{b'}(f),d\Gamma_V^{b}(A)\right]=\delta_{b,b'}a_{b'}(A^*f)
\quad\text{and}\quad
 \left[d\Gamma_V^{b}(A),a_{b'}^*(f)\right]=\delta_{b,b'}a^*_{b'}(Af)
.\label{eqn:gcom}
\end{equation}
We write $a^{e}_{b,k}$ for $a^{V,j,e}_b(f_k)$, and denote by $|f_{i_0}\rangle\langle f_{i}|$ the rank-one operator $X$ on $\HH$ which is $0$ on the orthogonal complement of $f_{i}$ and such that $Xf_{i}=f_{i_0}$.
Applying \eqref{eqn:gcom} with $A=|f_{i_0}\rangle\langle f_{i}|$, we have when $i_k\ne i_0$
\begin{equation}\label{eqn:bcom}
\left[d\Gamma_V^b(|f_{i_0}\rangle\langle f_{i}|), a_{b, i_k}^{e_k}\right] = \delta_{i_k,i}\cdot\delta_{e_k,2}\cdot a_{b,i_0}^*
\end{equation}
Consider a vector of the form $\phi=a_{b_1, i_1}^{e_1}\cdots a_{b_n, i_n}^{e_n}(\xi_V\otimes_s \Omega)$.
Choose $i_0$ different from $i_1,\ldots, i_n$, so that $a_{b,i_0}\phi=0$ for any $b\in\IndexSet$.
Applying \eqref{eqn:bcom},
\begin{equation}
 a_{b,i}\phi = \left[d\Gamma_V^b(|f_{i_0}\rangle\langle f_{i}|), a_{b, i_k}^{e_k}\right]\phi=a_{b,i_0} d\Gamma_V^b(|f_{i_0}\rangle\langle f_{i}|)\phi
\end{equation}
Applying \eqref{eqn:bcom} repeatedly now yields
\begin{equation}
\begin{split}\label{eqn:annihilate}
a_{b,i}\phi &=\sum_{k=1}^n \delta_{i,i_k}\delta_{e_k,2}\delta_{b,b_k}\cdot a_{b, i_0}\left(\prod_{r=1}^{k-1} a_{b_r, i_r}^{e_r}\right)\cdot a_{b,i_0}^* \left(\prod_{r=k+1}^{n} a_{b_r,i_r}^{e_r}\right)\left(\xi_V\otimes_s \Omega\right)
\end{split}
\end{equation}
Considering a monomial in creation and annihilation operators $\prod_{k=1}^n a_{b_k,i_k}^{e_k}$, the theorem follows by applying \eqref{eqn:annihilate} for each annihilation operator in the product with a new index $i_0$.
\end{proof}
There is also the following partial converse.
\begin{theorem}\label{thm:hs}
Let $\HH$ be a separable, infinite-dimensional complex Hilbert space, and let $\IndexSet$ be a countable index set.
Let $\Pt$ be a positive definite function on $\IndexSet$-indexed pair partitions.
Then there exist Hilbert spaces $V_\bn$ with representations $U_{\bn}$ of $S_{\bn}$ and densely defined transition maps $j_b:V_{\bn}\to V_{\bn+\delta_b}$ for $b\in\IndexSet$ satisfying \eqref{eqn:intertwine} and a unit vector $\xi_V\in V_\bzero$ such that the GNS representation of $(\COAI(\HH),\rho_\Pt)$ is unitarily equivalent to $(\FS_V(\HH),\COA_{V,j}(\HH),\xi_V\otimes_s \Omega)$.
\end{theorem}
\begin{remark}\label{rk:diff}
If we are given complex Hilbert spaces $V_\bn$ with representations $U_{\bn}$ of $S_{\bn}$ and maps $j_b:V_{\bn}\to V_{\bn+\delta_b}$, then Theorem \ref{thm:pf} gives a corresponding positive definite function $\Pt:\IPPi\to\BC$.
Applying Theorem \ref{thm:hs} gives complex Hilbert spaces $V_\bn'$ with representations $U_{\bn}'$ of $S_{\bn}$ and transition maps $j_b':V_{\bn}'\to V_{\bn+\delta_b}'$.
We will see in Example \ref{ex:N} that the Hilbert spaces $V_\bn'$ need not be the same as the original Hilbert spaces  $V_\bn$.
\end{remark}
The proof of Theorem \ref{thm:hs} is very similar to the proof of Theorem 2.7 in \cite{GM1} (though we are only able to prove it for infinite-dimensional $\HH$). Regardless, we will include the proof of the $\IndexSet$-indexed theorem here for the sake of completeness.

\begin{proof}[Proof of Theorem \ref{thm:hs}]
Choose an orthonormal basis $\{f_{k,b}\}_{k\in \BN, b\in\IndexSet}$ for $\HH$.
Let $\FS_\Pt(\HH)$, $\COA_\Pt(\HH)$, and $\Omega_\Pt$ be the complex Hilbert space, operator algebra, and distinguished unit vector of the GNS construction of $\COAI(\HH)$ with respect to the state $\rho_\Pt$. 
We denote the image of $a_b(f)$ in $\COA_\Pt(\HH)$ by $a^\Pt_b(f)$ and the image of $a_b^*(f)$ in $\COA_\Pt(\HH)$ by $(a^\Pt_b)^*(f)$.
We will use the notation $a^{\Pt,e}_{b}(f)$ to mean $\left(a^{\Pt}\right)^*_{b}(f)$ for $e=2$ and $a^{\Pt}_{b}(f)$ for $e=1$.
We will construct the complex Hilbert spaces $V_\bn$ as subspaces of $\FS_\Pt(\HH)$.

Suppose that for each function $\bn:\IndexSet\to\BN\cup\{0\}$ which is $0$ at all but finitely many points in $\IndexSet$ and each $i\in \IndexSet$, we have an injective function $\alpha_{\bn,i}: [\bn(i)]\to \BN$ and that $\alpha_{\bn,i}(j)=\alpha_{\bn',i}(j)$ when $j<\bn(i),\bn'(i)$.

Denote by $R_{\bn}^\alpha$ the set of the vectors of the form
\begin{equation}
a^{\Pt,e_1}_{b_1}(f_{b_1, i_1})\cdots a^{\Pt,e_{2p+|\bn|}}_{b_{2p+|\bn|}}(f_{b_{2p+|\bn|}, i_{2p+|\bn|}})\Omega_{\Pt},
\end{equation}
where $|\bn|=\sum_{b\in\IndexSet}\bn(b)$ satisfying the following conditions:
\begin{enumerate} 
\item In the product $a^{\Pt,e_1}_{b_1}(f_{b_1, i_1})\cdots a^{\Pt,e_{2p+|\bn|}}_{b_{2p+|\bn|}}(f_{b_{2p+|\bn|}, i_{2p+|\bn|}})$, a creation operator $a^{\Pt,2}_{b}(f_{b,\alpha_b(j)})$ appears exactly once provided that $1\le j\le \bn(b)$.
\item Among the remaining $2p$ operators in the product, there are $p$ creation operators $(a^{\Pt,2}_{b_q}(f_{b_q, l_q}))_{q=1}^p$ and $p$ annihilation operators $(a^{\Pt,1}_{b_q}(f_{b_q, l_q}))_{q=1}^p$. 
Moreover, each annihilation operator appears to the left of the corresponding creation operator in the product.
\end{enumerate}
We also let $V_{\bn}^\alpha$ be the span of the vectors in $R_{\bn}^\alpha$.
We define the map $j_{b'}^\alpha :V_\bn^\alpha\to V_{\bn+\delta_{b'}}^\alpha$ by restricting the creation operator $a^{\Pt,2}_{b'}(f_{b',\alpha_{\bn,b'}(\bn(b'))})$ to the subspace $V_\bn^\alpha$ of $\FS_\Pt(\HH)$.
It follows immediately from the definition of $V_\bn^\alpha$ that the image of this restriction lies in $V_{\bn+\delta_{b'}}^\alpha$.

We define a unitary representation $U_\bn^\alpha$ of $S_\bn$ on $V_\bn^\alpha$.
Since $\rho_\Pt$ is a Fock state, it is invariant under unitary transformations $U$ on $\HH$ in the sense that
\begin{equation}
\rho_\Pt\left(\prod_{k=1}^n a^{e_k}_{b_k}(f_{b_k,i_k})\right) = \rho_\Pt\left(\prod_{k=1}^n a^{e_k}_{b_k}(Uf_{b_k,i_k})\right).
\end{equation}
Therefore, there is a unitary map $\FS_\Pt(U)$ given by
\begin{equation}
\FS_\Pt(U):\prod_{k=1}^n a^{\Pt, e_k}_{b_k}(f_{b_k,i_k})\Omega_\Pt\mapsto \prod_{k=1}^n a^{e_k}_{b_k}(Uf_{b_k,i_k})\Omega_\Pt.
\end{equation}
The map $\FS_\Pt(U)$ induces an automorphism on the algebra of creation and annihilation operators by
\begin{equation}
\FS_\Pt(U) a^{\Pt, e_k}_{b_k}(h)\FS_\Pt(U^*) = a^{\Pt, e_k}_{b_k}(Uh).
\end{equation}
For $\sigma\in S_\bn$, let $U_\sigma^\alpha$ be the unitary operator on $\HH$ which for each $b\in\IndexSet$ acts by permuting $\{f_{b,\alpha_{\bn,b}(1)},\ldots, f_{b,\alpha_{\bn,b}(\bn(b))}\}$ according to $\sigma$ and fixes $f_{b,r}$ when $r>\bn(b)$.
The map $U_{\bn}^\alpha:\sigma\mapsto U_\sigma^\alpha$ is a unitary representation of $S_\bn$ on $V_\bn^\alpha$.

Define $\iota _{\bn,i}:[\bn(i)]\to\BN$ by $\iota_{\bn,i}(j)=j$, and let $R_\bn:=R_{\bn}^\iota$, $V_\bn:=V_{\bn}^\iota$, $j_{b}:=j_b^\iota$, and let $U_{\bn}:=U_\bn^\alpha$.
It follows from the definitions that these data satisfy the intertwining property \eqref{eqn:intertwine}.
We also define $\xi_V:=\Omega_\Pt \in V_\bzero$.

We now show that $(\FS_V(\HH),\COA_{V,j}(\HH),\xi_V\otimes_s \Omega)$ is unitarily equivalent to the GNS representation of $(\COAI(\HH),\rho_\Pt)$.
We will begin by showing that $\rho_\Pt=\rho_{V,j}$.
By Theorem \ref{thm:pf}, $\rho_{V,j}$ is a Fock state associated to some positive definite function $\Pt':\IPPi\to\BC$, so it will suffice to show that $\Pt'=\Pt$.

For the proof, we will also need to define for a unitary map $U$ on $\HH$,
\begin{equation}
\begin{split}
 \FS_{V}(U)&:\FS_{V}(\HH)\to\FS_{V}(\HH)\\
&v\otimes_s\left(\bigotimes_{b\in\IndexSet}h_{b,1}\otimes \cdots\otimes h_{b,\bn(b)}\right)\mapsto v\otimes_s\left(\bigotimes_{b\in\IndexSet}Uh_{b,1}\otimes \cdots\otimes Uh_{b,\bn(b)}\right)
\end{split},
\end{equation}
for all $v\in V_{\bn}$.
This induces an action on the creation and annihilation operators satisfying
\begin{equation}
 \FS_{V}(U)a_{V,j}^{e}(f) \FS_{V}(U^*)=a_{V,j}^{e}(Uf)
\end{equation}

For $\alpha_{\bn,i}$ as before, define
\begin{equation}
 \tilde V_{\bn}^\alpha:=\overline{\linspan\{v\otimes_s \bigotimes_{b\in\IndexSet}f_{b, \alpha(1)}\otimes\cdots\otimes f_{b,\alpha(\bn(b))}\}}.
\end{equation}
Define an isometry $ T_\bn:V_\bn \to \FS_{V,j}(\HH)$ by
\begin{equation}
v\mapsto v\otimes_s \bigotimes_{b\in\IndexSet}f_{b, 1}\otimes\cdots\otimes f_{b,\bn(b)}.
\end{equation}
Let $U_{\alpha, \bn}$ be a unitary map on $\HH$ which permutes the basis vectors $f_{b,j}$ such that $U_{\alpha}f_{\bn,j}=f_{\bn,\alpha(j)}$ whenever $1\le j\le \bn(b)$.
Define a map $T_\bn^\alpha: V^\alpha_\bn\to \tilde V^\alpha_\bn$ by
\begin{equation}
 T^{\alpha}_\bn:=\FS_\VP(U_{\alpha,\bn})T_\bn\FS_\Pt(U_{\alpha,\bn}^*).
\end{equation}
This map does not depend on the choice of $U_{\alpha, \bn}$ permuting the basis vectors according to $\alpha$.
It follows immediately from the definitions that the diagram 
\begin{equation}
\begin{CD}
V_{\bn}@> T_\bn>> \tilde V_{\bn}\\
@VVj_bV @VV(a^{V,j}_b)^*(f_{b,\bn(b)+1})V\\
V_{\bn+\delta_b}@>T_{\bn+\delta_b}>> \tilde V_{\bn+\delta_b}
\end{CD}
\end{equation}
is commutative, whence the diagram 
\begin{equation}
\begin{CD}
V_{\bn}^\alpha@> T_\bn^\alpha>> \tilde V_{\bn}\alpha \\
@VVj_b^\alpha V @VV(a^{V,j}_b)^*(f_{b,\alpha_{\bn,b}(\bn(b)+1)})V\\
V_{\bn+\delta_b}^\alpha@>T_{\bn+\delta_b}^\alpha>> \tilde V_{\bn+\delta_b}
\end{CD}
\end{equation}
also commutes.
A similar argument gives a corresponding commutative diagram for the annihilation operators, and this implies the equality of the states $\rho_\Pt$ and $\rho_{V,j}$, which implies $\Pt=\tilde \Pt$.

Finally, we must prove that the vacuum vector $\Omega_V:=\xi_{V}\otimes \Omega$ is cyclic for $\COAI(\HH)$.
It will suffice to show that for any $\bn$, any $v\in R_\bn$ and any vectors $h_1,\ldots, h_n\in \HH$, there is some $X\in \COAI(\HH)$ with 
\begin{equation}\label{eqn:cyclic}
X\Omega_{V}=v\otimes_s\bigotimes_{b\in\IndexSet} h_{b,1}\otimes\cdots\otimes h_{b,\bn(b)}.
\end{equation}
By the definition of $R_\bn$, we can write 
\begin{equation}\
v= \left(\prod_{k=1}^{2p+|\bn|}a^{\Pt,e_k}_{b_k}(f_{b_k, i_k})\right)\Omega_{\Pt}
\end{equation}
where a creation operator $a^{\Pt,2}_{b}(f_{b,j})$ appears exactly once for $1\le j\le \bn(b)$, and among the remaining $2p$ operators in the product, there are $p$ creation operators $(a^{\Pt,2}_{b_q}(f_{b_q, l_q}))_{q=1}^p$ and $p$ annihilation operators $a^{\Pt,1}_{b_q}(f_{b_q, l_q})_{q=1}^p$, with each annihilation operator appearing to the left of the corresponding creation operator in the product.
We need simply choose $X$ of the form
\begin{equation}
 X:=\frac{1}{\bn!}\cdot \prod_{k=1}^{2p+|\bn|}a^{e_k}_{b_k}(g_k),
\end{equation}
where the $g_k$ satisfy:
\begin{equation}
g_k:= 
\begin{cases}
  h_{b_k, r_k},&\text{if $1\le i_k\le \bn(b_k)$}\\
  h'_{l_k}, &\text{otherwise}
 \end{cases},
\end{equation}
where $r_k$ is defined so that $k$ is the $r_k$-th smallest element of the set $\{u: b_{u}=b_k, 1\le u\le \bn(b_k)\}$, the $(h'_i)_{i=1}^\infty$ are an orthonormal sequence of vectors which are orthogonal to each $h_{b,k}$, and $l_k=l_k'$ if and only if $b_k=b_{k'}$ and $i_k=i_{k'}$.
It follows from the definitions that this $X$ satisfies \eqref{eqn:cyclic}, so the proof is complete.
\end{proof}

We now pursue an algebraic characterization of positive definiteness for functions on $\IndexSet$-indexed pair partitions.
This will involve Gu{\c{t}}{\u{a}}'s $*$-semigroup of $\IndexSet$-indexed broken pair partitions \cite{Guta}.
\begin{definition}
 Let $X$ be an arbitrary finite ordered set and $(L_a,P_a,R_a)_{a\in\IndexSet}$ a disjoint partition of $X$ into triples of subsets indexed by elements of $\IndexSet$.
Suppose that for each $a\in\IndexSet$, we have a triple $(\VP_a,f_a^{(l)},f_a^{(r)})$ where $\VP_a\in\PP(P_a)$ and 
\begin{equation}
 f_a^{(l)}:L_a\to\{1,\ldots,|L_a|\}\quad\text{and}\quad f_a^{(r)}:R_a\to\{1,\ldots,|R_a|\}
\end{equation}
are bijective.
An order preserving bijection $\alpha:X\to Y$ induces a map
\begin{equation}
 \alpha_a:(\VP_a,f_a^{(l)},f_a^{(r)})\to (\alpha\circ\VP_a,f_a^{(l)}\circ\alpha^{-1},f_a^{(r)}\circ\alpha^{-1})
\end{equation}
where $\alpha\circ \VP:=\{(\alpha(i),\alpha(j)):(i,j)\in\VP\}$.
This determines an equivalence relation on the set of such $\IndexSet$-indexed triples, and we call an equivalence class under this relation an $\IndexSet$-indexed broken pair partition.
We denote the set of all $\IndexSet$-indexed broken pair partition by $\BPP^\IndexSet(\infty)$.
\end{definition}

There is a convenient diagrammatic representation of a broken pair partition.
Given a broken pair partition as just defined, we write the elements of the base set $X$ in order.
For each pair $(x,x')\in \VP_a$, we connect $x$ and $x'$ by a piecewise-linear path and label that path with the index $a$.
For each index $a\in\IndexSet$ such that $L_a\ne \emptyset$, we write the numbers $1,\ldots, |L_a|$ in order on the left side and connect each $y\in L_a$ to the number $f_a^{(l)}(y)$.
Likewise, for each color $a\in\IndexSet$ such that $R_a\ne \emptyset$, we write the numbers $1,\ldots, |R_a|$ in order on the left side and connect each $y\in R_a$ to the number $f_a^{(l)}(y)$.
When $|\IndexSet|$ is small, we may also use different line styles (e.g. dotted and solid lines) to indicate the different colors $a\in\IndexSet$. 
Figure \ref{fig:bpp} gives two examples of these diagrams.

The diagrammatic representations of the $\IndexSet$-colored broken pair partitions inspires some additional terminology.
Namely, we call the functions $f_a^{(l)}$ and $f_a^{(r)}$ the left and right leg functions for the color $a$.
Moreover, we call the piecewise-linear paths from the domains of $f_a^{(l)}$ and $f_a^{(r)}$ to the numbers $f_a^{(l)}(y)$ and $f_a^{(r)}(y)$ the left and right legs of the $\IndexSet$-colored broken pair partitions.
This terminology will be useful in describing the semigroup structure on $\IBPPi$.
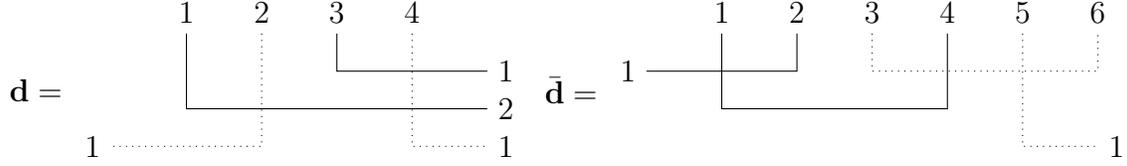
\begin{figure}
\begin{tikzpicture}
 \draw (1,0) node[anchor=south]{1} -- ++(0,-1) -- ++(4,0) node[anchor=west]{2};
 \draw[dotted] (2,0) node[anchor=south,color=black]{2}-- ++(0,-1.5) -- ++(-2,0) node[anchor=east,color=black]{1};
\draw (3,0) node[anchor=south]{3}-- ++(0,-0.5) -- ++(2,0) node[anchor=west]{1};
\draw[dotted] (4,0) node[color=black,anchor=south]{4}-- ++(0,-1.5) -- ++(1,0) node[color=black,anchor=west]{1};
\draw (-1,-.75) node{$\bpd=$};
\end{tikzpicture}
\begin{tikzpicture}
 \draw (1,0) node[anchor=south]{1} -- ++(0,-1) -- ++(3,0) -- ++(0,1) node[anchor=south]{4};
 \draw (2,0) node[anchor=south]{2} -- ++(0,-0.5) -- ++(-2,0) node[anchor=east]{1};
 \draw[dotted] (3,0) node[anchor=south]{3} -- ++(0,-0.5) -- ++(3,0) -- ++(0,0.5) node[anchor=south]{6};
 \draw[dotted] (5,0) node[anchor=south]{5} -- ++(0,-1.5) -- ++(1,0) node[anchor=west]{1};
\draw (-1,-.75) node{$\bar \bpd=$};
\end{tikzpicture}
\caption[Example of an indexed broken pair partition]{The diagram of the $\{-1,1\}$-colored broken pair partitions $\bpd$ and $\bar \bpd$. 
Here $\bpd=\{\VP_a,f_a^{(l)},f_a^{(r)}\}_{a\in\{-1,1\}}$ and $\bar \bpd=\{\bar \VP_a,\bar f_a^{(l)},\bar f_a^{(r)}\}_{a\in\{-1,1\}}$, with the following definitions. 
For $\bpd$, $\VP_{-1}=\VP_{1}$ is the unique pair partition on the empty set and the right and left leg functions are defined by $f^{(l)}_{-1}:\emptyset\to \emptyset$,$f^{(l)}_{1}:\{2\}\to \{1\}$ given by $f^{(l)}_1(2)=1$, $f^{(r)}_{-1}:\{1,3\}\to \{1,2\}$ given by $f^{(r)}_{-1}(1)=2$ and $f^{(r)}_{-1}(3)=1$, and
$f^{(r)}_{1}:\{4\}\to \{1\}$ given by $f^{(r)}_{1}(4)=1$.
 For $\bar \bpd$, $\bar \VP_{2}=\{(1,4)\}$, $\bar \VP_{1}=\{(3,6)\}$ and the right and left leg functions are defined by $\bar f^{(l)}_{-1}:\{2\}\to \{1\}$ with $\bar f^{(l)}_{-1}(2)=1$, $\bar f^{(l)}_{1}:\emptyset\to\emptyset$, $\bar f^{(r)}_1:\emptyset\to\emptyset$, and
$\bar f^{(r)}_{1}:\{5\}\to \{1\}$ given by $\bar f^{(r)}_{1}(5)=1$.
  The solid lines represent the ``color'' $-1$ and the dotted lines represent the ``color'' $1$.}
\label{fig:bpp}
\end{figure}

In the case that $|\IndexSet|=1$, we recover the (uncolored) broken pair partitions of {Gu{\c{t}}{\u{a}} and Maassen \cite{GM1}. 
Moreover, each $\bpd\in\IBPPi$ gives for each $a\in\IndexSet$ a broken pair partition $\bpd_a$ in the sense of \cite{GM1}.
However, all but finitely many of the $\bpd_a$ are the unique broken pair partition on the empty set.

The space $\IBPPi$ can be given the structure of a semigroup with involution, similar to the $*$-semigroup of broken pair partitions of \cite{GM1}.
In terms of the diagrams, multiplication of  two $\IndexSet$-colored broken pair partitions corresponds to concatenation of diagrams.
Right legs of the first diagram are joined with left legs of the second diagram of the same color to form pairs.
In the event that the second diagram has more left legs of some color $a$ than the first diagram has right legs of color $a$, we join the right legs of the first diagram with the largest-numbered left legs of the second diagram, and the remaining left legs of the second diagram are extended to become low-numbered left legs in the product.
An analogous rule is used when the first diagram has more right legs of some color $a$ than the second diagram has left legs of color $a$.

The precise definition of the product on $\IBPPi$ is as follows.
For $i=1,2$, let $\bpd_i=(\VP_{a,i},f_{a,i}^{(l)},f_{a,i}^{(r)})_{a\in\IndexSet}$ be an $\IndexSet$-colored broken pair partition on the ordered base set $X_i$.
The product is a broken pair partition on the base set $X:=X_1\coprod X_2$ with the order relation $x<x'$ if either $x<x'$ in $X_i$ or $x\in X_1$ and $x\in X_2$.
For each $a\in\IndexSet$, define $M_a=\min(|R_{a,1}|,|L_{a,2}|)$.
Following \cite{Guta}, we define
\begin{equation}
 \bpd_1\cdot\bpd_2=(\VP_a,f_a^{(l)},f_a^{(r)})_{a\in\IndexSet},
\end{equation}
where
\begin{equation}
 \VP_a=\VP_{a,1}\cup\VP_{a,2}\cup\left\{\left((f_{a,1}^{(r)})^{-1}([|R_{a,1}|-j]),(f_{a,2}^{(l)})^{-1}([|L_{a,2}|-j])\right): j\in [M_a]\right\}
\end{equation}
and $f_{a}^{(l)}$ is defined on the disjoint union of $L_{a,1}$ and $(f_{a,2}^{(l)})^{-1}([|L_{a,2}|-M_a])$ by
\begin{equation}
 f_a^{(l)}(i)=
\begin{cases}
 f_{a,1}^{(l)}(i),&\text{if }i\in L_{a,1}\\
 f_{a,2}^{(l)}(i)+|L_{a,1}|-M_a,&\text{if } i\in (f_{a,2}^{(l)})^{-1}([|L_{a,2}|-M_a])
\end{cases}.
\end{equation}
The function of right legs, $f_a^{(r)}$ is defined on the disjoint union of $R_{a,2}$ and $(f_{a,1}^{(r)})^{-1}([|R_{a,1}|-M_a])$ by
\begin{equation}
 f_a^{(r)}(i)=
\begin{cases}
 f_{a,2}^{(r)}(i),& \text{if }i\in R_{a,2} \\
 f_{a,1}^{(r)}(i)+|R_{a,2}|-M_a,&\text{if }i\in (f_{a,1}^{(r)})^{-1}([|R_{a,1}|-M_a])
\end{cases}.
\end{equation}
An example of multiplication of $\IndexSet$-colored broken pair partitions is illustrated in Figure \ref{fig:bppm}.\footnote{The multiplication for $\IBPPi$ stated here differs slightly from that stated in \cite{Guta}. We believe that the rule stated here is the one intended by the author of that work, as it ensures that condition \eqref{eqn:intertwine} is satisfied. However, we do not believe that this discrepancy is consequential for Gu{\c{t}}{\u{a}}'s results.}
\begin{figure}

\begin{tikzpicture}[scale=0.6]
 \draw (1,0) node[anchor=south]{1} -- ++(0,-1) -- ++(4,0) node[anchor=west]{2};
 \draw[dotted] (2,0) node[anchor=south,color=black]{2}-- ++(0,-1.5) -- ++(-2,0) node[anchor=east,color=black]{1};
\draw (3,0) node[anchor=south]{3}-- ++(0,-0.5) -- ++(2,0) node[anchor=west]{1};
\draw[dotted] (4,0) node[color=black,anchor=south]{4}-- ++(0,-1.5) -- ++(1,0) node[color=black,anchor=west]{1};
\draw (5.8,-.5) node{$\cdot$};
 \draw (7.5,0) node[anchor=south]{1} -- ++(0,-1) -- ++(3,0) -- ++(0,1) node[anchor=south]{4};
 \draw (8.5,0) node[anchor=south]{2} -- ++(0,-0.5) -- ++(-2,0) node[anchor=east]{1};
 \draw[dotted] (9.5,0) node[anchor=south]{3} -- ++(0,-0.5) -- ++(3,0) -- ++(0,0.5) node[anchor=south]{6};
 \draw[dotted] (11.5,0) node[anchor=south]{5} -- ++(0,-1.5) -- ++(1,0) node[anchor=west]{1};
\draw (13.25,-.5) node{$=$};
 \draw (15,0) node[anchor=south]{1} -- ++(0,-0.5) -- ++(5,0) -- ++(0,0.5) node[anchor=south]{6};
\draw[dotted] (16,0) node[anchor=south]{2}-- ++(0,-1.5) -- ++(-2,0) node[anchor=east]{1};
\draw (17,0) node[anchor=south]{3}-- ++(0,-1) -- ++(8,0) node[anchor=west]{1};
\draw[dotted] (18,0) node[anchor=south]{4}-- ++(0,-2.5) -- ++(7,0) node[anchor=west]{2};
\draw (19,0) node[anchor=south]{5} -- ++(0,-0.25) -- ++(3,0) -- ++(0,0.25) node[anchor=south]{8};
\draw[dotted] (21,0) node[anchor=south]{7} -- ++(0,-0.5) -- ++(3,0) -- ++(0,0.5) node[anchor=south]{10};
\draw[dotted] (23,0) node[anchor=south]{9}-- ++(0,-2) -- ++(2,0) node[anchor=west]{1};
\end{tikzpicture}
 \caption[Multiplication of indexed broken pair partitions]{The multiplication $\bpd\cdot\bar \bpd$ of the $\{-1,1\}$-colored broken pair partitions defined in Figure \ref{fig:bpp}.}
\label{fig:bppm}
\end{figure}
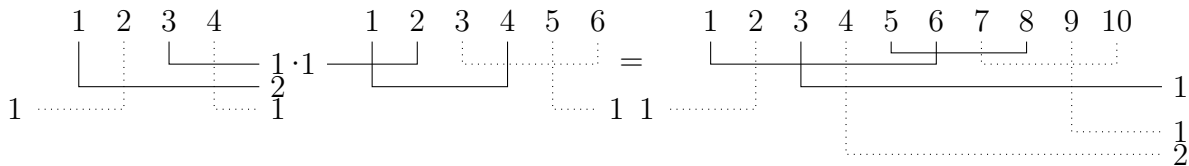

The involution is given by mirror reflection of the $\IndexSet$-colored broken pair partitions. 
Formally, if $\bpd=(\VP_a,f_a^{(l)},f_a^{(r)})_{a\in\IndexSet}$ with underlying set $X$ then $\bpd^*=(\VP_a^*,f_a^{(r)},f_{a}^{(l)})_{a\in\IndexSet}$ is an $\IndexSet$-colored broken pair partition with underlying set $X^*$, the same as $X$ but with the order reversed, where $\VP_a^*=\{(i,j):(j,i)\in\VP_a\}$.
The involution is illustrated in Figure \ref{fig:inv}.

\begin{figure}
 \begin{tikzpicture}
 \draw (-1,0) node[anchor=south]{6} -- ++(0,-1) -- ++(-3,0) -- ++(0,1) node[anchor=south]{3};
 \draw (-2,0) node[anchor=south]{5} -- ++(0,-0.5) -- ++(2,0) node[anchor=west]{1};
 \draw[dotted] (-3,0) node[anchor=south]{4} -- ++(0,-0.5) -- ++(-3,0) -- ++(0,0.5) node[anchor=south]{1};
 \draw[dotted] (-5,0) node[anchor=south]{2} -- ++(0,-1.5) -- ++(-1,0) node[anchor=east]{1};
\draw (-7,-.75) node{$\bar \bpd^*=$};
\end{tikzpicture}
\caption[Involution of an indexed broken pair partition]{The involution of the $\{-1,1\}$-colored broken pair partition $\bar \bpd$ depicted in Figure \ref{fig:pp}.}
\label{fig:inv}
\end{figure}
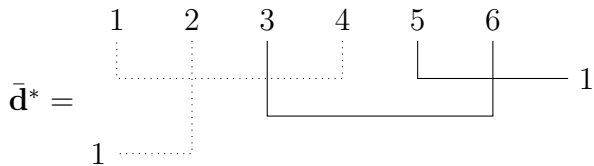

For each function $\bn:\IndexSet\to\BN$ which is zero except on finitely many elements of $\IndexSet$, let $\IBPPn$ be the subset of $\IBPPi$ consisting of elements having exactly $\bn(a)$ left legs of color $a$ and no right legs.

Let $d_{a}$ be the unique element of $\IBPPi$ with no right legs, no pairs, and only one left leg, colored $a\in \IndexSet$.
We call $d_a$ the $a$-colored left hook and $d_a^*$ the $a$-colored right hook.

An $\IndexSet$-colored broken-pair partition can be written as a sequence of left hooks, followed by permutations acting on the legs of the same color, followed by right hooks connecting with left legs of the same index, possibly followed by additional sequences of left hooks, permutations, and right hooks.
The ``standard form'' of an element of $\IPPi$ is the sequence of this form such that if two like-colored pairs cross, they do so in the rightmost permutation possible.
Figure \ref{fig:stdform} depicts the standard form of one example.

\begin{figure}
 \begin{tikzpicture}
  \draw (0,0) node[anchor=south]{1} -- ++(0,-1) -- ++(3.5,0) -- ++(0.5,0.5) -- ++(0.5,0) --++(0,.5) node[anchor=south]{5};
  \draw (1,0) node[anchor=south]{2} -- ++(0,-0.5) -- ++(2.5,0) -- ++(0.5,-0.5) -- ++(3.0,0) -- ++(0.5,0.5) --++(0.5,0)--++(0,0.5) node[anchor=south]{8};
  \draw[style=dotted] (2,0) node[anchor=south]{3} -- ++(0,-2.0) -- ++(1.5,0) -- ++(0.5,0.5) -- ++(1.5,0) --++(0,1.5) node[anchor=south]{6};
  \draw[style=dotted] (3,0) node[anchor=south]{4} -- ++(0,-1.5) -- ++(0.5,0) -- ++(0.5,-0.5) -- ++(6,0) --++(0,2) node[anchor=south]{10};
  \draw (6.5,0) node[anchor=south]{7} -- ++(0,-0.5)-- ++(0.5,0)-- ++(0.5,-0.5)--++(1.5,0)--++(0,1) node[anchor=south]{9};
  \draw[style=dashed] (3.5,0.25)-- ++(0,-2.5);
  \draw[style=dashed] (4.0,0.25)-- ++(0,-2.5);
  \draw[style=dashed] (7,0.25)-- ++(0,-2.5);
  \draw[style=dashed] (7.5,0.25)-- ++(0,-2.5);
 \end{tikzpicture}
\caption[Standard form of an indexed pair partition]{The standard form of the pair partition $(\VP,c)\in\IPP(10)$ with $\IndexSet=\{-1,1\}$ and $\VP=\{(1,5),(2,8),(3,6),(4,10),(7,9)\}$ and $c((1,5))=c((2,8))=c((7,9))=-1$ and $c((3,6))=c((4,10))=1$. The solid lines represent the ``color'' $-1$ and the dotted lines represent the ``color'' $1$.}
\label{fig:stdform}
\end{figure}
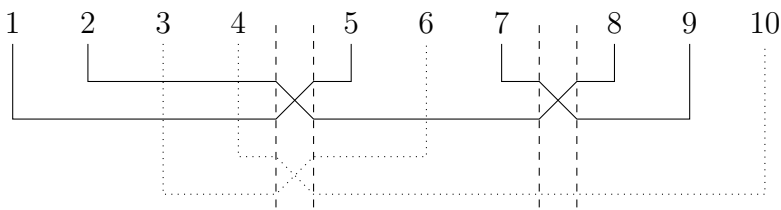

As in \cite{Guta}, we can use the standard form of $(\VP,c)$ for $\VP\in \PP(2m)$ to compute the value of $\Pt((\VP,c))$ as follows.
Consider $c$ as a function $[2m]\to \IndexSet$ taking the same value on points belonging to the same pair of $\VP$.
Partition $[2m]$ into $2t$ blocks $B_i^{(r)}$, $B_i^{(l)}$ for $i=1,\ldots, t$ such that the $B_i^{(l)}$ contain left legs of the pairs of $\VP$ and the $B_i^{(r)}$ contain right legs of the pairs of $\VP$.
Write $B_i^{(r)}:=\{k_{i-1},\ldots, p_{i}\}$ and $B_j^{(l)}:=\{p_{i+1},\ldots, k_{i}\}$ with $k_0=1$ and $k_r=2m$.
Then there are permutations $\pi_j$ such that
\begin{equation}\label{eqn:stdform}
(\VP,c)= \prod_{l=1}^{p_1}d_{c(l)}^* U_{\bn_1}(\pi_1)\prod_{l=p_1+1}^{k_1}d_{c(l)}\cdots U_{\bn_r}(\pi_r)\prod_{l=p_t+1}^{2m} d_{c(l)}
\end{equation}
The function on pair partitions can then be calculated as
\begin{equation}\label{eqn:pairstd}
 \Pt_{V,j}((\VP,c))=\left<\xi_{V}, \prod_{l=1}^{p_1}j_{c(l)}^* U_{\bn_1}(\pi_1)\prod_{l=p_1+1}^{k_1}j_{c(l)}\cdots U_{\bn_r}(\pi_r)\prod_{l=p_t+1}^{2m} j_{c(l)}\xi_{V}\right>. 
\end{equation}

An $\IndexSet$-colored pair partition $\VP$ can be considered as an element of $\IBPPi$ having no left or right legs in the obvious way. A function $\Pt:\IPPi\to\BC$ thus extends to a function $\hat\Pt:\IBPPi\to\BC$ by
\begin{equation}
 \hat\Pt(\bpd)=
\begin{cases}
 \Pt(\bpd),&\text{if }\bpd\in\IPPi,\\
0,&\text{otherwise}
\end{cases}.
\end{equation}
The following is an $\IndexSet$-indexed generalization of Theorem 3.2 of \cite{GM1}.
\begin{theorem}\label{thm:pd}
 A function $\Pt:\IPPi\to\BC$ is positive definite if $\hat\Pt:\IBPPi\to\BC$ is positive definite in the usual sense of positive definiteness for a function on a semigroup with involution.
\end{theorem}
The proof is very similar to the proof of Theorem 3.2 of \cite{GM1}, but we include it for completeness.
\begin{proof}
Suppose that $\hat\Pt$ is positive definite on the $*$-semigroup $\IBPPi$.
Then there is a representation $\chi$ of $\IBPPi$ on a complex Hilbert space $V$ having cyclic vector $\xi\in V$ such that
\begin{equation}
 \left<\xi,\chi(d)\xi\right>=\hat\Pt(d)
\end{equation}
for all $d\in\IBPPi$.

The complex Hilbert space $V$ is expressible as a direct sum
\begin{equation}
 V=\bigoplus_{\bn} V_\bn,
\end{equation}
where the sum is over functions $\bn:\IndexSet\to \BN$ with only finitely many nonzero values and 
\begin{equation}
V_\bn=\overline{\linspan\{\chi_t(d)\xi:d\in\IBPPn\}}.
\end{equation}
The action of $S_\bn$ on $\IBPPn$ (by permutation of the left legs) gives a unitary representation $U_\bn$ of $S_\bn$ on $V_\bn$.
Restriction of $j_b:=\chi(d_b)$ (where, as before $d_b$ is the $b$-colored left hook) gives a map $j_b:V_\bn\to V_{\bn+\delta_b}$ satisfying \eqref{eqn:intertwine}.
Choose a unit vector $\xi_V\in V_0$.

Let $\HH$ be an infinite-dimensional complex Hilbert space. Using the $U_\bn$, $V_\bn$ and $j_b$, we can construct the Fock space $\FS_V(\HH)$ and the algebra $\COA_{V,j}(\HH)$ with vacuum vector $\Omega_V$.
By Theorem \ref{thm:pf}, the vacuum state is a Fock state arising from some positive definite function $\Pt':\IBPPi\to\BC$.
It will suffice to show that $ \Pt'=\Pt$, whence it will follow that $\Pt$ is positive definite.
In fact, this follows from Theorem 2.3 of \cite{Guta}, but we also provide a proof for completeness.

Given an $\IndexSet$-indexed pair partition $(\VP,c)\in\IPP(2m)$ with $\VP=\{(l_1,r_1),\ldots,(l_m,r_m)\}$, let the standard form of $(\VP,c)$ be
\begin{equation}
 (\VP,c)=\prod_{l=1}^{p_1}d_{c(l)}^* U_{\bn_1}(\pi_1)\prod_{l=p_1+1}^{k_1}d_{c(l)}\cdots U_{\bn_r}(\pi_r)\prod_{l=p_t+1}^{2m} d_{c(l)}.
\end{equation}
Let $\HH=\ell^2(\BZ)$ have an orthonormal basis $(f_k)_{k=1}^\infty$, and choose a monomial 
\begin{equation}
M:=\prod_{k=1}^{2m}a_{b_k}^{e_k}(f_{i_k}),
\end{equation}
where $i_k$ is chosen such that either $k=l_{i_k}$ or $k=r_{i_k}$, and $b_k=c(l_{i_k},r_{i_k})$, and $e_k=1$ if $k=l_{i_k}$ and $e_k=2$ if $k=r_{i_k}$.
From the definition of a Fock state,
\begin{equation}
\Pt'(\VP, c)=\left<\Omega_V,M\Omega_V \right>.
\end{equation}
Using the definition of the creation operator and \eqref{eqn:intertwine}, we get
\begin{equation}
\begin{split}
\Pt'(\VP, c)&= \prod_{l=1}^{p_1}d_{c(l)}^* U_{\bn_1}(\pi_1)\prod_{l=p_1+1}^{k_1}d_{c(l)}\cdots U_{\bn_r}(\pi_r)\prod_{l=p_t+1}^{2m} d_{c(l)}\\
&=\hat \Pt(\VP, c)\\
&=\Pt(\VP, c).
\end{split}
\end{equation}
This completes the proof that $\Pt:\IPPi\to\BC$ is positive definite.

Suppose now that $\Pt:\IPPi\to\BC$ is positive definite.
Applying Theorem \ref{thm:hs} we get complex Hilbert spaces $V_\bn$ with representations $U_\bn$ of $S_\bn$ and densely defined maps $j_b:V_\bn\to V_{\bn+\delta_b}$ satisfying the intertwining relation \eqref{eqn:intertwine}.
Let
\begin{equation}
V=\bigoplus_{\bn}V_{\bn}.
\end{equation}
Since the left hooks $\{d_b:b\in\IndexSet\}$ and the actions $\lambda_\bn$ of the symmetric groups $S_\bn$ on $\IBPPn$ generate $\IBPPi$, we have a representation $\chi$ of $\IBPPi$, and it is easily verified that $\left<\xi,\chi(d)\xi\right>=\hat \Pt(d)$ for any unit vector $\xi\in V_0$ and any $d\in \IBPPi$, whence $\hat\Pt$ is positive definite.
\end{proof}

Gu{\c{t}}{\u{a}} \cite{Guta} considered the case in which the $V_\bn$ and the $j_a$ are defined as follows.
Let $\Pt:\IPPi\to\BC$ be a positive definite function.
As before, denote by $\IBPP(\bn,\bzero)$ the set of $\bpd\in\IBPPi$ having $|R_a|=0$ and $|L_a|=\bn(a)$ for each $a\in\IndexSet$.
Consider the GNS representation $(\chi_\Pt, V,\xi_\Pt)$ of $\IBPPn$ with respect to $\hat \Pt$, characterized by
\begin{equation}
 \left<\chi_\Pt(\bpd_1)\xi_\Pt,\chi_\Pt(\bpd_2)\xi_\Pt\right>_V=\hat \Pt(\bpd_1^*\bpd_2).
\end{equation}
The complex Hilbert space $V$ is given by
\begin{equation}
 V:=\bigoplus_{\bn} V_{\bn}\quad\text{where}\quad V_\bn=\overline{\linspan\{\chi_\Pt(\bpd)\xi_\Pt:\bpd\in\IBPPn\}}.
\end{equation}
Each $V_\bn$ has a representation of $S_\bn$ with $S_{\bn(a)}$ acting by permuting the left legs of $\IBPPn$ of color $a$.
That is for $\pi=(\pi_a)_{a\in \IndexSet}\in S_\bn$ and $(\VP_a,f_a^{(l)},f_a^{(r)})_{a\in\IndexSet}\in \IBPPn$,
\begin{equation}
 U_\bn(\pi)(\VP_a,f_a^{(l)},f_a^{(r)})_{a\in\IndexSet}=(\VP_a,\pi_a^{-1}\circ f_a^{(l)},f_a^{(r)})_{a\in\IndexSet}.
\end{equation}
We denote by $\FS_\Pt(\HH)$ the Fock-like space arising from using these $V_\bn$ in \eqref{eqn:focklike}
\begin{equation}
  \FS_{\Pt}(\HH):=\bigoplus_{\bn}\frac{1}{\bn!}V_\bn\otimes_s\bigotimes\HH^{\otimes\bn}.
\end{equation}
The creation and annihilation operators on $\FS_\Pt(\HH)$ will be assumed to be those associated to the following operators $j_a$.
For $a\in\IndexSet$, denote by $j_a$ the operator $\chi_\Pt(\bpd_{a})$, where $\bpd_{a}$ is the broken pair partition with no pairs, no right legs, and one $a$-colored left leg.

We can now state the following theorem of Gu{\c{t}}{\u{a}} \cite{Guta}.
\begin{theorem}
 Let $f_1,\ldots, f_n$ be vectors in a complex Hilbert space $\HH$. Then the expectation values with respect to the vacuum state $\rho_\Pt$ of the monomials in creation and annihilation operators on the Fock space $\FS_\Pt(\HH)$ have the expression
\begin{equation}
 \rho_\Pt\left(\prod_{i=1}^m a_{b_i}^{e_i}(f_i)\right)=\sum_{(\VP,c)\in\IPP(n)}\Pt(\VP,c)\prod_{(i,j)\in\VP}\left<f_i,f_j\right>\delta_{b_i,b_j}B_{e_ie_j},
\end{equation}
where the $e_i$ are chosen from $\{1,2\}$ and
\begin{equation}
B:= \begin{pmatrix}
     0&1\\
     0&0
    \end{pmatrix}.
\end{equation}
\end{theorem}

\secondarydivision{Factor representations of $S_\infty$}
In \thesis{\primarydivname\ \ref{pri:spherical}}\paper{\primarydivnamepl\ \ref{pri:tp} and \ref{pri:spherical}}, we will take an interest in noncommutative generalized Brownian motions which are related to factor representations of the group $S_\infty$ of permutations of $\BN$ which fix all but finitely many points.
Here we  briefly recall some relevant background information pertaining to those representations.

The finite factor representations of a group are determined by the group's characters, that is, the positive, normalized indecomposable functions which are constant on conjugacy classes.
In the case of $S_\infty$, the characters are given by the following famous result.
\begin{theorem}[Thoma's Theorem \cite{Thoma}]\label{thm:thoma}
 The normalized finite characters of $S_\infty$ are given by the formula
\begin{equation}
 \phi_{\alpha,\beta}(\sigma)=\prod_{m\ge2}\left(\sum_{i=1}^\infty\alpha_i^m+(-1)^{m+1}\sum_{i=1}^\infty\beta_i^m\right)^{\rho_m(\sigma)}
\end{equation}
where $\rho_m(\sigma)$ is the number of cycles of length $m$ in the permutation $\sigma$, and $(\alpha_i)_{i=1}^\infty$ and $(\beta_i)_{i=1}^\infty$ are decreasing sequences of positive real numbers such that $\sum_{i}\alpha_i+\sum_i\beta_i\le 1$.
\end{theorem}
The pairs of sequences $(\alpha_i)_{i=1}^\infty$ and $(\beta_i)_{i=1}^\infty$ satisfying the conditions in Theorem \ref{thm:thoma} are commonly called Thoma parameters.

We now recall Vershik and Kerov's representation of the symmetric group $S_n$ (for $n\in\{0,1,2,\ldots,\infty\}$) \cite{VK}.
\begin{notation}\label{not:vk}
Fix sequences $(\alpha_i)_{i=1}^\infty$ and $(\beta_i)_{i=1}^\infty$, and let $\gamma = 1-\sum_{i}\alpha_i-\sum_i\beta_i$ and let $\SN_+$ and $\SN_-$ be two copies of the set $\BN=\{1,2,\ldots\}$.
Let $Q:=\SN_+\cup\SN_-\cup[0,\gamma]$, and define a measure $\mu$ on $Q$ to be the Lebesgue measure on $[0,\gamma]$ and such that $\mu(i)=\alpha_i$ for $i\in\SN_+$ and $\mu(j)=\beta_j$ for $j\in\SN_-$.
Let $\SX_n$ denote the $n$-fold Cartesian product of $Q$ with the product measure $m_n=\prod_1^n\mu$, and let $S_n$ act on $\SX_n$ by $\sigma(x_1,\ldots, x_n)=(x_{\sigma^{-1}(1)},\ldots,x_{\sigma^{-1}(n)})$.
For $x,y\in \SX_n$, say that $x\sim y$ if there exists $\sigma\in S_n$ such that $x=\sigma y$.
Let $\tilde \SX_n=\{(x,y)\in \SX_n\times\SX_n: x\sim y\}$.
The complex Hilbert space $V_n^{(\alpha,\beta)}$  defined by
\begin{equation}\label{eqn:vkhs}
 V_n^{(\alpha,\beta)}:=\left\{f:\tilde \SX_n\to\BC\vert\infty>\|f\|^2=\int_{\SX_n}\sum_{y\sim x}|f(x,y)|^2dm_n^{(\alpha,\beta)}(x)\right\}
\end{equation}
carries a unitary representation $U_{n}^{(\alpha,\beta)}$ of $S(n)$ given by
\begin{equation}
 (U_n^{(\alpha,\beta)}(\sigma)h)(x,y)=(-1)^{i(\sigma,x)}h(\sigma^{-1}x,y),
\end{equation}
where $i(\sigma,x)$ is the number of inversions in the sequence $(\sigma i_1(x),\sigma i_2(x),\ldots)$ of indices $i_r(x)$ for which $\sigma x_i\in\SN_-$.
Denote by $\bone_n$ the indicator function of the diagonal $\{(x,x)\}\subset \tilde \SX_n$.
\end{notation}
Vershik and Kerov showed the following.
\begin{theorem}[\cite{VK}]
 On $V_n^{(\alpha,\beta)}$,
\begin{equation}
 \left\langle U_n^{(\alpha,\beta)}(\sigma)\bone_n,\bone_n\right\rangle = \phi_{\alpha,\beta}(\sigma).
\end{equation}
For $n=\infty$ we get the representation of $S_\infty$ associated to $\phi_{\alpha,\beta}$ in the convex hull of $\bone_\infty$.
\end{theorem}
There is an isometry $j_n:V_n^{(\alpha,\beta)}\to V_{n+1}^{(\alpha,\beta)}$ defined by
\begin{equation}
 (j_nh)(x,y)=\delta_{x_{n+1},y_{n+1}}h((x_1,\ldots, x_n), (y_1,\ldots, y_n))
\end{equation}
\secondarydivision{Generalized Brownian motions arising from factor representations of $S_\infty$}
Before considering multi-dimensional noncommutative generalized Brownian motions associated to representations associated to infinite symmetric groups, we review some of the work of \cite{BG} on Brownian motions connected with representations of $S_\infty$ with one process.
The Vershik-Kerov factor representations of the symmetric groups $S_n$ give all the data needed for a $1$-dimensional generalized Brownian motion.
Bo{\.z}ejko and Gu{\c{t}}{\u{a}} \cite{BG} were able to characterize the function on pair partitions arising from Theorem 2.6 of \cite{GM1} (the one-dimensional version of Theorem \ref{thm:pf}).
Their result depends on the following terminology.
\begin{definition}[\cite{BG}]\label{def:cycles}
 Let $\VP\in\PP(2m)$, and denote by $\hat\VP$ the unique noncrossing pair partition such that the set of left points of $\VP$ and $\hat \VP$ coincide.
A cycle in $\VP$ is a sequence of pairs $((l_1,r_1),\ldots,(l_m,r_m))$ of $\VP$ such that the pairs $(l_1,r_2), (l_2,r_3),\ldots,(l_m,r_1)$ belong to $\hat \VP$.
(In the case that $m=1$ we interpret this condition as $(l_1,r_1)\in\hat \VP$.)
The number $m$ is called the length of the cycle.
Denote by $\rho_m(\VP)$ the number of cycles of length $m$ in the pair partition $\VP$.
\end{definition}
Bo{\.z}ejko and Gu{\c{t}}{\u{a}}'s formula is as follows.
\begin{theorem}[\cite{BG}]\label{thm:bg}
 Let $(\alpha_i)_{i=1}^\infty$ and $(\beta_i)_{i=1}^\infty$ be decreasing sequences of positive real numbers such that $\sum_{i}\alpha_i+\sum_i\beta_i\le 1$.
Let $V_{n}^{(\alpha,\beta)}$ be the complex Hilbert space of the Vershik-Kerov representation of $S_n$, and let $j_n:V_n^{(\alpha,\beta)}\to V_{n+1}^{(\alpha,\beta)}$ be the natural isometry.
Let $\xi_{V^{(\alpha,\beta)}}=\bone_0$.
Denote by $\Pt_{\alpha,\beta}$ the function on $\PP(\infty)$ associated to these representations by Theorem \ref{thm:pf}.
Then 
\begin{equation}
 \Pt_{\alpha,\beta}(\VP)=\prod_{m\ge 2}\left(\sum_{i=1}^\infty\alpha_i^m+(-1)^{m+1}\sum_{i=1}^\infty\beta_i^m\right)^{\rho_m(\VP)}.
\end{equation}
\end{theorem}
\begin{remark}
There is another equivalent characterization of the cycle decomposition of a pair partition.
As in Definition \ref{def:cycles}, let $\VP=\{(a_1,z_1),\ldots,(a_n,z_n)\}\in\PP(2m)$ and let $\hat\VP$ be the noncrossing pair partition whose left points coincide with those of $\VP$.
Let $\sigma\in S_n$ be the permutation such that $\hat\VP=\{(a_1,z_{\sigma_{-1}(1)}),\ldots,(a_n,z_{\sigma^{-1}(n)})\}$.
If the cycles of $\sigma$ are $\tau_i=(b_{i1}\cdots b_{ir_i})\in S_n$ ($1\le i\le m$), then the cycles of $\VP$ are $\{(a_{b_{i1}},z_{b_{i1}}),\ldots,(a_{b_{ir_i}},z_{b_{ir_i}})\}$.
Moreover, Theorem \ref{thm:bg} says that $\Pt_{\alpha,\beta}(\VP)=\phi_{\alpha,\beta}(\sigma)$.
\end{remark}

We are now in a position to show that the framework for multi-dimensional generalized Brownian motion presented here is more general than that presented in \cite{Guta}.
More precisely, we will exhibit complex Hilbert spaces $V_\bn$ with representations $U_{\bn}$ of $S_{\bn}$ and transition maps $j_b:V_{\bn}\to V_{\bn+\delta_b}$ and $V_\bn'$ with representations $U_{\bn}'$ of $S_{\bn}$ and maps $j_b:V_{\bn}'\to V_{\bn+\delta_b}'$ such that both sets of data give rise to the same function on pair partitions according to Theorem \ref{thm:pf}.
\begin{example}\label{ex:N}
We work with the index set $\IndexSet=\{1\}$, which places us in the setting of the generalized Brownian motion with only one process, developed by Gu{\c{t}}{\u{a}} and Maassen in \cite{GM1}.
Fix an integer $N$ with $|N|>1$ and let $\HH$ be a complex Hilbert space.
We will consider generalized Brownian motions associated to the character $\phi_N$ of $S_\infty$ given by the sequences 
\begin{equation}
\alpha_n=
\begin{cases}
\frac{1}{N},&\text{if $1\le n\le N$}\\
0,&\text{otherwise}
\end{cases}
\quad\text{and}\quad
\beta_n=
\begin{cases}
\frac{1}{N},&\text{if $1\le n\le -N$}\\
0,&\text{otherwise}
\end{cases}
\end{equation}
For each $n$, let $V_{n}^{(N)}$ be the complex Hilbert space of the Vershik-Kerov representation of $S_n$, and let $j^{(N)}:V_n^{(N)}\to V_{n+1}^{(N)}$ be the natural isometry.
Let $\xi_{V^{(N)}}=\bone_0$ be the indicator function of the diagonal.
Denote by $\Pt_N$ the function on $\PP(\infty)$ associated to these representations by Theorem \ref{thm:pf}.
It was shown in  \cite{BG} that
\begin{equation}
\Pt_N(\VP)=\left(\frac{1}{N}\right)^{n-\rho(\VP)}.
\end{equation}

We will exhibit another sequence of complex Hilbert spaces $\hat V_n^{(N)}$ with unitary representations $\hat U_n^{(N)}$ which gives rise to the same positive function on pair partitions.
Since the character $\phi_N: S_\infty\to\BC$ restricts to a positive definite function on $S_n$, there is a representation $\hat U_n^{(N)}$ of $S_n$ on a complex Hilbert space $\hat V_n^{(N)}$ with a cyclic vector $\xi_n$ such that 
\begin{equation}
\left<\xi_n,\hat U_n^{(N)}(\pi)\xi_n\right>=\phi_N(\pi).
\end{equation}
for every $\pi\in S_n$.
There is also a natural inclusion $\hat j^{(N)}: \hat V_n^{(N)}\to\hat V_{n+1}^{(N)}$ satisfying 
\begin{equation}
\hat j(\hat U_n^{(N)}(\pi)\xi_n)=\hat U_{n+1}^{(N)}(\iota_n\pi)\xi_{n+1},
\end{equation}
 where $\iota_n$ is the inclusion $S_n\to S_{n+1}$ induced by the natural inclusion $[n]\subset[n+1]$.
By construction, the maps $\hat j^{(N)}$ and representations $\hat V_n^{(N)}$ satisfy the intertwining relation \eqref{eqn:intertwine}, so we can construct the Fock space $\FS_{\hat V^{(N)},\hat j^{(N)}}$ with creation and annihilation operators $a_{\hat V^{(N)}, \hat j^{(N)}}^*(f)$ and $a_{\hat V^{(N)}, \hat j^{(N)}}(f)$.

The action of the algebra $\COA_{\hat V^{(N)}, \hat j^{(N)}}(\HH)$ on $\FS_{\hat V^{(N)}, \hat j^{(N)}}(\HH)$ is unitarily equivalent to the action of the algebra creation and annihilation operators on the following deformed Fock space.
Let $\FS^{(alg)}(\HH)=\bigoplus_n \HH^{\otimes n}$.
Define a sesquilinear form on $\FS^{(alg)}(\HH)$ by sesquilinear extension of
\begin{equation}
\left<f_1\otimes\cdots f_n, g_1\otimes\cdots g_m \right>_N=\delta_{mn}\sum_{\pi\in S_n}\phi_N(\pi)\left<f_1,g_{\pi(1)}\right>\cdots\left<f_n,g_{\pi(n)}\right>.
\end{equation}
This form is positive definite and thus gives an inner product on $\FS^{(alg)}(\HH)$.
Let $\FS_N(\HH)$ be the completion of $\FS^{(alg)}(\HH)$ with respect to the inner product $\left<\cdot,\cdot\right>_N$.
Let $D_N$ be the operator in $\FS^{(alg)}(\HH)$ whose restriction to $\HH^{\otimes n}$ is given by
\begin{equation}
D_{N}^{(n)}:=
\begin{cases}
1+\frac{1}{N}\sum_{k=2}^{n} \tilde U_{n}(\tau_{1,k}),&\text{if $n>0$}\\
1,&\text{otherwise},
\end{cases}
\end{equation}
where $\tau_{i,k}\in S_n$ is the permutation transposing $i$ and $k$ and fixing all other elements of $[n]$ and $\tilde U_{n}$ is the representation of $S_n$ such that  $\tilde U_n(\pi)$ permutes the tensors in $\HH^{\otimes n}$ according to $\pi$.
For $f\in \HH$ let $l(f)$ and $l^*(f)$ denote the left annihilation and creation (respectively) operators on the free Fock space over $\HH$.
We define annihilation and creation operators on $\FS^{(alg)}(\HH)$ by
\begin{equation}
\begin{split}
a_N(f)&=l(f)D_N\\
a_N^*(f)&=l^*(f).
\end{split}
\end{equation}
It was shown in \cite{BG} that these operators are bounded with respect to $\left<\cdot,\cdot\right>_N$ and thus extend to bounded linear operators on $\FS_N(\HH)$.

The map
\begin{equation}
\begin{split}
\FS_N(\HH)&\to \FS_{\hat V_N}(\HH)\\
v_1\otimes\cdots v_n&\mapsto \xi_n\otimes_s v_n\otimes\cdots\otimes v_1
\end{split}
\end{equation}
is unitary.
It was shown in \cite{BG} that the vacuum state on the algebra of creation and annihilation operators on $\FS_{N}(\HH)$ is the Fock state associated to the function $\Pt_N(\VP)=\left(\frac{1}{N}\right)^{n-\rho(\VP)}$.
This shows that $\Pt_{V^{(N)},j^{(N)}}=\Pt_{\hat V^{(N)},\hat j^{(N)}}$ even though $\dim V^{(N)}_n<\dim \hat V^{(N)}_n$ for $n\ge1$.
\end{example}
\primarydivision{Generalized Brownian motions associated to tensor products of representations of $S_\infty$}
\label{pri:tp}
In this \paper{\lprimarydivname} \thesis{\lsecondarydivname}, we are interested in the case where $\IndexSet=\{-1,1\}$ and the $V_\bn$ arise from unitary representations of the group $S_\infty$ of permutations of $\BN$ fixing all but finitely many points.
\begin{equation}
f_n(x)=\begin{cases}
1,&\text{if $x=q_i$ for $1\le i\le n$}\\
0,&\text{otherwise}
\end{cases}
\end{equation}

\begin{notation}
When $\IndexSet=\{1,-1\}$, we will represent a function $\bn:\IndexSet\to\BN$ by the pair $\bn(-1),\bn(1)$, so we write $V_{r,s}$ for $V_\bn$ where $\bn(-1)=r$ and $\bn(1)=s$.
We will also represent an element $(\VP, c)$ of $\IPPi$ as $(\VP_{-1},\VP_{1})$, where $\VP_b=c^{-1}(b)$.
\end{notation}
One of the simplest such cases is that arising from the tensor product of two unitary representations of $S_\infty$.
In this setting, we can prove the following.
\begin{proposition}\label{prop:2}
 Let $(U^{(i)},V^{(i)})$  be unitary representations of $S_\infty$ for $i\in\IndexSet$.
Suppose that each $V^{(i)}_n$ is a subspace of $V^{(i)}$ carrying a unitary representation $U^{(i)}_n$ of $S_n$ with $j^{(i)}_n: V^{(i)}_{n}\to V^{(i)}_{n+1}$ an isometry.
Assume that we have distinguished unit vectors $\xi_{V^{(i)}}\in V^{(i)}_0$ and let $\xi_{V}=  \xi_{V^{(-1)}}\otimes \xi_{V^{(1)}}$.
Let $V_{m,n}=V^{(-1)}_m\otimes V^{(1)}_n$, $j_{-1}=j^{(-1)}\otimes 1$, and $j_1=1\otimes j^{(1)}$.
Then
\begin{equation}
 \Pt_{V,j}(\VP,c)=\Pt_{V^{(-1)},j^{(-1)}}(\VP_{-1})\cdot \Pt_{V^{(1)},j^{(1)}}(\VP_1).
\end{equation}
\end{proposition}
\begin{proof}
From the definitions, it is clear that for any $v_1\in V_m^{(-1)}$ and $v_2\in V_n^{(1)}$, 
\begin{equation}
\begin{split}
j_1j_{-1}(v_1\otimes v_2)&=j_{-1}j_1(v_1\otimes v_2)\\
j_1^*j_{-1}^*(v_1\otimes v_2)&=j_{-1}^*j_1^*(v_1\otimes v_2)\\
j_1^*j_{-1}(v_1\otimes v_2)&=j_{-1}j_1^*(v_1\otimes v_2)\\
j_1j_{-1}^*(v_1\otimes v_2)&=j_{-1}^*j_1(v_1\otimes v_2).
\end{split}
\end{equation}
Consequently, if $b\ne b'$ the operators $a_b^{V,j,e}(f)$ and $a_{b'}^{V,j,e'}(f')$ commute for all $e,e'\in \{1,2\}$ and all $f,f'\in\HH$.

Assume that $\VP:=\{(l_1,r_1),\ldots,(l_m,r_m)\}$ with $l_k<r_k$ and $l_k<l_{k+1}$ for all $k$.
Let $\HH$ be $\ell^2(\BN)$ with orthonormal basis $(h_k)_{k=1}^\infty$
We can compute $ \Pt_{V,j}(\VP,c)$ as
\begin{equation}
 \Pt_{V,j}(\VP,c)=\left<\left(\prod_{p=1}^{2m}a_{c(p)}^{V,j, e_p}(h_{k_p}) \right)\xi_V\otimes_s \Omega,\xi_V\otimes_s \Omega \right>
\end{equation}
where $k_p$ is the unique $k\in [n]$ such that $p$ is an element of the $k$-th pair of $\VP$ and $e_p=2$ if $p$ is a right point in $\VP$ and $e_p=1$ if $p$ is a right point.

Now, using the fact that $a_{c(p)}^{e_r}(h_{k_p})$ commutes with $a_{c(p')}^{e_{p'}}(h_{k_{p'}})$ when $c(p)\ne c(p')$, we have
\begin{equation}
\begin{split}
  \Pt_{V,j}(\VP,c)&=\left<\left(\prod_{c(p)=-1}a_{-1}^{V,j, e_p}(h_{k_p}) \right)\left(\prod_{c(p)=1}a_{1}^{V,j, e_p}(h_{k_p}) \right)\xi_V\otimes_s \Omega,\xi_V\otimes_s \Omega \right>\\
&=\prod_{b\in\IndexSet}\left<\prod_{c(p)=b}a_{b}^{V^{(b)},j^{(b)}, e_p}(h_{k_p})\xi_{V^{(b)}}\otimes_s \Omega,\xi_{V^{(b)}}\otimes_s \Omega\right>\\
&=\Pt_{V^{(-1)},j^{(-1)}}(\VP_1)\cdot \Pt_{V^{(1)},j^{(1)}}(\VP_2).
\end{split}
\end{equation}
This completes the proof.
\end{proof}

Combining Theorem \ref{thm:bg} with our Proposition \ref{prop:2} immediately gives the following.
\begin{corollary}
Let $\IndexSet=\{1,2\}$. Fix $(\alpha_i)_{i=1}^\infty$ and $(\beta_i)_{i=1}^\infty$ decreasing sequences of positive real numbers such that $\sum_{i}\alpha_i+\sum_i\beta_i\le 1$ and let $V_n^{(1)}=V_n^{(2)}=V_n^{(\alpha,\beta)}$ with the Vershik-Kerov representation of $S_n$.
For $i\in\{1,2\}$, let $j^{(i)}:V^{(i)}_n\to V^{(i)}_{n+1}$ be the natural isometry.
Let $\xi_{V^{(i)}}=\bone_0$ and let $\xi_{V}=  \xi_{V^{(1)}}\otimes \xi_{V^{(2)}}$.
Let $V_{m,n}=V^{(-1)}_m\otimes V^{(1)}_n$, $j_{-1}=j^{(-1)}\otimes 1$, and $j_1=1\otimes j^{(1)}$.
Then for $(\VP,c)\in\IPPi$,
\begin{equation}
 \Pt_{V,j}(\VP,c)=\prod_{m\ge 2}\left(\sum_{i=1}^\infty\alpha_i^m+(-1)^{m+1}\sum_{i=1}^\infty\beta_i^m\right)^{\rho_m(\VP_1)+\rho_m(\VP_2)}
\end{equation}
\end{corollary}
\primarydivision{Generalized Brownian motions associated to spherical representations of $(S_\infty\times S_\infty, S_\infty)$}
\label{pri:spherical}
In this \lprimarydivname, we again use the index set $\IndexSet=\{-1,1\}$. Of course, we could have taken $\IndexSet$ to be any two-element set, but we have chosen $\{-1,1\}$ for the reason that if $b\in\IndexSet$ then we can concisely refer to the other index as $-b\in\IndexSet$.

\begin{notation}
With $\IndexSet=\{-1,1\}$, we will represent a function $\bn:\IndexSet\to\BN$ by the pair $\bn(-1),\bn(1)$, so we write $V_{r,s}$ for $V_\bn$ where $\bn(-1)=r$ and $\bn(1)=s$.
\end{notation}

G. Olshanski initiated the study of a broad class of representations of infinite symmetric groups \cite{Olshanski}, and this study has been further developed by Okounkov \cite{Okounkov}.
In this framework, one considers unitary representations of a pair of groups $K\subset G$ forming a Gelfand pair.
Two groups $(G,K)$ form a Gelfand pair if for every unitary representation $(T,\HH)$ of $G$, the operators $P_KT(g)P_K$ commute with each other as $g$ ranges over $G$.
Here $P_K$ denotes the orthogonal projection of $\HH$ onto the subspace of $K$-invariant vectors for the representation $T$.

Of interest to us are the spherical representations, which are defined as irreducible unitary representations of $G$ with a nonzero $K$-fixed vector $\xi$.
If $T$ is such a representation of the pair $(G,K)$, then the function $g\mapsto\left<\xi,T(g)\xi\right>$ is called a spherical function of $(G,K)$.
Here we consider the case where $G=S_\infty\times S_\infty$ and $K=S_\infty$ is the diagonal subgroup.
It is well-known (c.f. \cite{Olshanski}) that the finite factor representations of a discrete group $G$ are in bijective correspondence with the spherical representations of the Gelfand pair $(G\times G, G)$, where $G$ is a subgroup of $G\times G$ by the diagonal embedding.

 In light of Thoma's Theorem (Theorem \ref{thm:thoma}) this means that the spherical functions of  $(S_\infty\times S_\infty,S_\infty)$ are parametrized by the Thoma parameters and that the spherical function associated to the pairs $(\alpha_{i})_{i=1}^\infty$ and $(\beta_{i})_{i=1}^\infty$ is given by the formula
\begin{equation}
 \chi_{\alpha,\beta}\left(\pi,\pi'\right)=\phi_{\alpha,\beta}\left(\pi'\pi^{-1}\right)=\prod_{m\ge2}\left(\sum_{i=1}^\infty\alpha_i^m+(-1)^{m+1}\sum_{i=1}^\infty\beta_i^m\right)^{\rho_m(\pi'\pi^{-1})}.
\end{equation}

For the generalized Brownian motion construction, we can consider the following data.
Let $(\alpha_i)_{i=1}^\infty$ and $(\beta_i)_{i=1}^\infty$ be a Thoma parameter.
That is, let $(\alpha_i)_{i=1}^\infty$ and $(\beta_i)_{i=1}^\infty$ be decreasing sequences of positive real numbers such that $\sum_{i}\alpha_i+\sum_i\beta_i\le 1$.
Given $n_{-1},n_1\in \BN\cup\{0\}$ let $n=\max(n_{-1},n_1)$ and define
\begin{equation}
 V_{n_{-1},n_1}=V_n^{(\alpha,\beta)},
\end{equation}
where $V_n^{(\alpha,\beta)}$ is as in \eqref{eqn:vkhs}.
Then $V_{n_{-1},n_1}$ carries a natural representation of $S_n\times S_n$ defined by
\begin{equation}\label{eqn:sphrep}
(U_n^{(\alpha,\beta)}(\sigma,\pi)h)(x,y)=(-1)^{i(\sigma, x)+i(\pi, y)}h(\sigma^{-1}x,\pi^{-1}y),
\end{equation}
and thus a representation of $S_{n_{-1}}\times S_{n_1}$ considering $S_{n_{-1}}\times S_{n_1}$ as a subgroup of $S_{n}\times S_{n}$.
For $n_{-1}=n_1$, it is easy to see that the indicator function of the diagonal is fixed by the diagonal subgroup.

Moreover, we define the map $j_{-1}: V_{n_{-1},n_1}\to V_{n_{-1}+1,n_1}$ to be the natural embedding.
When $\sum \alpha_i+\sum\beta_i=1$, this means that $j_{-1}$ is given by
\begin{equation}
j_{-1}\delta_{(x^{(-1)},x^{(1)})} =
\begin{cases}
\delta_{(x^{(-1)},x^{(1)})},&\text{if $n_1>n_{-1}$},\\
\sum_{z\in Q} \delta_{((x_1^{(-1)},\ldots, x_n^{(-1)},z),(x_1^{(1)},\ldots, x_n^{(1)},z))}, &\text{otherwise}.
\end{cases}
\end{equation}
Likewise, we define the map $j_{1}: V_{n_{-1},n_1}\to V_{n_{-1},n_1+1}$ to be the natural embedding.

We will also need to make use of the maps $j_{b}^*$ for $b\in\IndexSet$.
The map $j_{-1}^*$ is given by
\begin{equation}
 j_{-1}^*\delta_{\left(x^{(-1)},x^{(1)}\right)} =
\begin{cases}
\delta_{(x^{(-1)},x^{(1)})},&\text{if $n_1\ge n_{-1}$},\\
 \mu\left(x_{n}^{(-1)}\right)\delta_{x^{(-1)}_{n},x^{(1)}_{n}}\delta_{\left(\left(x_1^{(-1)},\ldots, x_{n-1}^{(-1)}\right),\left(x_1^{(1)},\ldots, x_{n-1}^{(1)}\right)\right)},&\text{otherwise}.
\end{cases}
\end{equation}
Here $x^{(n-1)}$ refers to the first $n-1$ terms of the $n$-tuple $(x_1,\ldots, x_{n_{-1}-1})$ and the measure $\mu$ is as in Notation \ref{not:vk}.
The maps $j_1^*$ are defined analogously.

To motivate our results in the 2-colored case, we will consider another interpretation of the cycle decomposition of a pair partition.
This interpretation involves some graph theory.
We assume that a reader is familiar with the notion of a directed graph, a subgraph of a directed graph, and a cycle in a directed graph.
These definitions can be found, for instance, in \cite{BJG}.
For a subgraph $H$ of $G$, we write $V(H)$ to mean the vertex set of $H$ and $A(H)$ to mean the arc set of $H$.

Given a pair partition $\VP\in \PP(2m)$, we define a directed graph $G_{\VP}$ with vertex set $[2m]$.
For each $(l,r)\in \VP$ with $l<r$, we add an arc $(l,r)$ to $G_\VP$.
For each $(l',r')\in \hat\VP$ (as defined in Definition \ref{def:cycles}) with $l'<r'$, we add an arc $(r',l')$ to $G_{\VP}$.

The directed graph $G_{\VP}$ is the union of vertex-disjoint cycles, and the cycles of the graph $G_{\VP}$ give the cycles of the pair partition $\VP$.
More precisely, if $C$ is a cycle of $G_{\VP}$ then $A(C)\cap \VP$ is a cycle of $\VP$.
In particular, this means that $\rho_m(\VP)$ is the number of cycles of $G_{\VP}$ of length $2m$.

\begin{definition}\label{def:monpath}
For a directed graph $G$ whose vertex set $V$ has a total order $<$, an increasing path $P$ in $G$ is a sequence of arcs $(s_1,s_2), (s_2,s_3),\ldots, (s_r,s_{r+1})$ of $G$ such that $s_1<s_2<\cdots<s_{r+1}$.
We call $r$ the length of $P$.
A maximal increasing path in $G$ is an increasing path which is not contained in any increasing path in $G$ of greater length.
We define the notions of decreasing paths and maximal decreasing paths in $G$ analogously.
A monotone path is a path which is either increasing or decreasing, and a maximal monotone path is a monotone path which is not contained in any longer monotone path.
\end{definition}

In the directed graph $G_{\VP}$, each arc is a maximal monotone path, so the length of a cycle is the same as the number of maximal increasing paths in that cycle.

Combining with Theorem \ref{thm:bg},
\begin{equation}
 \Pt_{\alpha,\beta}(\VP)=\prod_{m\ge 2}\left(\sum_{i=1}^\infty\alpha_i^{m}+(-1)^{m+1}\sum_{i=1}^\infty\beta_i^{m}\right)^{\gamma_{m}(G_{\VP})}.
\end{equation}
where $\gamma_{m}(G_{\VP})$ denotes the number of cycles of $G_{\VP}$ having $m$ maximal increasing paths.

The case of a 2-colored pair partition is naturally more complicated.
As in the uncolored case, our function on 2-colored pair partitions $(\VP,c)$ will be calculated with the aid of the cycle decomposition of a directed graph (denoted $G_{\VP,c}$), but the construction of a graph from a 2-colored pair partition will be rather more involved.
However, in the case that the coloring function is the constant function $c(l,r)=1$, the graph $G_{\VP,c}$ will be identical to the graph $G_{\VP}$ just described.

Before defining the graph $G_{\VP,c}$ we fix some notation.
\begin{notation}
 For $(\VP,c)\in\IPPi$ with $\VP=\{(l_1,r_1),\cdots, (l_m,r_m)\}$, let $L_\VP:=\{l_1,\ldots,l_n\}$ be the set of left points and $R_\VP:=\{r_1,\ldots, r_n\}$ denote the set of right points.
 If $c:\VP\to\{-1,1\}$ is a coloring function, define for $b\in\IndexSet$ the functions
 \begin{equation}
  \begin{split}
   \vl_{\VP,c}^{b}&:[0,2m+1]\to \BN\cup\{0\}\\
   \vl_{\VP,c}^{b}(u)&=\left|\left\{j\in[n]: l_j\le  u\le  r_j, c(l_j,r_j)=b\right\}\right|
  \end{split}
 \end{equation}
Also define
\begin{equation}
  \vl_{\VP,c}(u)=\max\{\vl_{\VP,c}^{-1}(u),\vl_{\VP,c}^{1}(u)\}\quad \text{and}\quad \vls_{\VP,c}(u)=\vl_{\VP,c}^{c(u)}(u).
\end{equation}
\end{notation}
\begin{remark}
In terms of the diagrams (e.g. Figure \ref{fig:pp}), $\vl_{\VP,c}^{b}(m)$ is the number of $b$-colored paths intersecting the vertical line drawn through $m$ (provided that the diagrams are drawn so as to minimize this quantity).
Furthermore, $\vls_{\VP,c}(m)$ is the number of paths of the same color as $m$ which intersect the vertical line drawn through $m$ (provided that the diagrams are drawn so as to minimize this quantity).
\end{remark}

The following properties are immediate consequences of the definitions and will be used frequently.
\begin{proposition}\label{prop:pprops}
Suppose that $(\VP,c)\in\IBPPi$, $\VP=\left\{(l_1,r_1),\ldots, (l_m,r_m)\right\}$ and $b\in\IndexSet=\{-1,1\}$.
\begin{enumerate}
\item If $k\in [2m]$ and $\vl^{b}_{\VP,c}(k)>\vl^{b}_{\VP,c}(k-1)$ then $c(k)=b$ and $k\in L_\VP$. \label{itm:ls}
\item If $k\in [2m]$ and $\vl^{b}_{\VP,c}(k+1)<\vl^{b}_{\VP,c}(k)$ then $c(k)=b$ and $k\in R_\VP$. \label{itm:rs}
\item If $k\in [2m+1]$ then $\vl^{b}_{\VP,c}(k)-\vl^{b}_{\VP,c}(k-1)\in \{-1,0,1\}$. \label{itm:jump}
\item If $k,k'\in [0, 2m+1]$ with $k<k'$, $\vl^{b}_{\VP,c}(k)<\vl^{b}_{\VP,c}(k')$ and $u \in [\vl^{b}_{\VP,c}(k), \vl^{b}_{\VP,c}(k')]$, then there is some $l\in [k,k']$ such that $\vl^b_{\VP,c}(l)=u$.\label{itm:ivpl}
\item If $k,k'\in [0, 2m+1]$ with $k<k'$, $\vl^{b}_{\VP,c}(k)>\vl^{b}_{\VP,c}(k')$ and $u \in [\vl^{b}_{\VP,c}(k'), \vl^{b}_{\VP,c}(k)]$, then there is some $l\in [k,k']$ such that $\vl^b_{\VP,c}(l)=u$.\label{itm:ivpr}
\item If $k\in L_\VP$ and $k\in[2m]$ then $\vls_{\VP,c}(k)\ge \vl_{\VP,c}^{c(k)}(k-1)$ and $\vl^{b}_{\VP,c}(k)\ge \vl_{\VP,c}^{b}(k+1)$ for $b\in\IndexSet$.\label{itm:l}
\item If $k\in R_\VP$ and $k\in [2m]$ then $\vls_{\VP,c}(k)\ge \vl_{\VP,c}^{c(k)}(k+1)$ and $\vl^{b}_{\VP,c}(k)\ge \vl_{\VP,c}^{b}(k-1)$ for $b\in\IndexSet$.\label{itm:r}
\end{enumerate}
\end{proposition}

We need some additional notation.
\begin{notation}
For an $\IndexSet$-indexed pair partition $(\VP,c)$, define 
\begin{equation}
\begin{split}
\dom_{\VP,c}&=\{k\in[2m]: \vls_{\VP,c}(k)> \vl_{\VP,c}^{-c(k)}(k)\}\\
\subo_{\VP,c} &= [2m]\setminus \dom_{\VP,c}.
\end{split}
\end{equation}
\end{notation}
The next remark should clarify the importance of these terms.
\begin{remark}
If $(\VP,c)\in \IPP(2n)$ then by \eqref{eqn:fock},  $\Pt_{\alpha,\beta}(\VP,c)$ can be computed by evaluating the vacuum state at a word in creation and annihilation operators.
More precisely, we can write
\begin{equation}
\Pt_{\alpha,\beta}(\VP,c) = \left<\xi_{V^{\alpha,\beta}}\otimes \Omega,  A_1\cdots A_{2n}\left(\xi_{V^{\alpha,\beta}}\otimes \Omega\right)\right>
\end{equation}
where $A_k$ is a creation operator if $k\in R_{\VP,c}$ and an annihilation operator if $k\in L_{\VP,c}$.
 In either case, the color of the operator $A_k$ is $c(k)$.
One can characterize these operators more precisely, but we will not need to do so at this point.

For any $k\in [2n]$, the vector $\left(A_{k+1}\cdot\cdots\cdot A_{2n}\right)\xi_{V^{\alpha,\beta}}\otimes \Omega$ lies in the space $V_{\bn_k}^{\alpha,\beta}\otimes \HH^{\otimes \bn_k}$ for some function $\bn_k:\IndexSet\to\BZ$.  
A $c(k)$-colored creation operator maps the space $V_{\bn_k}^{\alpha,\beta}\otimes \HH^{\otimes \bn_k}$ to $V_{\bn_k+\delta_{c(k)}}^{\alpha,\beta}\otimes \HH^{\otimes \bn_k+\delta_{c(k)}}$.
If $k\in R_\VP$, then $k\in\subo_{\VP,c}$ if and only if the transition map 
\begin{equation}
j_{c(k)}: V_{\bn_k}^{\alpha,\beta}\to V_{\bn_k+\delta_{c(k)}}^{\alpha,\beta}
\end{equation}
is the identity map on $V_{n_k}^{\alpha,\beta}$, where $n_k=\max \{\bn_k(b): b\in\IndexSet\}$.
The analogous statement also holds for $k\in L_{\VP}$.

\end{remark}
We will make use of the following equivalence relation on $[2m]$.
\begin{notation}
For $k,k'\in[2m]$, say that $k\vcequiv k'$ if $\vls_{\VP,c}(k)=\vls_{\VP,c}(k')$.
\end{notation}
\begin{proposition} 
 If $k\in L_\VP$ then $\{k'>k: k'\vcequiv k\}\ne \emptyset$.
 If $k\in R_\VP$ then $\{k'<k: k'\vcequiv k\}\ne \emptyset$.
\end{proposition}
\begin{proof}
 We will consider the case in which $k\in L_\VP$.
 By Proposition \ref{prop:pprops} (item \ref{itm:l}), $\vl_{\VP,c}^{c(k)}(k+1)\ge \vl_{\VP,c}^{c(k)}(k)$.
 Let 
 \begin{equation}
 r=\max\{k'>k: \vl_{\VP,c}^{c(k)}(k') \ge \vls_{\VP,c}(k)\},
 \end{equation}
 so that $\vl_{\VP,c}^{c(k)}(r+1) < \vls_{\VP,c}(k)$
By Proposition \ref{prop:pprops} (item \ref{itm:jump}), $\vl_{\VP,c}^{c(k)}(r)-\vl_{\VP,c}^{c(k)}(r+1)=1$, so it must be the case that $\vl_{\VP,c}^{c(k)}(r) = \vls_{\VP,c}(k)$.
By Proposition \ref{prop:pprops} (item \ref{itm:rs}), $c(r)=c(k)$ whence $\vls_{\VP,c}(r) = \vls_{\VP,c}(k)$ and $k\vcequiv r$.
\end{proof}
In defining the graph $G_{\VP,c}$, the following function will be very important.
\begin{definition}
Define a map $\ZB_{\VP,c}(k):[2m]\to[2m]$ by
\begin{equation}
\ZB_{\VP,c}(k):=
\begin{cases}
\min\{k'>k: k'\vcequiv k\},&\text{if $k\in L_\VP$ and $k\in \dom_{\VP,c}$};\\
\max\{k'<k: k'\vcequiv k\},&\text{if $k\in R_\VP$ and $k\in \dom_{\VP,c}$};\\
\max\{k'<k: k'\vcequiv k \},&\text{if $k\in L_\VP$ and $k\in \subo_{\VP,c}$};\\
\min\{k'>k:  k'\vcequiv k\} ,&\text{if $k\in R_\VP$ and $k\in \subo_{\VP,c}$}.
\end{cases}
\end{equation}
\end{definition}

\begin{notation}
 For $k\in[2m]$, let $I_{\VP,c}(k)$ be the interval
\begin{equation}
 I_{\VP,c}(k):=
\begin{cases}
 [k+1,\ZB_{\VP,c}(k)-1],&\text{if $\ZB_{\VP,c}(k)>k$}\\
 [\ZB_{\VP,c}(k)+1,k-1],&\text{if $\ZB_{\VP,c}(k)<k$}.
\end{cases}
\end{equation}
\end{notation}

\begin{proposition}\label{prop:dom}
Suppose that $k\in \dom_{\VP,c}$ and $k'\in I_{\VP,c}(k)$.
Then 
\begin{equation}
\vl_{\VP,c}^{-c(k)}(k')<\vls_{\VP,c}(k)\le \vl_{\VP,c}^{c(k)}(k')
\end{equation}
\end{proposition}
\begin{proof}
We will assume that $k\in L_\VP$ since the case of $k\in R_\VP$ is similar.
Suppose that $k'\in I_{\VP,c}(k)$ satisfies $\vl_{\VP,c}^{c(k)}(k')< \vl_{\VP,c}^{c(k)}(k)$.
We can assume that $k'$ is the smallest element of $I_{\VP,c}(k)$ satisfying this inequality, so that by Proposition \ref{prop:pprops} ,  $\vl_{\VP,c}^{c(k)}(k'-1)= \vl_{\VP,c}^{c(k)}(k)$ (by item \ref{itm:jump}) and $c(k'-1)=c(k)$ (by item \ref{itm:rs}).
Thus, $k'-1\vcequiv k$, which contradicts the definition of $I_{\VP,c}(k)$.

Now suppose that $\vl_{\VP,c}^{-c(k)}(k')\ge \vl_{\VP,c}^{c(k)}(k)$.
Assume that $k'$ is the smallest element of $I_{\VP,c}(k)$ satisfying this condition so that $\vl_{\VP,c}^{-c(k)}(k'-1)< \vl_{\VP,c}^{c(k)}(k)$ (again it is straightforward to rule out the possibility that $k'=k+1$).
By Proposition \ref{prop:pprops}, $\vl_{\VP,c}^{-c(k)}(k'-1)-\vl_{\VP,c}^{-c(k)}(k')=-1$ (by item \ref{itm:jump}) and $c(k'-1)=-c(k)$ (by item \ref{itm:l}).
Thus, $\vl_{\VP,c}^{-c(k)}(k'-1)= \vl_{\VP,c}^{c(k)}(k)$ whence $k'-1\vcequiv k$, contradicting the definition of $I_{\VP,c}(k)$.
\end{proof}

\begin{proposition}\label{prop:sub}
If $k\in\subo_{\VP,c}$ and $k'\in I_{\VP,c}(k)$ then 
\begin{equation}
\vl_{\VP,c}^{c(k)}(k')< \vl_{\VP,c}^{-c(k)}(k)\le \vl_{\VP,c}^{-c(k)}(k').
\end{equation}
\end{proposition}
\begin{proof}
We will again assume that $k\in L_\VP$ as the case of $k\in R_\VP$ is similar.
Suppose that $k'\in I_{\VP,c}(k)$ is such that $\vl_{\VP,c}^{c(k)}(k')\ge  \vl_{\VP,c}^{-c(k)}(k)$.
Assume that $k'$ is the largest element of $I_{\VP,c}(k)$ satisfying this inequality.
One can check that if $k'=k-1$ then $k-1\vcequiv k$ and $I_{\VP,c}(k)=\emptyset$, so we assume that $k'+1\in I_{\VP,c}(k)$ and thus $\vl_{\VP,c}^{c(k)}(k'+1)< \vl_{\VP,c}^{-c(k)}(k)$.
By Proposition \ref{prop:pprops}, it follows that $\vl_{\VP,c}^{c(k)}(k'+1)-\vl_{\VP,c}^{c(k)}(k')=-1$ and $c(k'+1)=c(k)$.
Thus, $k'+1\vcequiv k$, which contradicts the definition of $I_{\VP,c}(k)$.

Now suppose that there is some $k'\in I_{\VP,c}(k)$ such that $\vl_{\VP,c}^{-c(k)}(k')< \vl_{\VP,c}^{-c(k)}(k)$.
Again assume that $k'$ is the largest element of $I_{\VP,c}(k)$ satisfying this inequality.
Since $\vl_{\VP,c}^{-c(k)}(k-1)\ge \vl_{\VP,c}^{-c(k)}(k)$, we must have $k\ne k-1$, so that $k'\in I_{\VP,c}(k)$ and $\vl_{\VP,c}^{-c(k)}(k'+1)> \vl_{\VP,c}^{-c(k)}(k)$.
By Proposition \ref{prop:pprops}, $\vl_{\VP,c}^{-c(k)}(k'+1)-\vl_{\VP,c}^{-c(k)}(k)=1$ and $c(k'+1)=-c(k)$.
But also $\vl_{\VP,c}^{-c(k)}(k'+1)= \vl_{\VP,c}^{c(k)}(k)$ whence $k'+1\vcequiv k$, contradicting the definition of $I_{\VP,c}(k)$.
\end{proof}
\begin{proposition}\label{prop:Zs}
Suppose that $(\VP,c)\in \IPP(2m)$ and $k\in[2m]$.
Then the following hold:
\begin{enumerate}
\item If $k\in \dom_{\VP,c}$ then $\ZB_{\VP,c}(k)\in \dom_{\VP,c}$ if and only if $c(k)=c(\ZB_{\VP,c}(k))$;\label{itm:dom}
\item If $k\in \subo_{\VP,c}$ then $\ZB_{\VP,c}(k)\in \subo_{\VP,c}$ if and only if $c(k)=c(\ZB_{\VP,c}(k))$;\label{itm:sub}
\item If $k\in L_\VP$ then $\ZB_{\VP,c}(k)\in L_\VP$ if and only if $c(\ZB_{\VP,c}(k))=-c(k)$;\label{itm:L}
\item If $k\in R_\VP$ then $\ZB_{\VP,c}(k)\in R_\VP$ if and only if $c(\ZB_{\VP,c}(k))=-c(k)$; \label{itm:R}
\item $\ZB_{\VP,c}(\ZB_{\VP,c}(k))=k$. \label{itm:inv}
\end{enumerate}
\end{proposition}
\begin{proof}
We fix an $\IndexSet$-indexed pair partition $(\VP,c)$.
For compactness and readability, we will abbreviate $\ZB_{\VP,c}(k)$ by $k^*$.

If $k\in L_\VP\cap \dom_{\VP,c}$ then by definition $k^*>k$.
By Proposition \ref{prop:dom}, 
\begin{equation}
\vl_{\VP,c}^{-c(k)}(k^*-1)<\vls_{\VP,c}(k)< \vl_{\VP,c}^{c(k)}(k^*-1).
\end{equation}
Using the fact that $k^*\vcequiv k$ and applying Proposition \ref{prop:pprops}, it  follows that if $c(k^*)=c(k)$ then $k^*\in R_\VP$ and $k^*\in \dom_{\VP,c}$ whereas if $c(k^*)=c(-k)$ then $k^*\in L_\VP$ and $k^*\in\subo{\VP,c}$.

If $k\in L_\VP\cap\subo_{\VP,c}$ then by definition $k^*<k$.
By Proposition \ref{prop:sub}, 
\begin{equation}
\vl_{\VP,c}^{c(k)}(k^*+1)< \vl_{\VP,c}^{-c(k)}(k)<\vl_{\VP,c}^{-c(k)}(k^*+1).
\end{equation}
Using the fact that $k^*\vcequiv k$ and applying Proposition \ref{prop:pprops}, it  follows that if $c(k^*)=c(k)$ then $k^*\in R_\VP$ and $k^*\in\subo_{\VP,c}$ whereas if $c(k^*)=c(-k)$ then $k^*\in L_\VP$ and $k^*\in \dom_{\VP,c}$.

Collecting all of these cases, as well as the analogous statements for $k\in R_\VP$ gives items \ref{itm:dom}, \ref{itm:sub}, \ref{itm:L}, and \ref{itm:R}.
Making use of the definition of $\ZB_{\VP,c}$, it follows that $(k^*)^*<k^*$ if and only if $k<k^*$, whence $k=(k^*)^*$.
\end{proof}

\begin{corollary}
 If $(\VP,c)\in \IPP(2m)$ then the function $\ZB_{\VP,c}:[2m]\to [2m]$ defines a pair partition $\bar\VP^{(c)}\in \IPP(2m)$ by
\begin{equation}
\bar\VP^{(c)}:=\left\{\left(k,\ZB_{\VP,c}(k)\right): k<\ZB_{\VP,c}(k) \right\}.
\end{equation}
\end{corollary}

We define a coloring function on $\bar\VP^{(c)}$ by:
\begin{equation}
 \bar c\left(k,\ZB_{\VP,c}(k)\right)=
\begin{cases}
 c(k),&\text{if $k\in\subo_{\VP,c}$}\\
-c(k),&\text{if $k\in\dom_{\VP,c}$}.
\end{cases}
\end{equation}
\begin{remark}
This definition of $\bar c\left(k,\ZB_{\VP,c}(k)\right)$ does not depend on which point of a pair is chosen as $k$ because $c(k)=c(\ZB_{\VP,c}(k))$ if and only if $k$ and $\ZB_{\VP,c}(k)$ are either both in $\subo_{\VP,c}$ or both in $\dom_{\VP,c}$.
\end{remark}
\begin{notation}\label{not:inv}
Define a map $(\cdot,\cdot)^{(-1)}$ on $[2m]\times [2m]$ by $(u,v)^{(-1)}=(v,u)$.
Also let $(\cdot,\cdot)^{(1)}$ be the identity map on $[2m]\times [2m]$.
\end{notation}

We are now ready to define the graph $G_{\VP,c}$.
\begin{definition}
If $(\VP,c)\in\IPP(2m)$, then $G_{\VP,c}$ is the directed graph with vertices $[2m]$ and arcs defined as follows.
Let 
\begin{equation}
\begin{split}
F_{\VP,c}&=\{(l,r)^{(c(l,r))}:(l,r)\in \VP \}\\
\bar F_{\VP,c}&=\{(k,k')^{(\bar c(k,k'))}: (k,k')\in \bar\VP^{(c)} \}.
\end{split}
\end{equation}

The graph $G_{\VP,c}$ has arc set 
\begin{equation}
A_{\VP,c}:=F_{\VP,c}\cup \bar F_{\VP,c}
\end{equation}
\end{definition}

\begin{proposition}
 The sets $F_{\VP,c}$ and $\bar F_{\VP,c}$ have empty intersection.
\end{proposition}
\begin{proof}
We need only show that if $(l,r)\in \VP$ with $\ZB_{\VP,c}(l)=r$ then $c(l,r)=-\bar c(l,r)$.
By the definition of $\bar c$, it will suffice to show that $l\in\dom_{\VP,c}$.
Since $l<\ZB_{\VP,c}(l)$, $\ZB_{\VP,c}(l)=r\in R_\VP$ and $c(l)=c(\ZB_{\VP,c}(l))$, this follows from Proposition \ref{prop:Zs} and the definition of $\ZB_{\VP,c}$.
\end{proof}

\begin{example}
We consider the example of the graph $G_{\VP,c}$ for the $\IndexSet$-indexed pair partition $(\VP,c)$ with
\begin{equation}\label{eqn:ppex}
\begin{split}
\VP=\{(1,5),(2,10),(3,8),(6,7),(4,12),(9,11)\};\\
c(1,5)=c(2,9)=c(8,10)=c(6,7)=-1;\\
c(3,8)=c(4,12)=1.
\end{split}
\end{equation}
The $\IndexSet$-indexed pair partition $(\VP,c)$ is depicted in Figure \ref{fig:ppex}.

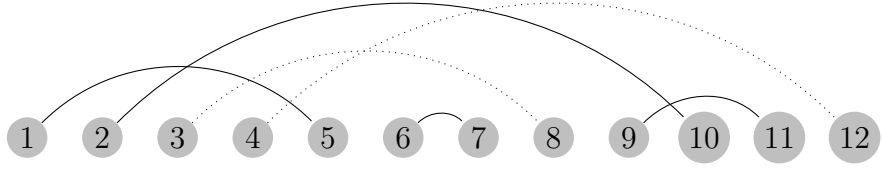
\begin{figure}
 \begin{tikzpicture}[shorten >=1pt,->]
  \tikzstyle{vertex}=[circle,fill=black!25,minimum size=14pt,inner sep=2pt]
  \foreach \x in {1,...,12} 
    \node[vertex] (\x) at (\x,0) {\x};
  \draw [-] (1) to[out=45,in=135] (5);
  \draw [-] (2) to[out=45,in=135] (10);
  \draw [-] (9) to[out=45,in=135] (11);
  \draw [-,style=dotted] (3) to[out=45,in=135] (8);
  \draw [-,style=dotted] (4) to[out=45,in=135] (12);
  \draw [-] (6) to[out=45,in=135] (7);
\end{tikzpicture}
\caption[An indexed pair partition $(\VP, c)$]{The $\{-1,1\}$-indexed pair partition $(\VP,c)$ given in \eqref{eqn:ppex}.}
\label{fig:ppex}
\end{figure}


We can determine the values of $\vl_{\VP,c}^{-1}(k)$ and $\vl_{\VP,c}^{1}(k)$ by drawing a vertical line through the diagram at $k$ and counting the intersections with paths of the respective colors.
For instance, a vertical line through $k=3$ intersects 2 solid paths and 1 dotted path (at the endpoint), so $
\vl^{(1)}_{\VP,c}(3)=1$ and $\vl^{(-1)}_{\VP,c}(3)=-1$.
Since $c(3)=1$, this means that $3\in\subo_{\VP,c}$.

Repeating the process for the other elements of $[12]$ we can fill in the first few rows of Table \ref{tab:pex}.
From the data, one sees that the equivalence class of $3$ under $\vcequiv$ is $\{1, 3, 11,12\}$ whence 
\begin{equation}
 \ZB_{\VP,c}(3)=1.
\end{equation}

Continuing in this way for other elements of $2m$, the pair partition $\bar\VP^{(c)}$ is given by
\begin{equation}
 \bar\VP^{(c)}:=\{(1,3),(2,4), (5,6), (7,8), (9,10), (11,12) \}
\end{equation}
and the color function $\bar c:\bar\VP^{(c)}\to \IndexSet$ is given by 
\begin{equation}
\begin{split}
&\bar c(5,6)=\bar c(7,8))=\bar c(11,12)=-1;\\
&\bar c(1,3)=\bar c(2,4))=\bar c(9,10)=1.
 \end{split}
\end{equation}
This $\IndexSet$-indexed broken pair partition is depicted in Figure \ref{fig:ppexb}.
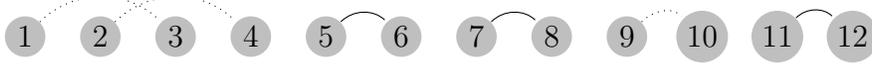
\begin{figure}
 \begin{tikzpicture}[shorten >=1pt,->]
  \tikzstyle{vertex}=[circle,fill=black!25,minimum size=14pt,inner sep=2pt]
  \foreach \x in {1,...,12} 
    \node[vertex] (\x) at (\x,0) {\x};
  \draw [-, style=dotted] (1) to[out=45,in=135] (3);
  \draw [-, style=dotted] (2) to[out=45,in=135] (4);
  \draw [-] (5) to[out=45,in=135] (6);
  \draw [-] (7) to[out=45,in=135] (8);
  \draw [-, style=dotted] (9) to[out=45,in=135] (10);
  \draw [-] (11) to[out=45,in=135] (12);
\end{tikzpicture}
\caption[The indexed pair partition $(\bar{\VP}^{(c)},\bar c)$]{The $\IndexSet$-indexed pair partition $(\bar \VP^{(c)},\bar c)$ where $(\VP,c)$ is the pair partition depicted in the Figure \ref{fig:ppex}.}
\label{fig:ppexb}
\end{figure}

Accordingly, the sets $F_{\VP,c}$ and $\bar F_{\VP,c}$ are given by
\begin{equation}
 \begin{split}
  F_{\VP,c}&=\{ (5,1), (10,2), (3,8), (4,12), (7,6), (11,9)\}\\
  \bar F_{\VP,c}&=\{ (6,5), (12,11), (1,3), (2,4), (7,8), (9,10)\}
 \end{split}
\end{equation}

\begin{table}
\begin{tabular}{*{6}{|c}|}\hline
$k$	&$c(k)$	&$\vls_{(\VP,c)}(k)$&$\vl_{(\VP,c)}^{-c(k)}(k)$	& $\dom_{\VP,c}$ or $\subo_{\VP,c}$	&$Z_{(\VP,c)}(k)$ \\ \hline
1	&-1	&1			&0			& $\dom_{\VP,c}$	&3 \\ \hline
2	&-1	&2			&0			& $\dom_{\VP,c}$	&4 \\ \hline
3	&1	&1			&2			& $\subo_{\VP,c}$		&1 \\ \hline
4	&1	&2			&2			& $\subo_{\VP,c}$		&2 \\ \hline
5	&-1	&2			&2			& $\subo_{\VP,c}$		&6 \\ \hline
6	&-1	&2			&2			& $\subo_{\VP,c}$		&7 \\ \hline
7	&-1	&2			&2			& $\subo_{\VP,c}$		&8 \\ \hline
8	&1	&2			&1			& $\dom_{\VP,c}$	&7 \\ \hline
9	&-1	&2			&1			& $\dom_{\VP,c}$	&10 \\ \hline
10	&-1	&2			&1			& $\dom_{\VP,c}$	&9 \\ \hline
11	&-1	&1			&1			& $\subo_{\VP,c}$		&12 \\ \hline
12	&-1	&1			&1			& $\dom_{\VP,c}$	&11 \\ \hline
\end{tabular}
\caption{Data pertaining to the $\IndexSet$-indexed pair partition depicted in Figure \ref{fig:stdform}.}
\label{tab:pex}
\end{table}

The directed graph $G_{\VP,c}$ is depicted in Figure \ref{fig:graph}.
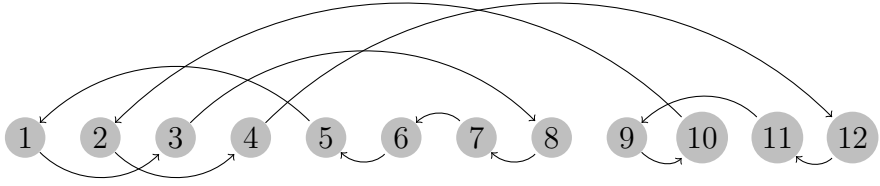
\begin{figure}
 \begin{tikzpicture}[shorten >=1pt,->]
  \tikzstyle{vertex}=[circle,fill=black!25,minimum size=14pt,inner sep=2pt]
  \foreach \x in {1,...,12} 
    \node[vertex] (\x) at (\x,0) {\x};
  \draw [<-] (1) to[out=45,in=135] (5);
  \draw [<-] (2) to[out=45,in=135] (10);
  \draw [->] (3) to[out=45,in=135] (8);
  \draw [->] (4) to[out=45,in=135] (12);
  \draw [<-] (6) to[out=45,in=135] (7);
  \draw [<-] (9) to[out=45,in=135] (11);
  \draw [<-] (5) to[out=-45,in=-135] (6);
  \draw [<-] (11) to[out=-45,in=-135] (12);
  \draw [->] (1) to[out=-45,in=-135] (3);
  \draw [->] (2) to[out=-45,in=-135] (4);
  \draw [<-] (7) to[out=-45,in=-135] (8);
  \draw [->] (9) to[out=-45,in=-135] (10);
\end{tikzpicture}
\caption[The directed graph associated to a two-colored pair partition]{The directed graph $G_{\VP,c}$ for the $\{-1,1\}$-colored pair partition depicted in Figure \ref{fig:ppex}. The graph $G_{\VP,c}$ has 2 cycles, one with vertex set $\{1,3,8,7,6,5\}$ and one with vertex set $\{2,4,12,11,9,10\}$. The former cycle has one maximal increasing path, and the latter has 2 maximal increasing paths.}

\label{fig:graph}
\end{figure}
\end{example}

\begin{proposition}\label{lem:cycles}
 Let $(\VP,c)\in\IPPi$ be a $\{-1,1\}$-colored pair partition.
Then the graph $G_{\VP,c}$ is the union of vertex-disjoint directed cycles.
\end{proposition}
\begin{proof}
Let $m=|\VP|$.
It is an immediate consequence of the definition of $G_{\VP,c}$ that for each $k\in[2m]$ there are exactly two arcs having $k$ as either the start point or the end point.
We must show that each vertex $k\in[2m]$ is the starting point of one edge and the end point of another edge. 

We will assume $k\in L_\VP$, so that $(k,r)\in \VP$ for some $r\in[2m]$.
If $k\in\dom_{\VP,c}$ then $\ZB_{\VP,c}(k)>k$, so $(k,\ZB_{\VP,c}(k))\in \bar\VP^{(c)}$.
Since $c(k,r)=-\bar c(k,\ZB_{\VP,c}(k))$, $k$ is the starting point of exactly one of these arcs.
If $k\in\subo_{\VP,c}$ then $\ZB_{\VP,c}(k)<k$, so $(\ZB_{\VP,c}(k), k)\in \bar\VP^{(c)}$.
Furthermore, $c(k,r)=\bar c(\ZB_{\VP,c}(k), k)$, whence $k$ is again the starting point of exactly one of the arcs.
\end{proof}
\begin{remark}
 If $c(d)=1\in\IndexSet$ for every $d\in \VP$ then every $\dom_{\VP,c}=[2m]$ and $\bar\VP^{(c)} = \hat{\VP}$ (where $\hat{\VP}$ is as in Definition \ref{def:cycles}) with $\bar c(d)=-1$ for all $d\in \bar\VP^{(c)}$.
Thus  $G_{\VP,c}=G_{\VP}$.
\end{remark}
A cycle $C$ in the graph $G_{\VP,c}$ can be decomposed into maximal monotone paths (Definition \ref{def:monpath}).
That is, there are maximal monotone paths $P_1,\ldots, P_s$ such that each arc of $C$ lies along exactly one of the $P_j$.

\begin{notation}
For a directed graph $G$ on a totally ordered vertex set, we denote by $\gamma_m(G)$ the number of cycles of $G$ having exactly $m$ maximal increasing paths (equivalently, $m$ maximal decreasing paths).
\end{notation}

With this established, we are able to state the main result of this section.
\begin{theorem}\label{thm:sppf}
Let $(\alpha_i)_{i=1}^\infty$ and $(\beta_i)_{i=1}^\infty$ be decreasing sequences of positive real numbers such that $\sum_{i}\alpha_i+\sum_i\beta_i\le 1$.
Let $V_{n_{-1},n_{1}}^{(\alpha,\beta)}$ be the complex Hilbert space of the Vershik-Kerov representation of $S_{\max(n_{-1},n_1)}$ endowed with the representation of \eqref{eqn:sphrep}, and let
\begin{equation}
j_n^{-1}:V_{n_{-1},n_{1}}^{(\alpha,\beta)}\to V_{n_{-1}+1,n_{1}}^{(\alpha,\beta)}\quad\text{and}\quad
j_n^{1}:V_{n_{-1},n_{1}}^{(\alpha,\beta)}\to V_{n_{-1},n_{1}+1}^{(\alpha,\beta)}
\end{equation}
 be the natural embedding.
Let $\xi_{V^{(\alpha,\beta)}}=\bone_0$ be the indicator function of the diagonal.
Denote by $\Pt_{\alpha,\beta}$ the function on $\IPPi$ associated to this sequence of representations by Theorem \ref{thm:pf}.

 Let $(\VP,c)$ be a $\{-1,1\}$-colored pair partition.
Then
\begin{equation}
 \Pt_{\alpha,\beta}(\VP,c)=\prod_{m\ge 2}\left(\sum_{i=1}^\infty\alpha_i^{m}+(-1)^{m+1}\sum_{i=1}^\infty\beta_i^{m}\right)^{\gamma_m(G_{\VP,c})}.
\end{equation}
\end{theorem}
We will present the proof of Theorem \ref{thm:sppf} in the case that $\sum_n\alpha_n=1$.
This case will contain the key ideas of the more general argument, but will simplify notation considerably and allow us to consider discrete sums instead of integrals.

\begin{notation}
For $(\VP,c)\in \IPP(2m)$, define
\begin{equation}
\begin{split}
R_{\VP,c}^D&:=\{k\in R_{\VP}: \text{$k\in\dom_{\VP,c}$} \}\\
L_{\VP,c}^D&:=\{k\in L_{\VP}: \text{$k\in\dom_{\VP,c}$} \}.
\end{split}
\end{equation}
\end{notation}

\begin{proposition}
Let $\IndexSet=\{-1,1\}$ and let $(\VP,c)\in \IPP(2m)$.
Every maximal increasing path in $(\VP,c)$ has its starting point in $L_{\VP,c}^D$ and its terminal point in $R_{\VP,c}^D$.
Every maximal decreasing path in $(\VP,c)$ has its starting point in $R_{\VP,c}^D$ and its terminal point in $L_{\VP,c}^D$.
\end{proposition}
\begin{proof}
This follows immediately from the definitions of the map $\ZB_{\VP,c}$ and the graph $G_{\VP,c}$.
A left point $l\in L_{\VP,c}^D$, is adjacent to vertices $u,v$ with $l<u$ and $l<v$ whereas if $l\in L_\VP\in\subo_{\VP,c}$, $l$ is adjacent to one vertex $u'$ with $u'<l$ and another vertex $v'$ with $l<v'$.
\end{proof}

The proof of Theorem \ref{thm:sppf} will be aided by some additional terminology.
\begin{definition} \label{def:elem}
 Let $\HH$ be an infinite-dimensional complex Hilbert space with orthonormal basis $\basis:=\{h_i\}_{i=1}^\infty$.
We will say that a vector $\eta\in \FS_V(\HH)$ is $\basis$-elementary if there is a constant $C$, and some $n_{-1},n_{1}$ such that there is a pair of tuples $x^{(-1)},x^{(1)}\in \SX_n$ of length $n:=\max(n_{-1},n_1)$  with $x^{(-1)}\sim x^{(1)}$ (where $\SX_n$ and the relation $\sim$ are as in Notation \ref{not:vk}) and indices $i_1^{(-1)},\ldots, i_{n_{-1}}^{(-1)}\in\BN$ and $i_1^{(1)},\ldots, i_{n_{1}}^{(1)}\in\BN$ all distinct such that 
\begin{equation} \label{eqn:elemform}
 \eta=C \cdot \delta_{\left(x^{(-1)},x^{(1)}\right)}\otimes_s h_{i_1^{(-1)}}\otimes \cdots\otimes h_{i_{n_{-1}}^{(-1)}}\otimes h_{i_1^{(1)}}\otimes \cdots\otimes h_{i_{n_{1}}^{(1)}}.
\end{equation}
Denote by $\elem_V^\basis(\HH)$ the set of all $\basis$-elementary vectors in $\FS_V(\HH)$.
\end{definition}
\begin{notation}
For the remainder of this section, we assume that $\HH$ is an infinite-dimensional complex Hilbert space with orthonormal basis $\basis:=\{h_i\}_{i=1}^\infty$.
 For a nonzero $\basis$-elementary vector $\eta\in V_{n_{-1},n_{1}}\otimes_s \HH^{\otimes n_{-1}}\otimes \HH^{\otimes n_{1}}\subset \FS_V(\HH)$, define $\bn_\eta:\IndexSet\to \BN$ by $\bn_\eta(b)=n_b$.
\end{notation}

The next proposition says that, up to permutations, a nonzero $\basis$-elementary vector can be expressed uniquely in the form \eqref{eqn:elemform}.
It follows immediately from the definition of the symmetric tensor product $\otimes_s$.
\begin{proposition}\label{prop:elemperm}
Suppose that a nonzero $\basis$-elementary vector $\eta$ has two expressions of the form \eqref{eqn:elemform},
\begin{equation} 
\begin{split}
 \eta&=C \cdot \delta_{(x^{(-1)},x^{(1)})}\otimes_s h_{i_1^{(-1)}}\otimes \cdots\otimes  h_{i_{n_{-1}}^{(-1)}}\otimes h_{i_1^{(1)}}\otimes \cdots \otimes h_{i_{n_{1}}^{(1)}}\\
&=\tilde C \cdot \delta_{(\tilde x^{(-1)},\tilde x^{(1)})}\otimes_s h_{\tilde i_1^{(-1)}}\otimes \cdots \otimes h_{\tilde i_{n_{-1}}^{(-1)}}\otimes h_{\tilde i_1^{(1)}}\otimes \cdots \otimes h_{\tilde i_{n_{1}}^{(1)}}
\end{split}
\end{equation} 
where for $b\in\IndexSet$, $x^{(b)}, \tilde x^{(b)}\in \SX_n$ and $n:=\max(n_{-1},n_1)$.
Then there is some $(\pi_{-1},\pi_1)\in S_{n_{-1}}\times S_{n_{1}}$ such that for each $b\in\IndexSet=\{-1,1\}$, $i_{\tilde s}=i_{\pi_{b}^{-1} s}$ for $s\in [n_{b}]$ and $\tilde x^{(b)}=(\iota_{n}^{n_b}\pi_{b})x^{(b)}$, where $\iota_{n}^{n_b}$ denotes the natural inclusion from $S_{n_b}\to S_n$.
\end{proposition}
Proposition \ref{prop:elemperm} invites the following.
\begin{corollary}
Let $\HH$ be an infinite-dimensional complex Hilbert space with orthonormal basis $\basis:=\{h_i\}_{i=1}^\infty$.
Let $\eta\in\FS_{V}(\HH)$ be a $\basis$-elementary vector.
Fix some expression for $\eta$ of the form \eqref{eqn:elemform}.
\begin{enumerate}
\item If $\eta\ne 0$ then the sets
\begin{equation}
\TensorIndices^{(b)}(\eta):=\{i^{(b)}_1,\ldots, i^{(b)}_{\bn_\eta(b)}\} \quad (b\in \IndexSet)\quad\text{and} \quad \TensorIndices(\eta):=\TensorIndices^{(-1)}(\eta)\cup \TensorIndices^{(1)}(\eta)
\end{equation}
do not depend on the choice of the expression for $\eta$ in the form \eqref{eqn:elemform}.
\item If $\eta\ne 0$ and $b\in\IndexSet$, the function $S_\eta^b:\TensorIndices(\eta)\to Q$ (with $Q$ as in Notation \ref{not:vk}) given by
\begin{equation}
 S_\eta(i_u^{(b)}):=x^{(b)}_u
\end{equation}
does not depend on the choice of the expression for $\eta$ in the form \eqref{eqn:elemform}.
\item If $\eta\ne 0$ and $b\in\IndexSet$ is such that $\bn_\eta(b)>\bn_\eta(-b)$ then the function
\begin{equation}
x_\eta:[\bn_\eta(-b)+1,\bn_\eta(b)]\to Q
\end{equation}
given by
\begin{equation}
x_\eta(k)=x^{(-b)}_k
\end{equation}
does not depend on the choice of the expression for $\eta$ in the form \eqref{eqn:elemform}.

\item Moreover,
\begin{equation}
D(\eta)=
\begin{cases}
0,&\text{if $\eta=0$}\\
C,&\text{if $\eta\ne0$ is expressed in the form \eqref{eqn:elemform}}
\end{cases}
\end{equation}
does not depend on the choice of the expression for $\eta$ in the form \eqref{eqn:elemform}.
\end{enumerate}
\end{corollary}

We can now characterize the effect of an annihilation operator on a $\basis$-elementary vector.
In doing so, it will be convenient to fix the Hilbert spaces $V_\bn$ and the transition maps $j_b$ and omit the superscripts $V$ and $j$.
The next two propositions follow immediately from the definition of the annihilation operators.
\begin{proposition}\label{prop:elemanndom}
If $\eta\in \FS_V(\HH)$ is a $\basis$-elementary vector and $b\in\IndexSet$ is such that $\bn_\eta(b)> \bn_\eta(-b)$ and $i\in\TensorIndices_{\eta}^{(b)}$, then $a_b(h_i)\eta$  is also a  $\basis$-elementary vector and
\begin{equation}
D(a_b(h_i)\eta)=\frac{\delta_{x_\eta(\bn_\eta(b)), S_\eta(i)}\cdot \mu( S_\eta(i))}{\bn_\eta(b)}\cdot  D(\eta). 
\end{equation}

Furthermore if $a_b(h_i)\eta\ne0$,  then the following hold:
\begin{enumerate}
\item $\TensorIndices^{(b)}(a_b(h_i)\eta)=\TensorIndices^{(b)}(\eta)\setminus\{i\}$;
\item $\TensorIndices^{(-b)}(a_b(h_i)\eta)=\TensorIndices^{(-b)}(\eta)$;
\item For $i'\in [\bn_\eta(-b)+1, \bn_\eta(b)-1]$, $x_{a_b(h_i)\eta}(\bn_\eta(b))=x_{\eta}(\bn_\eta(b))$;
\item For any $i'\in \TensorIndices(a_b(h_i)\eta)=\TensorIndices({\eta})\setminus\{i\}$, $S_{a_b(h_i)\eta}(i')= S_{\eta}(i')$.
\end{enumerate}
\end{proposition}
\begin{proposition}\label{prop:elemannsub} 
 Let $\eta\in\FS_V(\HH)$ be a nonzero $\basis$-elementary vector. Fix $i\in\TensorIndices_\eta$, and let $b\in\IndexSet$ be such that $\bn_\eta(b)\le \bn_\eta(-b)$.
 Then $a_b(h_i)\eta$ is a $\basis$-elementary vector and
\begin{equation}
 D(a_b(h_i)\eta)=\frac{1}{\bn_\eta(b)}\cdot D(\eta).
\end{equation}

 If $a_b(h_i)\eta\ne 0$ then the following also hold:
 \begin{enumerate}
 \item $\TensorIndices^{(b)}({a_b(h_i)\eta})=\TensorIndices^{(b)}(\eta)\setminus\{i\}$;
 \item $\TensorIndices^{(-b)}({a_b(h_i)\eta})=\TensorIndices^{(-b)}(\eta)$;
 \item For $i'\in [\bn_\eta(b)+1, \bn_\eta(b)]$, $x_{a_b(h_i)\eta}(\bn_\eta(b))=x_{\eta}(\bn_\eta(b))$;
 \item $x_{a_b(h_i)\eta}(\bn_\eta(b))=S_{\eta}(i)$;
 \item For any $i'\in \TensorIndices(a_b(h_i)\eta)=\TensorIndices({\eta})\setminus\{i\}$, $S_{a_b(h_i)\eta}(i')= S_{\eta}(i')$.
\end{enumerate}
\end{proposition}

A creation operator need not take a $\basis$-elementary vector to another $\basis$-elementary vector.
However, we can express a creation operator as a sum of operators which preserve the $\basis$-elementary property.
\begin{notation}
 For $z\in Q=\BN$, let $\epsilon_z:V_n\to V_{n+1}$ to be the linear map such that $\epsilon_z \delta_{(x,y)}=\delta_{((x,z),(y,z))}$.
 If $h\in\HH$ and $n_{-1}\ge n_{1}$ define an operator
\begin{equation}
a_{-1,z}^*(h_r): V_{n_{-1}}\otimes_s \HH^{\otimes n_{-1}}\otimes \HH^{\otimes n_{1}} \to V_{n_{-1}+1}\otimes_s \HH^{\otimes n_{-1}+1}\otimes \HH^{\otimes n_{1}}
\end{equation} 
  by
\begin{equation}
a_{-1,z}^*(h_r) \eta=\left(\bn_{\eta}(b)\epsilon_z\otimes r_{-1}(h)\right)\eta.
\end{equation}
Define $a_{1,z}^*(h_r)$ similarly for $n_{1}\ge n_{-1}$.
\end{notation}
\begin{remark}
In the case that $\sum \alpha_i=1$ (and thus $\beta_i=0$ and $\gamma=0$), we have
\begin{equation}
\sum_{z\in Q} a_{b,z}^*(h)\eta=a_{b}^*(h)\eta
\end{equation}
for a $\basis$-elementary vector $\eta\in \FS(V)$ with $\bn_{\eta}(b)>\bn_{\eta}(-b)$.
\end{remark}
The following two propositions follow directly from the definition of the respective operators.
\begin{proposition} \label{prop:elemcrdom}
If $\eta\in\FS_V(\HH)$ is a $\basis$-elementary vector, $b\in\IndexSet$ is such that $\bn_\eta(b)\ge \bn_\eta(-b)$, and $i\in\BN\setminus \TensorIndices(\eta)$ then for any $z\in Q$, the vector $a_{b,z}^*(h_i)\eta$ is also $\basis$-elementary.
Furthermore,  the following hold:
\begin{enumerate}
\item If $\eta\ne 0$ then $\TensorIndices^{(b)}({a_{b,z}^*(h_i)\eta)}=\TensorIndices^{(b)}({\eta})\cup \{i\}$;
\item If $\eta\ne 0$ then $\TensorIndices^{(-b)}({a_{-b,z}^*(h_i)\eta})=\TensorIndices^{(-b)}({\eta})$;
\item For any $i'\in \TensorIndices({\eta})$, $S_{a_{b,z}^*(h_i)\eta}(i')=S_{\eta}(i')$;
\item $S_{a_{b,z}^*(h_i)\eta}(i)=z$;
\item $D(a_{b,z}^*(h_i)\eta)=(\bn_{\eta}(b)+1)\cdot D(\eta)$.
\end{enumerate}
\end{proposition}
\begin{proposition} \label{prop:elemcrsub}
If $\eta\in \FS_V(\HH)$ is a $\basis$-elementary vector, $b\in\basis$ is such that $\bn_\eta(-b)>\bn_\eta(b)$, and $i\in\BN\setminus \TensorIndices_\eta$ then the vector $a_{b}^*(h_i)\eta$ is also $\basis$-elementary.
Furthermore, 
\begin{enumerate}
\item If $\eta\ne 0$ then $\TensorIndices^{(b)}({a_{b}^*(h_i)\eta})=\TensorIndices^{(b)}({\eta})\cup \{i\}$;
\item If $\eta\ne 0$ then $\TensorIndices^{(-b)}({a_{-b}^*(h_i)\eta})=\TensorIndices^{(-b)}({\eta})$;
\item For any $i'\in \TensorIndices({\eta})$, $S_{a_{b,z}^*(h_i)\eta}(i')=S_{\eta}(i')$;
\item If $\eta\ne 0$ then $S_{a_{b}^*(h_i)\eta}(i)= x_{\eta}(\bn_\eta(b))$.
\item $D(a_{b}^*(h_i)\eta)=(\bn_\eta(b)+1)D(\eta)$.
\end{enumerate}
\end{proposition}
We are now ready to prove the main theorem.
\begin{proof}[Proof of Theorem \ref{thm:sppf}]
We reiterate that we are focusing on the case $\sum{\alpha_i}=1$ for simplicity.
The key ideas of the more general case $\sum\alpha_i+\sum\beta_i\le 1$ are in this case, but this case is slightly more straightforward in that it allows us to work with discrete sums instead of integrals.
In this case, we can assume that $Q=\BN$.

Let $n=|\VP|$ and let $\HH$ be an infinite-dimensional complex Hilbert space with an orthonormal basis $\basis:=\{h_i\}_{i=1}^\infty$.
By Theorem \ref{thm:pf}, we can compute $ \Pt_{V,j}(\VP,c)$ as
\begin{equation}
 \Pt_{V,j}(\VP,c)=\left<a_{c(1)}^{e_1}(h_{k_1})\cdots a_{c(n)}^{e_n}(h_{k_n}) \left(\bone_0\otimes_s \Omega\right),\bone_0\otimes_s \Omega \right>,\label{eqn:sppfdef}
\end{equation}
where  $i$ is an element of the $k_i$-th pair of $\VP$, and $e_i=2$ if $i\in R_\VP$ and $e_i=1$ if $i\in L_\VP$.
For each $k\in[2m]$ define
\begin{equation}
 A_{i}:=
\begin{cases}
 a_{c(i)}^*(h_{k_i}),&\text{if $i\in R_\VP$}\\
 a_{c(i)}(h_{k_i}),&\text{if $i\in L_\VP$},
\end{cases}
\end{equation}
so that
\begin{equation}
 \Pt_{V,j}(\VP,c)=\left<\left(A_{1}\cdots A_{2m} \right)\bone_0\otimes_s \Omega,\bone_0\otimes_s \Omega \right>,\label{eqn:sppfdefcpt}
\end{equation}

Denote by $\Lambda_{\VP,c}$ the space of functions $\lambda:R_{\VP,c}^D\to Q$.
For $\lambda\in\Lambda_{\VP,c}$ and $i\in [2m]$, define
\begin{equation}
 A_{\lambda,i}:=
\begin{cases}
a_{c(i),\lambda(i)}^*(h_{k_i}),&\text{if $i\in R_{\VP,c}^D$}\\
A_{i},&\text{otherwise}.
\end{cases}
\end{equation}
We define
\begin{equation}
 A_{\lambda}^{(k)}:= A_{\lambda,k}\cdots A_{\lambda,2m} 
\end{equation}
for $1\le k\le 2m$.
For convenience, we also set $A_{\lambda}^{(2m+1)}=1$.
It follows immediately from the definitions and \eqref{eqn:sppfdefcpt} that 
\begin{equation}
  \Pt_{V,j}(\VP,c)=\sum_{\lambda\in\Lambda_{\VP,c}}\left<A_\lambda^{(1)}\bone_0\otimes_s \Omega,\bone_0\otimes_s \Omega \right>
\end{equation}
Also define
\begin{equation}
\eta_{\lambda,k}:=A_{\lambda}^{(k)}\bone_0\otimes_s \Omega.
\end{equation}
It can be seen from the definition of $\vl_{\VP,c}^{b}$ that if $\eta_{\lambda,k}\ne 0$ then
\begin{equation}\label{eqn:nvl}
\bn_{\eta_{\lambda,k}}(b)=
\begin{cases}
 \vl_{\VP,c}^{(b)}(k)-1,&\text{if $k\in L_\VP$ and $c(k)=b$},\\
 \vl_{\VP,c}^{(b)}(k),&\text{otherwise}.
\end{cases}
\end{equation}
for $b\in\IndexSet$.

It follows from Propositions \ref{prop:elemanndom}, \ref{prop:elemannsub}, \ref{prop:elemcrdom} and \ref{prop:elemcrsub} that $\eta_{\lambda,k}$ is $\basis$-elementary for $1\le k\le 2m+1$.
Since $\eta_{\lambda,1}\in \BC \bone_0\otimes_s \Omega$, it is determined by the constant $D(\eta_{\lambda,1})$.
For $k\in[2m]$, define
\begin{equation}
B_{\lambda,k}:=
\begin{cases}
\delta_{x_{\eta_{\lambda,k+1}}(\bn_{\eta_{\lambda,k+1}}(b)), S_{\eta_{\lambda,k+1}}(i)}\cdot \mu( S_{\eta_{\lambda,k+1}}(i))&\text{if $k\in L_{\VP,c}^D$} \\
1,&\text{otherwise},
\end{cases}
\end{equation}
so that 
\begin{equation}
D(\eta_{\lambda,1})=\prod_{k=1}^{2m} B_{\lambda,k}=\prod_{k\in L_{\VP,c}^D} B_{\lambda,k}.
\end{equation}

For a maximal increasing path $P$ of $G_{\VP,c}$ from $l_s$ to $r_t$ with $l_s<r_t$, define a function  $H_{P,\lambda}:[l_s+1,r_t]\to Q$ as follows.
If $u\in [l_s+1,r_t]$, denote by $f_u=(a_u,z_u)$ the unique arc along the path $P$ such that $a_u< u\le z_u$.
Then we define
\begin{equation}
 H_{P,\lambda}(u):=
\begin{cases}
x_{\eta_{\lambda,u}}(\vls_{\VP,c}(a_u)),&\text{if $f_{u}=(a_u,z_u)\in \bar\VP^c$}\\
S_{\eta_{\lambda,u}}(i),&\text{if $f_{u}=(l_i,r_i)\in \VP$}.
\end{cases}
\end{equation}
It can be verified using \eqref{eqn:nvl} that if $b\in\IndexSet$ is such that $\vl^{c(b)}(a_u)>\vl^{c(-b)}(a_u)$ and $f_u\in\bar F_{\VP,c}$ then $\bn_{\eta_{\lambda,u}}(b)<\vls_{\VP,c}(a_u)\le \bn_{\eta_{\lambda,u}}(-b)$, so that $\vls_{\VP,c}(a_u)$ is indeed in the domain of $x_{\eta_{\lambda,u}}$.

We will show that $H_{P,\lambda}(u)=\lambda(l_s)$ for all $u\in [l_s+1,r_t]$.
By definition, 
\begin{equation}
\eta_{\lambda,r_t}= A_{\lambda,r_t}\eta_{\lambda,r_t+1}=a_{c(r_t),\lambda(r_t)}^*(h_{t})\eta_{\lambda,r_t+1}.
\end{equation}
As an endpoint of a maximal monotone path, $r_t\in \dom_{\VP,c}$,  
and  $\vl_{\VP,c}^{c(r_t)}(r_t)>\vl_{\VP,c}^{-c(r_t)}(r_t)$ implies that  $\bn_{\eta_{\lambda,r_t-1}}(c(r_t))> \bn_{\eta_{\lambda,r_t+1}}(-c(r_t))$.
If $f_{r_t}\in \bar F_{\VP,c}$ then by Proposition \ref{prop:elemcrdom},
\begin{equation}
 H_{P,\lambda}(r_t)=x_{a_{c(r_t),\lambda(r_t)}^*(h_{t})\eta_{\lambda,r_t+1}}(\vls_{\VP,c}(c(r_t)))=\lambda(r_t).
\end{equation}
If, on the other hand, $f_{u}=(l_t,r_t)\in \VP$ then
\begin{equation}
 H_{P,\lambda}(r_t)=S_{a_{c(r_t),\lambda(r_t)}^*(h_{t})\eta_{\lambda,r_t+1}}(t)=\lambda(r_t),
\end{equation}
where we have again applied Proposition \ref{prop:elemcrdom}. 

To show that $H_{P,\lambda}(u)=\lambda(r_t)$ for all $u$, it will now suffice to show that $H_{P,\lambda}(u)=H_{P,\lambda}(u+1)$ for $u\in [l_s+1,r_t-1]$.
For $u$ such that $f_u=f_{u+1}$ this is an immediate consequence of the definition of $H_{P,\lambda}(u)$.
Otherwise, $u$ is a vertex along the path $P$ and $u\in\subo_{\VP,c}$.
If $u\in L_\VP$, say $u=l_i$, then $(a_u,z_u)=(u,z_u)\in \bar F_{\VP,c}$, thus 
\begin{equation}
\begin{split}
 H_{P,\lambda}(u)&=x_{\eta_{\lambda,u}}(\vls_{\VP,c}(c(a_u)))\\
&=x_{a_{c(u)}(h_{i})\eta_{\lambda,u+1}}(\vls_{\VP,c}(c(a_u)))\\
&=S_{\eta_{\lambda,u+1}}(i)\\
&=H_{P,\lambda}(u+1).
\end{split}
\end{equation}
Here we have made use of Proposition \ref{prop:elemannsub}.

If instead $u\in R_\VP$, say $u=r_i$, then  $(a_u,z_u)=(l_i,r_i)\in \VP$, so
\begin{equation}
\begin{split}
 H_{P,\lambda}(u)&=S_{\eta_{\lambda,u}}(i)\\
&=S_{a_{c(u)}^*(h_{i})\eta_{\lambda,u+1}}(i)\\
&=x_{\eta_{\lambda,u+1}}(\bn_{\eta_{\lambda,u+1}}(c(u)))\\
&=x_{\eta_{\lambda,u+1}}(\vls_{\VP,c}(r_i ))\\
&=H_{P,\lambda}(u+1),
\end{split}
\end{equation}
where we have used Proposition \ref{prop:elemcrsub}.

Thus, $H_{P,\lambda}(u)=\lambda(r_t)$ for all $u\in [l_s+1,r_t]$.
In particular, taking $u=l_s+1$ shows that $x_{\eta_{\lambda, l_s+1}}(\vls_{\VP,c}(l_s))=\lambda(r_t)$ if the initial arc of $P$ is in $\bar F_{\VP,c}$ and $S_{\eta_{\lambda, l_s+1}}(i)=\lambda(r_t)$ if the initial arc is $(l_i,r_i)\in\VP$.
By a similar argument, if $P'$ is a maximal decreasing path from some $r_{t'}\in R_{\VP,c}^D$ to $l_s$ then $x_{\eta_{\lambda, l_s+1}}(\vls_{\VP,c}(l_s))=\lambda(r_t)$ if the final arc of $P'$ is in $\bar F_{\VP,c}$ and $S_{\eta_{\lambda, l_s+1}}(i')=\lambda(r_t)$ if the final arc of $P'$ is $(l_{i'},r_{i'})\in\VP$.
Thus,
\begin{equation}
 B_{\lambda, l_s}=\delta_{x_{\eta_{\lambda,k+1}}(\bn_{\eta_{\lambda,k+1}}(b)), S_{\eta_{\lambda,k+1}}(i)}\cdot \mu( S_{\eta_{\lambda,k+1}}(i))=\delta_{\lambda(r_t),\lambda(r_{t'})}\cdot \mu(\lambda(r_t)).
\end{equation}
Therefore, if $\mathfrak{A}:L_{\VP,c}^D\to R_{\VP,c}^D$ is the function taking $l\in L_{\VP,c}^D$ to the starting point of the maximal  decreasing path terminating at $l$ and $\mathfrak{D}:L_{\VP,c}^D\to R_{\VP,c}^D$ is the function $l$ to the endpoint of the maximal  increasing path starting at $l$ then
\begin{equation}
D(\eta_{\lambda,1})=\prod_{k\in L_{\VP,c}^D} B_{\lambda,k}=\prod_{l\in L_{\VP,c}^D} \delta_{\lambda(\mathfrak{D}(l)),\lambda(\mathfrak{A}(l))}\cdot \mu(\mathfrak{D}(l)).
\end{equation}
This means that $D(\eta_{\lambda,1})=0$ unless $\lambda(r)=\lambda(r')$ whenever there is some $l\in L_{\VP,c}^D$ such that there are maximal monotone paths from $l$ to $r$ and $r$ to $l'$.
But this condition holds only if $\lambda(r)=\lambda(r')$ whenever $r$ and $r'$ lie along the same cycle of $G_{\VP,c}$. Denote by $\Lambda^C_{\VP,c}$ the set of functions in $\Lambda_{\VP,c}$ satisfying this condition.
For $\lambda\in\Lambda^C_{\VP,c}$ and a cycle $K$ of $G_{\VP,c}$ we write $\lambda(K)$ for the common value $\lambda(r)$ for any $r\in R_{\VP,c}^D$ along the cycle $K$.

 Denoting the set of cycles of $G_{\VP,c}$ by $C(G_{\VP,c})$ and the number of maximal increasing paths of a cycle $K$ by $M(K)$ we deduce that
\begin{equation}
D(\eta_{\lambda,1})=
\begin{cases}
\prod_{K\in C(G_{\VP,c})} \mu(\lambda(K))^{M(K)},&\text{if $\lambda\in \Lambda^C_{\VP,c}$},\\
0,&\text{otherwise}.
\end{cases}
\end{equation}
Finally,
\begin{equation}
\begin{split}
 \Pt_{\alpha,\beta}(\VP,c)&=\sum_{\lambda\in\Lambda_{\VP,c}}\left<A_\lambda^{(1)}\bone_0\otimes_s \Omega,\bone_0\otimes_s \Omega \right>\\
&=\sum_{\lambda\in\Lambda_{\VP,c}}D(\eta_{\lambda,1})\\
&=\sum_{\lambda\in\Lambda^C_{\VP,c}}\prod_{K\in C(G_{\VP,c})} \mu(\lambda(K))^{M(K)}\\
&=\prod_{m\ge 2}\left(\sum_{i=1}^\infty\alpha_i^{m}\right)^{\gamma_m(G_{\VP,c})}.
\end{split}
\end{equation}
\end{proof}

\primarydivision{A special case of the generalized Brownian motions associated to spherical representations of $(S_\infty\times S_\infty, S_\infty)$}\label{pri:Nex}

In this \lprimarydivname, we specialize the investigation begun in \primarydivname\ \ref{pri:spherical} to a countable class of Thoma parameters which were also considered in \cite{BG}.
Namely, for $N\in \BZ\setminus\{0\}$ we will consider the spherical function $\varphi_N$ of the Gelfand pair $(S_\infty\times S_\infty,S_\infty)$ arising from the Thoma parameters with
\begin{equation}
\begin{split}
\alpha_n=
\begin{cases}
1/N,&\text{$N>0$ and $1\le n\le N$};\\
0,&\text{otherwise}.
\end{cases}
\quad\text{and}\quad
\beta_n=
\begin{cases}
1/N,&\text{$N<0$ and $1\le n\le |N|$};\\
0,&\text{otherwise}.
\end{cases}
\end{split}
\end{equation}
The character $\varphi_N$ on $S_\infty$ is given by
\begin{equation}
\varphi_N(\pi)=\left(\frac{1}{N}\right)^{m-\gamma^{(m)}\left(\pi\right)}
\end{equation}
where $m$ is large enough so that $\sigma(k)=k$ for $k>m$ and $\gamma^{(m)}\left(\sigma\right)$ is the number of cycles in the permutation $\sigma\in S_\infty$ when $\sigma$  is considered as an element of $S_m$.
Although $\gamma^{(m)}\left(\sigma\right)$ depends on the choice of $m$, the quantity  $m-\gamma^{(m)}\left(\sigma\right)$ does not.

We denote by $\psi_N$ the associated spherical function on the Gelfand pair $(S_\infty\times S_\infty, S_\infty)$.
That is,
\begin{equation}
\psi_N(\pi_{-1},\pi_1)=\varphi_N(\pi_{1}^{-1}\pi_{-1})=\left(\frac{1}{N}\right)^{m-\gamma^{(m)}\left(\pi_{1}^{-1}\pi_{-1}\right)},
\end{equation}
where $(\pi_{-1}, \pi_1)\in S_\infty\times S_\infty$.

The function on $\{-1,1\}$-indexed pair partitions associated to $\psi_N$ by Theorem \ref{thm:sppf} is given by
\begin{equation}
\Pt_N(\VP,c)=\left(\frac{1}{N}\right)^{m(G_{\VP,c})-\gamma(G_{\VP,c})},
\end{equation}
where $m(G_{\VP,c})$ and $\gamma(G_{\VP,c})$ denote the number of maximal increasing paths and number of cycles of  the graph $G_{\VP,c}$ defined in \primarydivname\ \ref{pri:spherical}, respectively.

For a complex Hilbert space $\HH$, the function $\Pt_N$ gives rise to a Fock state $ \rho_{N}$ on the algebra $\COAI(\HH)$.
We denote by $\FS^\IndexSet_{N}(\HH)$, $\Omega_{N}$, and $\COAI_{N}(\HH)$ the Hilbert space, distinguished cyclic vector, and algebra of operators of the GNS construction for this pair.
We will see that for $N<0$, the field operators on $\FS_{N}(\HH)$ are bounded operators which generate a von Neumann algebra containing the projection onto vacuum vector.

\begin{notation}
Fix an integer $N\ne 0$ and an infinite-dimensional complex Hilbert space $\HH$ with orthonormal basis $\basis:=\{h_n:n\in\BN\}$.
Denote by $a_{b,i}^*$ the creation operator $a_{\Pt_N,b}^*(h_i)$ and by $a_{b,i}$ the annihilation operator $a_{\Pt_N,b}(h_i)$.
As we have done before, we will write $a_{b,i}^e$ for either a creation ($e=2$) or annihilation operator ($e=1$) and let $\omega_{b,i}=a_{b,i}+a_{b,i}^*$.
Let $\Gamma_N^\IndexSet(\HH)$ be the von Neumann algebra generated by the spectral projections of the $\omega_{b,i}$ ($b\in\IndexSet$, $i\in\BN$).

Let $\SP(\HH, \basis)$ be the set of finite words in the $a_{b,i}^e$.
Each word in $\SP(\HH, \basis)$ can be considered a (possibly unbounded) operator on $\FS_N^\IndexSet(\HH)$ simply by regarding it as a product of the creation and annihilation operators that comprise the word.
The set $\SP(\HH,\basis)$ inherits the involution $*$ from $\COAI(\HH)$.
For $A\in \SP(\HH, \basis)$, $b\in\IndexSet$ and $i\in\BN$, let $\nc_{b,i}(A)$ be the number of occurrences of the creation operator $a_{b,i}^*$ in the word $A$ and $\na_{b,i}(A)$ the number of occurrences of the annihilation operator $a_{b,i}$ in the word $A$.
For $b\in\IndexSet$, define $\bw^{A}_{b}:\BN\to\BZ$ by
\begin{equation}
\bw^{A}_{b}(n)=\nc_{b,n}(A)-\na_{b,n}(A).
\end{equation}

Given a function $\bw_b:\BN\to\BZ$ for each $b\in\IndexSet$ taking only finitely many nonzero values, define
\begin{equation}
\SP_\bw(\HH,\basis):=\left\{A\in\SP(\HH,\basis):\bw^{A}=\bw\right\}.
\end{equation}
Denote by $H_\bw(\HH,\basis)$ the space
\begin{equation}
H_\bw(\HH,\basis):=\linspan\{A\Omega_N: A\in\SP_\bw(\HH,\basis)\}.
\end{equation}
For a function $\bw_b:\BN\to\BZ$ which is $0$ at all but finitely many points, we define $|\bw_b|=\sum_{n=1}^\infty \bw_b(n)$.
\end{notation}
\begin{remark}
If $\bw_b(n)<0$ for some $b\in\IndexSet$ and some $n\in\BN$ then $H_\bw(\HH,\basis)=0$.
\end{remark}
\begin{definition}\label{def:comp}
Suppose that $A\in\SP(\HH,\basis)$ and the terms in the product $A$ are indexed by some ordered set $S$,
\begin{equation}
A=\prod_{k\in S} a_{b_k,i_k}^{e_k}.
\end{equation}
We will say that $(\VP,c)\in \IPP(S)$ is compatible with $A$ if 
\begin{equation}
\prod_{(l,r)\in\VP}\delta_{e_l,1}\delta_{e_r,2}\delta_{c(l,r), b_l}\delta_{c(l,r), b_r}\delta_{i_l,i_r}=1.
\end{equation}
Denote by $\Comp(A)$ the set of all $(\VP,c)\in \IPP(S)$ which are compatible with $A$.
\end{definition}
\begin{remark}
The motivation for Definition \ref{def:comp} is that \eqref{eqn:fock}, for $A\in\SP(\HH,\basis)$,
\begin{equation}\label{eqn:fockcompat}
 \rho_N(A)=\sum_{(\VP,c)\in\Comp(A) }\Pt_{N}(\VP,c).
\end{equation}
\end{remark}

\begin{remark}
The condition $(\VP,c)\in \Comp(A)$ uniquely determines $(\bar{\VP}^{(c)},\bar c)$. 
This is because once we know that $(\VP,c)\in \Comp(A)$, we can immediately determine the color function $c$ and which points are left points of $\VP$ and which are right points of $\VP$.
This, in turn, determines the index function $\bar{\VP}^{(c)}$ and the involution $\ZB_{\VP,c}$, which gives $\bar{\VP}^{(c)}$.
\end{remark}
\begin{notation}
For $T\subseteq S$, we denote by $A|_T$ the product of those elements in the word $A$ which are indexed by elements of $T$.
\end{notation}
The next proposition is an immediate consequence of the definitions.
\begin{proposition}\label{prop:domcomp}
Suppose that $A\in\SP_\bw(\HH,\basis)$ is given by
\begin{equation}
A=\prod_{k=1}^r a_{b_k,i_k}^{e_k}.
\end{equation}
 and let $(\VP,c)\in\Comp(A)$.
 For each $s\in [r]$ denote by $A_s$ the product $A_s=\prod_{k={s+1}}^r a_{b_k,i_k}^{e_k}$. 
 Then $s\in\dom_{\VP,c}$ if and only if one of the following holds:
 \begin{enumerate}
\item   $s\in R_\VP$ and $\left|\bw^{A_{s}}_{c(s)}\right| \ge \left|\bw^{A_{s}}_{-c(s)} \right|$;
\item $s\in L_\VP$ and $\left|\bw^{A_{s}}_{c(s)}\right| > \left|\bw^{A_{s}}_{-c(s)} \right|$.
\end{enumerate}
\end{proposition}

The definition \eqref{eqn:fock} of the Fock state $\rho_N$ and the definition of $\bw^A_b$ give the following.
\begin{proposition}\label{prop:orthog}
Suppose that $\bw_{-1}, \bw_1,\bw_{-1}',\bw_1':\BN\to\BZ$ are zero except at finitely many points.
The spaces $H_\bw(\HH,\basis)$ and $H_{\bw'}(\HH,\basis)$ are orthogonal unless $\bw_b=\bw_b'$ for $b\in\IndexSet$.
\end{proposition}
This enables us to make the following definition.
\begin{definition}
Let $\NO_{b,n}$ be the operator defined on the dense subspace $\oplus_{\bw}H_{\bw}(\HH,\basis)$ of $\FS^\IndexSet_N(\HH)$ by linear extension of
\begin{equation}
\eta\mapsto \bw_b(n)\cdot \eta\quad\text{for }\eta\in H_{\bw}(\HH,\basis).
\end{equation}
\end{definition}

The next proposition is an exclusion principle analogous to that proven in \cite{BG}.
\begin{lemma}\label{lem:excl}
If $N<0$ and there is some $b\in\IndexSet$ and some $n\in\BN$ such that if $\bw_{b}(n)>|N|$, then  $H_\bw(\HH, \basis)=0$.
\end{lemma}
\begin{corollary}\label{cor:numbounded}
If $N<0$ then $\NO_{b,n}$ is bounded for all $b\in\IndexSet$ and $n\in\BN$.
Moreover, $\|\NO_{b,n}\|=|N|$.
\end{corollary}
Our proof of Lemma \ref{lem:excl} will make use of some basic combinatorics, which we now recall.

\begin{notation}\label{not:stirling}
Denote by $|s(n, k)|$ the number of permutations in the symmetric group $S_n$ which can be written as the product of $k$ disjoint cycles.
\end{notation}
The numbers $|s(n, k)|$ are known as the unsigned Stirling number of the first kind.
It is well-known (c.f. \cite{Stanley}) that the unsigned Stirling numbers of the first kind satisfy the relation
\begin{equation}\label{eqn:stirid}
x(x+1)\cdots (x+n-1)=\sum_{k=0}^n |s(n,k)| x^k.
\end{equation}
\begin{proposition}\label{prop:sumNpow}
If $N\in\BN$ and $N<0$ then
\begin{equation}
\sum_{\sigma\in S_{|N|+1}}N^{c(\sigma)}=0
\end{equation} 
\end{proposition}
\begin{proof}
This follows immediately from \eqref{eqn:stirid}:
\begin{equation}
\sum_{\sigma\in S_{|N|+1}}N^{c(\sigma)}=\sum_{k=0}^{|N|+1}|s(|N|+1,k)| N^{k}=N(N+1)\cdots (N+(|N|+1)-1)=0.
\end{equation}
\end{proof}
To prove Lemma \ref{lem:excl}, we will need the following proposition, which is proven by applying the definitions.
\begin{proposition}\label{prop:match}
Suppose that $A, B\in \SP(\HH,\basis)$ are words of length $\ell_A$ and $\ell_B$, respectively.
Assume that $A$ can be expressed as a product
\begin{equation}
A=a_{b_1, i_1}^{e_1}\cdots a_{b_{\ell_A}, i_{\ell_A}}^{e_{\ell_A}}
\end{equation}
and that $B^*$ can be expressed as
\begin{equation}
B^*=a_{b_{-\ell_B}, i_{-\ell_B}}^{e_{-\ell_B}}\cdots a_{b_{-1}, i_{-1}^{e_{-1}}}
\end{equation}
If $(\VP,c)\in \Comp(B^*A)$ then for any $i\in\BN$, there is a unique subset $\MS^{+,b,i}_{A,B}(\VP,c)\subset [\ell_A]$ and a unique subset $\MS^{-,b,i}_{A,B}(\VP,c)\subset -[\ell_B]$ such that the following conditions are satisfied:
\begin{enumerate}
\item $|\MS^{+,b,i}_{A,B}(\VP,c)|=|\MS^{-,b,i}_{A,B}(\VP,c)|=\bw_b^A(i)$;\label{itm:matchcard}
\item If $k\in \MS^{+,b,i}_{A,B}(\VP,c)$ then $e_k=2$, $c(k)=b$, and $i_k=i$\label{itm:crcomp};
\item If $k\in \MS^{-,b,i}_{A,B}(\VP,c)$ then $e_k=1$, $c(k)=b$, and $i_k=i$\label{itm:ancomp};
\item If $k\in \MS^{+,b,i}_{A,B}(\VP,c)$ then $\pi_{\VP}(k)\in \MS^{-,b,i}_{A,B}(\VP,c)$\label{itm:crpair};
\item If $k\in \MS^{-,b,i}_{A,B}(\VP,c)$ then $\pi_{\VP}(k)\in \MS^{+,b,i}_{A,B}(\VP,c)$\label{itm:anpair}.
\end{enumerate}
Here $\pi_\VP$ is the permutation obtained by regarding the pairs of $\VP$ as transpositions.
\end{proposition}
\begin{remark}
The sets $\MS^{+,b,i}_{A,B}(\VP,c)$ and $\MS^{-,b,i}_{A,B}(\VP,c)$ may depend on the choice of $(\VP,c)\in\Comp(B^*A)$, but they are uniquely determined by this choice.
\end{remark}
\begin{proof}[Proof of Lemma \ref{lem:excl}]
It will suffice to consider the case $\bw_{b}(i)=|N|+1$, and we will further assume that $b=1$.
Consider a word $A\in\SP(\HH,\basis)$ containing $r\ge |N|+1$ creation operators $a_{b,i}^*$ and $r-|N|-1$ annihilation operators $a_{b,i}$, say
\begin{equation}
A=a_{b_1, i_1}^{e_1}\cdots a_{b_s, i_s}^{e_s}
\end{equation}
We need to show that ${\rho}_{N}(A^*A)=0$.
It will be convenient to view the word $A^*A$ as a product of operators whose terms are indexed by the set $J:=\{-s,-s+1,\ldots,-1,1,\ldots,s-1,s\}$.
Thus $b_k=b_{-k}$, $e_k\ne e_{-k}$, and $i_k=i_{-k}$ for $k\in J$.
By \eqref{eqn:fockcompat},
\begin{equation}\label{eqn:exclfock}
 \rho_N(A^*A)=\sum_{(\VP,c)\in\Comp(A^*A)} \Pt_N(\VP,c)
\end{equation}

For $(\VP,C)\in \Comp(A^*A)$, consider the sets $\MS^{+,b,i}_{A,A}(\VP,c)$ and $\MS^{-,b,i}_{A,A}(\VP,c)$ provided by Proposition \ref{prop:match}, and denote these simply by $\MS^+(\VP,c)$ and $\MS^-(\VP,c)$.
These sets have cardinality $\bw^A_b(i)=N+1$.
The $\IndexSet$-indexed pair partition $(\VP,c)$ can be  seen as a pair partition $\WP_{\VP,c}\in\Comp(A^*A|_{J\setminus (\MS^+(\VP,c)\cup \MS^-(\VP,c))})$ together with a bijection $\iota_{\VP,c}: \MS^+(\VP,c)\to \MS^-(\VP,c)$.

It will suffice to show that for any $F^+,F^-\subset J$ with $|F^+|=|F^-|=|N|+1$ and any $\WP\in\Comp(A^*A|_{J\setminus (F^+\cup F^-)})$,
\begin{equation}
\sum_{\substack{(\VP,c)\in\Comp(A^*A)\\ \WP_{\VP,c}=\WP}}\Pt_N(\VP,c)=0.\label{eqn:sumbij}
\end{equation}
If the sum is empty, there is nothing to show.
Otherwise, define a directed graph $\hat G_{F^-, F^+}$ whose vertices are the elements of the index set $J$, and whose arc set is
\begin{equation}
\left\{(k,k')^{c_{A^*A}(k)}:(k,k')\in\WP\right\}\cup \left\{(k,k')^{\bar c_{A^*A}(k,k')}:(k,k')\in \bar \VP_{A^*A}\right\},
\end{equation}
where $(k,k')^{(b)}$ is as in Notation \ref{not:inv}.
A vertex $k\in J\setminus (F^+\cup F^-)$ of $\hat G_{F^-, F^+}$ is the start point and end point of exactly one arc.
A vertex $k\in F^+$ is the start point of $1$ arc and is not the end point of any arc.
A vertex $k\in F^-$ is the end point of $1$ arc and is not the start point of any arc.
Therefore, each vertex $k\in F^+$ of $\hat G_{F^-, F^+}$ is the starting point of a maximal path which ends at some vertex $k'\in F^-$.
This gives a bijection $\epsilon:F^+\to F^-$.
(To clarify these notions, we consider a specific case, including diagrams in Example \ref{ex:excl}.)

A term in the sum in \eqref{eqn:sumbij} can be characterized by the bijection $\iota_{\VP,c}: F^+\to F^-$.
The graph $G_{\VP,c}$ can be formed from $\hat G_{F^-, F^+}$ by adding the edges arising from $\iota_{\VP,c}$.
The number of maximal increasing paths $m(G_{\VP,c})$ does not depend on the choice of $\iota_{\VP,c}$, and we denote the common value by $m$. 
Furthermore $\gamma(G_{\VP,c})=\gamma(\hat F)+\gamma(\iota_{\VP,c}\epsilon^{-1})$ where $\gamma(\iota_{\VP,c}\epsilon^{-1})$ is the number of cycles of $\iota_{\VP,c}\epsilon^{-1}$ as a permutation on $F^-$.
As $\iota_{\VP,c}$ ranges over all bijections $F^+\to F^-$, the permutation $\iota_{\VP,c}\epsilon^{-1}$ ranges over the symmetric group, whence
\begin{equation}
\begin{split}
\sum_{\substack{(\VP,c)\in\Comp(A^*A)\\ \WP_{\VP,c}=\WP \\ G^\pm(\VP,c)=F^\pm}}\Pt_N(\VP,c)&=\sum_{\sigma\in S_{N+1}}\left(\frac{1}{N}\right)^{m-(\gamma(\hat F)+\gamma(\sigma))}\\
&=\left(\frac{1}{N}\right)^{m-\gamma(\hat F)}\sum_{\sigma\in S_{N+1}}N^{\gamma(\sigma)}\\
&=0.
\end{split}
\end{equation}
\end{proof}

\begin{example}\label{ex:excl}
We consider a simple example, with a diagram, to clarify some of the ideas in the proof of Lemma \ref{lem:excl}.
Let $N=-1$, let $a=a_{1,1}$ and $a^*=a_{1,1}^*$ and define $A:=a^*aa^*a^*$ so that $\bw^A(1)=3-1=2>|N|=1$.
If $(\VP,c)\in\Comp(A^*A)$ then the sets $\MS_{A,A}^{+,1,1}(\VP,c)$ and $\MS_{A,A}^{-,1,1}(\VP,c)$ have cardinality $\bw^A(1)=2$.
As usual, we refer to these sets by $\MS^+(\VP,c)$ and $\MS^-(\VP,c)$.
Indexing the product on $[-4,4]\setminus\{0\}$, we will consider the $(\VP,c)\in\Comp(A^*A)$ having $\MS^+(\VP,c)=\{1,4\}$ and $\MS^-(\VP,c)=\{-3,-1\}$.
There are two such $(\VP,c)$, with pair partitions 
\begin{equation}
\VP_1=\{(-4, -2), (-3,-1),(-1,4), (2,3)\}\quad\text{and}\quad \VP_2=\{(-4, -2), (-3,4),(-1,1), (2,3)\}
\end{equation}
and with color functions $c_1$ and $c_2$ defined to be $1$ on all pairs of their respective pair partitions.

The graph $\hat G_{\{-3,-1\},\{1,4\}}$ is depicted in Figure \ref{fig:excl}.
This graph, whose name we abbreviate by $\hat G$, can be completed to either of the graphs $G_{(\VP_1,c_1)}$ or $G_{(\VP_2,c_2)}$ by adding the appropriate arcs.
The former arises from the map $\iota_1: \{1,4\}\to \{-3,-1\}$ given by $1\mapsto -3$ and $4\mapsto-1$ and the latter arises from $\iota_2$ with $4\mapsto -3$ and $1\mapsto-1$.
Following the maximal paths of the graph, one sees that the bijection $\epsilon$ is given by $4\mapsto -3$ and $1\mapsto -1$.
Thus, the bijection $\iota_2\epsilon^{-1}$ is the identity permutation on  $\{-3,-1\}$ and $\iota_1\epsilon^{-1}$ is the $2$-cycle on the set $\{-3, -1\}$.
Both of the graphs $G_{\VP_1,c_1}$ and $G_{\VP_2,c_2}$ have $4$ maximal increasing paths, and these graphs have $2$ and $3$ cycles, respectively.
The graph $\hat G$ has exactly $1$ cycle.
\begin{figure}
 \begin{tikzpicture}[shorten >=1pt,->]
  \tikzstyle{vertex}=[circle,fill=black!25,minimum size=14pt,inner sep=2pt]
  \foreach \x in {-4,-3,-2,-1,1,2,3,4} 
    \node[vertex] (\x) at (\x,0) {\x};
  \draw [->] (-4) to[out=45,in=135] (-2);
    \draw [->] (2) to[out=45,in=135] (3);
        \draw [->] (3) to[out=-135,in=-45] (2);
        \draw [->] (4) to[out=-135,in=-45] (-4);
        \draw [->] (1) to[out=-135,in=-45] (-1);
        \draw [->] (-2) to[out=-135,in=-45] (-3);
\end{tikzpicture}
\caption[An example of the graph used in the proof of Lemma \ref{lem:excl}]{The directed graph $\hat G_{\{-1,-3\},\{1,4\}}$ for the word considered in Example \ref{ex:excl}.}
\label{fig:excl}
\end{figure}
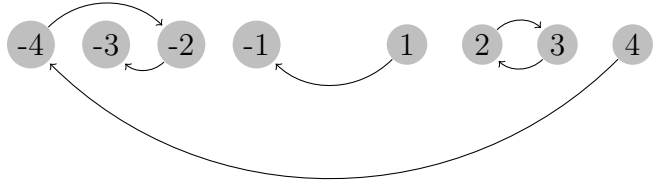
\end{example}

The following is a partial analog of Lemma 5.1 of \cite{BG}.
\begin{proposition}\label{prop:commutation}
Suppose that $A\in \SP_{\bw}(\HH,\basis)$ and $b\in\IndexSet$ with $\left|\bw_b\right|\ge \left|\bw_{-b}\right|$.
Then for $i\in\BN$,
\begin{equation}
a_{b,i}a_{b,i}^*A\Omega_N=\left(1+\frac{1}{N}\NO_{b,i}\right) A\Omega_N.
\end{equation}
\end{proposition}
\begin{proof}
It will suffice to show that if $B\in \SP_{\bw}(\HH,\basis)$ then
\begin{equation}\label{eqn:comstate}
\rho_N(B^*a_{b,i}a_{b,i}^*A)=\left(1+\frac{\bw_b^A(i)}{N}\right)\rho_N(B^*A).
\end{equation}
To save space, we define $X:=B^*a_{b,i}a_{b,i}^*A$.
We will assume that $b=1$ as the case $b=-1$ is similar.
Let $\ell_A$ be the length of the word $A$ and $\ell_B$ the length of the word $B$, so that $X$ is a word of length $\ell_A+\ell_B+2$.
It will be convenient to choose $J=[\ell_A+\ell_B+2]$ as our index set for the product $X=B^*a_{b,i}a_{b,i}^*A$ and $K=J\setminus\{\ell_B+1,\ell_B+2\}$ as the index set for the product $B^*A$.
These choices allow us to write the products $X$ and $B^*A$ as
\begin{equation}
X=\prod_{k\in J} a_{b_k,i_k}^{e_k}\quad\text{and}\quad B^*A=\prod_{k\in K} a_{b_k,i_k}^{e_k}
\end{equation}
for some choices of $b_k$, $i_k$, and $e_k$.

Using the assumption that $\left|\bw_b\right|\ge \left|\bw_{-b}\right|$, if $(\VP,c)\in\Comp(X)$ then Proposition \ref{prop:domcomp} implies that $\ell_B+1\in\dom_{\VP,c}$, which means that $\ZB_{\VP,c}(\ell_B+1)=\ell_B+2$, whence $(\ell_B+1,\ell_B+2)\in\bar{\VP}^{(c)}$.
Then
\begin{equation}
\rho_N(X)=\sum_{\substack{(\VP,c)\in \Comp(X)\\ (\ell_B+1,\ell_B+2)\in \VP }}\Pt_N(\VP,c)+\sum_{\substack{(\VP,c)\in \Comp(X)\\ (\ell_B+1,\ell_B+2)\not\in \VP }}\Pt_N(\VP,c).
\end{equation}
If $(\VP,c)\in \Comp(X)$ with $ (\ell_B+1,\ell_B+2)\in \VP$ then $(\WP, d)\in \Comp(B^*A)$, where $\WP=\VP\setminus\{(\ell_B+1,\ell_B+2)\}$ and $d$ is the restriction $c|_{\WP}$.
Furthermore, $\bar \WP^{(d)}=\bar{\VP}^{(c)}\setminus \{(\ell_B+1,\ell_B+2)\}$ and $\bar d=\bar c|_{\bar{\WP}}$.
Thus the graph $G_{\VP,c}$ can be formed from the graph $G_{\WP,d}$ by adding the two vertices $(\ell_B+1,\ell_B+2)$ and $1$ arc between these vertices in each direction.
This means that $\gamma(G_{\VP,c})=\gamma(G_{\WP,d})+1$ and $m(G_{\VP,c})=m(G_{\WP,d})+1$, whence $\Pt_N(\VP,c)=\Pt_N(\WP,d)$, and 
\begin{equation}\label{eqn:comsum2pair}
\sum_{\substack{(\VP,c)\in \Comp(X)\\ (\ell_B+1,\ell_B+2)\in \VP }}\Pt_N(\VP,c)=\sum_{(\WP,d)\in \Comp(B^*A)}\Pt_N(\WP,d)= \rho_N(B^*A).
\end{equation}

Now suppose that $(\VP,c)\in \Comp(X)$ with $ (\ell_B+1,\ell_B+2)\not\in \VP$.
We will need the sets $\MS_{a_{b,i}^*A,a_{b,i}^*B}^{+,b, i}(\VP,c)$ and $\MS_{a_{b,i}^*A,a_{b,i}^*B}^{-,b, i}(\VP,c)$ given by Proposition \ref{prop:match}, and we denote them simply by $\MS^{+}(\VP,c)$ and $\MS^{-}(\VP,c)$.
Since we have chosen a different index set than in the statement of Proposition \ref{prop:match}, we have in this case $\MS^+(\VP,c)\subset [\ell_B+2,\ell_A+\ell_B+2]$ and $\MS^-(\VP,c)\subset [\ell_B+1]$.

The assumption $ (\ell_B+1,\ell_B+2)\not\in \VP$ implies that $\pi_\VP(\ell_B+2)\in \MS^+(\VP,c)$ and $\pi_\VP(\ell_B+1)\in \MS^-(\VP,c)$.
In particular, $\pi_\VP(\ell_B+2)\in \MS^-(\VP,c)$ and $\pi_\VP(\ell_B+1)\in \MS^+(\VP,c)$.
If $\WP=\VP\setminus \{(\pi_\VP(\ell_B+2),\ell_B+2), (\ell_B+1, \pi_\VP(\ell_B+1))\}\cup \{(\pi_\VP(\ell_B+2),\pi_\VP(\ell_B+1))\}$ and $d:\WP\to\IndexSet$ is given by $d(p)=c(p)$ for $p\in\VP$ and $d(\pi_\VP(\ell_B+2),\pi_\VP(\ell_B+1))=c(\ell_B+2)=c(\ell_B+1)$ then $(\WP,d)\in \Comp(B^*A)$.
The graphs $G_{\VP,c}$ and $G_{\WP,d}$ have the same number of cycles, but $G_{\VP,c}$ has one more maximal increasing path than $G_{\WP,d}$, whence $\Pt_N(\WP,c)=\frac{1}{N}\Pt_N(\WP,d)$.

The correspondence $(\VP,c)\mapsto(\WP,d)$ is not injective.
Given $(\WP,d)\in \Comp(B^*A)$, one can replace any of the $\bw_1^{(A)}(i)$ pairs $(k,k')$ with $k\in \MS^-(\VP,c)$ and $k'\in \MS^+(\VP,c)$  with the pairs $(k,\ell_B+2)$ and $(\ell_A+2, k')$ to get an element of $\Comp(X)$.
We have thus shown that 
\begin{equation}\label{eqn:comsum2npair}
\sum_{\substack{(\VP,c)\in \Comp(X)\\ (\ell_B+1,\ell_B+2)\not\in \VP }}\Pt_N(\VP,c)=\bw_1^{(A)}(i)\sum_{(\WP,d)\in \Comp(B^*A)}\frac{1}{N}\Pt_N(\WP,d)=\frac{\bw_1^{(A)}(i)}{N} \rho_N(B^*A).
\end{equation}
This proves \eqref{eqn:comstate}. %
\end{proof}

\begin{lemma}
For all $b\in\IndexSet$ and all $i\in\BN$, the creation operator $a_{b,i}^*$ is bounded.
\end{lemma}
\begin{proof}
Let
\begin{equation}
\FS_{<} = \bigoplus_{|\bw_b|< |\bw_{-b}|} H_{\bw}(\HH,\basis)\quad\text{and}\quad \FS_{\le} = \bigoplus_{|\bw_b|\le |\bw_{-b}|} H_{\bw}(\HH,\basis),
\end{equation}
and define $\quad \FS_{\ge}$ and $\FS_{>}$ analogously.
Then we have two decompositions of $\FS_N^\IndexSet(\HH)$:
\begin{equation}
\FS_N^\IndexSet(\HH) = \FS_{\le}\oplus \FS_>=\FS_<\oplus\FS_\ge.
\end{equation}
Using these decompositions, we can write a creation operator $a_{b,i}^*$ as
\begin{equation}
a_{b,i}^* = a_{b,i,<}^*\oplus a_{b,i,\ge}^*
\end{equation}
where $a_{b,i,<}^*: \FS_<\to \FS_\le$ and $a_{b,i,<}^*: \FS_\ge\to \FS_>$ are the restrictions of $a_{b,i}^*$.
It will suffice to show that $a_{b,i,<}^*$ and $a_{b,i,\ge}^*$ are bounded operators.

Boundedness of the operator $a_{b,i,\ge}^*$ is an immediate consequence of Proposition \ref{prop:commutation}, so we need only show that $a_{b,i,<}^*$ is bounded.
This will follow from boundedness of its adjoint, which we denote by $a_{b,i,<}$.
We will find an upper bound for
\begin{equation}
 \rho_N\left(B^*a_{b,i}^*a_{b,i}A\right)
\end{equation}
with $A,B\in \SP_\bw(\HH,\basis)$ having norm $1$ and $|\bw_b|\le |\bw_{-b}|$.
By Lemma \ref{lem:excl}, we can assume that $0\le \bw_b(i)\le |N|$.

Denote by $\ell_A$ the length of the word $A$ and by $\ell_B$ the length of the word $B$.
Let $X=B^*a_{b,i}a_{b,i}^*A$ so that $X$ has length $\ell_A+\ell_B+2$.
Index the product $X$ on the set $J=[\ell_A+\ell_B+2]$ and let $K=J\setminus\{\ell_B+1,\ell_B+2\}$ as the index set for the product $B^*A$.
These choices permit us to write the products $X$ and $B^*A$ as
\begin{equation}
X=\prod_{k\in J} a_{b_k,i_k}^{e_k}\quad\text{and}\quad B^*A=\prod_{k\in K} a_{b_k,i_k}^{e_k}
\end{equation}
for some choices of $b_k$, $i_k$, and $e_k$.

Suppose that $(\VP,c)\in\Comp(X)$.
We will assume here that $b=1$, as the case $b=-1$ is very similar.
By the assumption $|\bw_b|\le |\bw_{-b}|$ and Proposition \ref{prop:domcomp}, $\ell_B+1, \ell_B+2\in\subo_{\VP,c}$.
Thus $(\ell_B+1,\ell_B+2)\in (\bar{\VP},\bar c)$.
Letting $\pi_{\VP}$ be the permutation arrived at by treating $\VP$ as a product of transpositions, $\pi_{\VP}(\ell_B+1)\in [\ell_B]$ and $\pi_{\VP}(\ell_B+2)\in [\ell_B+3, \ell_A+\ell_B+2]$,.
Thus, there exists an increasing path in $G_{\VP,c}$, $\pi_{\VP}(\ell_B+1)\to \ell_B+1\to \ell_B+2\to \pi_{\VP}(\ell_B+2)$.
Replacing this path with a single edge, $(\pi_{\VP}(\ell_B+1), \pi_{\VP}(\ell_B+2))$ gives the graph $G_{\WP,d}$ for $\WP=\VP\setminus \{(\ell_B+1,\ell_B+2)\}\cup (\pi_{\VP}(\ell_B+1), \pi_{\VP}(\ell_B+2))$ and $d(p)=c(p)$ for $p\in \VP$ and $d(\ell_B+1,\ell_B+2)=c(\pi_{\VP}(\ell_B+1), \ell_B+1)$.
The graphs $G_{\WP,d}$ and $G_{\VP,c}$ have the same numbers of cycles and maximal increasing paths, and the correspondence $(\VP,c)\mapsto (\WP,d)$ is $\bw_b(i)$-to-$1$.
Since $0\le \bw_b(i)<|N|$, it follows that $\rho_N\left(B^*a_{b,i}^*a_{b,i}A\right)\le |N|$, whence $\left\|a_{b,i,<}^*\right\|<\sqrt{|N|}$.
\end{proof}

\begin{proposition}\label{prop:wlim}
For $M\in\BO(\FS_N^\IndexSet(\HH))$, $b\in\IndexSet$ and $n\in\BN$ define
\begin{equation}
\Phi_{b,n}(M):=\omega_{b,2n}\cdots \omega_{b,n+1}M\omega_{b,n+1}\cdots \omega_{b,2n}.
\end{equation}
The following limiting relations hold:
\begin{enumerate}
\item $\wlim_{n\to\infty} \Phi_{b,n}(a_{b,i}^*a_{b,i}^*) = 0$; \label{itm:2cr}
\item $\wlim_{n\to\infty} \Phi_{b,n}(a_{b,i}a_{b,i}) = 0$; \label{itm:2an}
\item $\wlim_{n\to\infty} \Phi_{b,n} (a_{b,i}a_{b,i}^*) =1+\frac{1}{N}\NO_{b,i}; $\label{itm:ancr}
\item $\wlim_{n\to\infty} \Phi_{b,n} (a_{b,i}^*a_{b,i}) =\frac{1}{N^2} \NO_{b,i} $;\label{itm:cran}
\item $\wlim_{n\to\infty} \Phi_{b,n} (1) = 1$;\label{itm:one}
\item $\wlim_{n\to\infty} \Phi_{b,n}(\NO_{b,i}) = \NO_{b,i}$.\label{itm:no}
\end{enumerate}
\end{proposition}
\begin{proof}
Items \ref{itm:one} and \ref{itm:no} are straightforward, and item \ref{itm:ancr} follows from Proposition \ref{prop:commutation} together with items \ref{itm:one} and \ref{itm:no}.
Item \ref{itm:2an} will follow immediately from \ref{itm:2cr} and continuity of the map $X\mapsto X^*$ in the weak operator topology.
It will therefore suffice to prove items \ref{itm:2cr} and \ref{itm:cran}.
We will assume that $b=1$ as the case $b=-1$ is similar.

For the proof of \ref{itm:2cr},  fix words $A,B\in \SP(\HH,\basis)$, and let $\ell_A$ be the length of the word $A$ and $\ell_B$ the length of the word $B$.
We index the terms in
\begin{equation}
X_n:=B^*a_{b,2n}\cdots a_{b,n+1}a_{b,i}^*a_{b,i}^*a_{b,n+1}^*\cdots a_{b,2n}^*A,
\end{equation}
considered as a product of the operators $a_{b,k}^e$, with the set
\begin{equation}\label{eqn:Jindexdef}
J_n:=[-\ell_B-n-1, \ell_A+n+1]\setminus\{0\}.
\end{equation}
We write the word $X_n$ as a product
\begin{equation}
X_n:=\prod_{k\in J_n} a_{b_k,i_k}^{e_k}.
\end{equation}
By the definition of $X_n$, $b_k=1$ for $k\in [-n-1,n+1]\setminus\{0\}$, $e_k=1$ for $k\in -[2,n+1]$, $e_k=2$ for $k\in [n+1]\cup\{-1\}$ and $i_k=n+|k|-1$ for $k\in [-n+1,n+1]\setminus \{-1,0,1\}$.
Moreover $i_1=i_{-1}=i$.

Since for $k\in[2,n+1]$, $k$ and $-k$ are the only indices for which  an operator $a_{n+k-1}^e$ appears in the product $X_n$, if  $(\VP,c)\in\Comp(X_n)$ then $(k,-k)\in\VP$ for $k\in [2,n]$.
Assume that $n$ is large enough that $n+|\bw^{A}_{1}|-|\bw^{A}_{-1}|>2r$ for some fixed $r\in\BN$.
Together with Proposition \ref{prop:domcomp}, this condition ensures that $k\in\dom_{\VP,c}$ for all $k\in [2r]$.
One can check that
\begin{equation}
-1\vcequiv-2, \ 1\vcequiv -3,\  2\vcequiv -4,\cdots, 2r\vcequiv -2r-2,
 \end{equation}
 whence $(-2,-1), (-3,1), (-4,2),\cdots, (-r-2,r)\in\VP$.
 This means that the graph $G_{\VP,c}$ has a path 
 \begin{equation}
-2\to 2\to -4\to 4 \to -6\to 6\to \cdots \to -2r\to 2r.
 \end{equation}
 Each arc $(-k,k)$ for $k\in [r]$ is a maximal increasing path, so the cycle containing this path has at least $r$ maximal increasing paths, whence $\Pt_N(\VP,c)\le N^{1-r}$.
 Since the cardinality $\left|\Comp(X_n)\right|$ does not depend on $n$, it follows that
 \begin{equation}
  \rho_{N}(X_n)=\sum_{(\VP,c)\in\Comp(X_n)}\Pt_N(\VP,c)\le CN^{1-r}
 \end{equation}
 for some constant $C$, whence $\Phi_n(a_{b,i}^*a_{b,i}^*)\to 0$ weakly.
 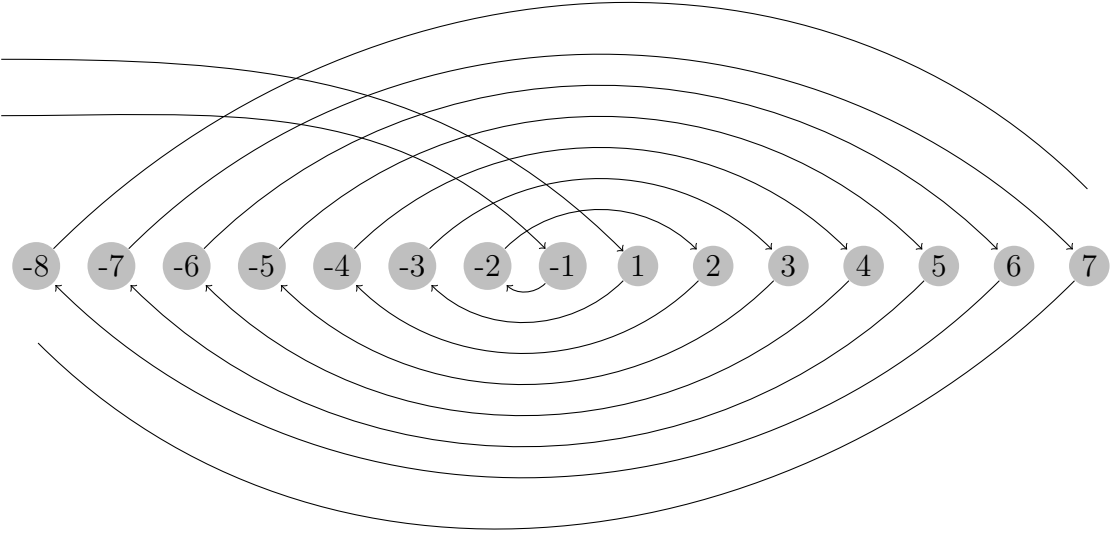
\begin{figure}
 \label{fig:2cr}
  \begin{tikzpicture}[shorten >=1pt,->]
  \tikzstyle{vertex}=[circle,fill=black!25,minimum size=14pt,inner sep=2pt]
  \foreach \x in {-8,-7,-6,-5,-4,-3,-2,-1} 
    \node[vertex] (\x) at (\x,0) {\x};
    \foreach \x in {1,2,3,4,5,6,7} 
      \node[vertex] (\x) at (\x-1,0) {\x};
  \draw [->] (-2) to[out=45,in=135] (2);
\draw [->] (-3) to[out=45,in=135] (3);
\draw [->] (-4) to[out=45,in=135] (4);
\draw [->] (-5) to[out=45,in=135] (5);
\draw [->] (-6) to[out=45,in=135] (6);
\draw [->] (-7) to[out=45,in=135] (7);
\draw [-] (-8) to[out=45,in=135] (6,1);
\draw [->] (-1) to[out=-135,in=-45] (-2);
\draw [->] (1) to[out=-135,in=-45] (-3);
\draw [->] (2) to[out=-135,in=-45] (-4);
\draw [->] (3) to[out=-135,in=-45] (-5);
\draw [->] (4) to[out=-135,in=-45] (-6);
\draw [->] (5) to[out=-135,in=-45] (-7);
\draw [->] (6) to[out=-135,in=-45] (-8);
\draw [-] (7) to[out=-135,in=-45] (-8,-1);
\draw [<-] (-1) to[out=135,in=0] (-8.5,2);
\draw [<-] (1) to[out=135,in=0] (-8.5,2.75);
\end{tikzpicture}
 \caption[A diagram for the proof of Proposition \ref{prop:wlim}]{Part of the directed graph $\hat G_{\VP,c}$ considered in the proof of item \ref{itm:2cr} of \ref{prop:wlim}.}
 \end{figure}

We now move on to proving item \ref{itm:cran}.
Fix  words $A,B\in \SP(\HH, \basis)$ of length $\ell_A$ and $\ell_B$ with $|\bw^{A}|=|\bw^{B}|$.
It will suffice to show that 
\begin{equation}\label{eqn:wlimstate}
{\rho}_N(B^*a_{b,2n}\cdots a_{b,n+1} a_{b,i}^*a_{b,i}a_{b,n+1}^*\cdots a_{b,2n}^*A)=\frac{\bw^{A}_1(i)}{N^2} {\rho}_N(B^*A)
\end{equation}
for $n$ sufficiently large.

We assume that $n$ is large enough that $n+|\bw^{A}_{1}|> |\bw^{A}_{-1}|$ and that no $a_{b,m}^e$ with $m>n$ appears in either word $A$ or $B$.
We let
\begin{equation}
X_n:=B^*a_{b,2n}\cdots a_{b,n+1} a_{b,i}^*a_{b,i}a_{b,n+1}^*\cdots a_{b,2n}^*A,
\end{equation}
and write this product as
\begin{equation}
X_n:=\prod_{k\in K_n} a_{b_k,i_k}^{e_k}.
\end{equation}
where $K_n:=\bigcup_{l\in \{-1,0,1\}}K_n^{(l)}$ with
\begin{equation}
K_{n}^{(-1)}=-[\ell_B],\quad K_n^{0}=\left\{\overline{-n-1},\ldots, \overline{-1}, \overline{1},\cdots,\overline{n+1}\right\},\quad K_n^{(1)}=[\ell_A],
\end{equation}
where for $\overline{k}$ denotes a distinct copy of the integer $k\in\BZ$.
We order $K_n$ by imposing the usual order on $\BZ$ on each $K_{n}^{(l)}$ and for $k\in K_{n}^{(l)}$ and $k'\in K_{n}^{(l')}$ with $l<l'$ set $k<k'$.
Setting $K_n':=K_n^{(-1)}\cup K_n^{(1)}$,
\begin{equation}
B^*A=\prod_{k\in K_n'} a_{b_k,i_k}^{e_k}.
\end{equation}

If $(\VP,c)\in \Comp(X_n)$ then as in the proof of item \ref{itm:2cr}, $(\overline{-k},\overline{k})\in \VP$ for $k\in[2,n+1]$.
Let $k^+,k^-\in K_n'$ be such that $(k^-,\overline{-1}), (\overline{1},k^{+})\in\VP$.
Let
\begin{equation}
\WP:=\VP\setminus\{(\overline{-k},\overline{k}):k\in[2,n]\}\cup\{(k^-,k^+)\}
\end{equation}
and define $d:\WP\to\IndexSet$ by $d(p)=c(p)$ for $c\in \VP\setminus\{(\overline{-k},\overline{k}):k\in[2,n]\}$ and $d(k^-,k^+)=c(k^-)=c(k^+)$.
Then $(\WP,d)\in\Comp(B^*A)$, and $\vls_{\VP,c}(k)=\vls_{\WP,d}(k)$ for any $k\in K_n'$.
By Proposition \ref{prop:domcomp}, $\overline{-1}, \overline{1}\in\dom_{\VP,c}$ with $\overline{-1}\vcequiv \overline{-2}$ and $\overline{1}\vcequiv \overline{2}$ so that $(\overline{-2},\overline{-1}), (\overline{1},\overline{2})\in \bar{\VP}^{(c)}$.
If $k\in [2,n]$ and $\overline{k}\in\dom_{\VP,c}$ then $(\overline{-k},\overline{k})\in\bar{\VP}^{(c)}$ so that $\overline{-k}$ and $\overline{k}$ are on a cycle with exactly $1$ maximal increasing path.
If instead $k\in [2,n]$ with $\overline{k}\in\subo_{\VP,c}$ then $\ZB_{\VP,c}(\overline{k})=\ZB_{\WP,d}(\overline{-k})$ and $\ZB_{\VP,c}(\overline{-k})=\ZB_{\WP,d}(\overline{k})$.

This shows that the correspondence $(\VP,c)\mapsto (\WP,d)$ preserves the cycle structure of the corresponding graph except that it removes some number of cycles with exactly $1$ maximal increasing path and reduces by $2$ the number of maximal increasing paths in the cycle through $k^+$ and $k^-$.
As such, $\Pt_N(\VP,c)=\frac{1}{N^2}\Pt_N(\WP,d)$. 
Similar to the proof of Proposition \ref{prop:commutation}, there are $\bw_1^{(A)}$ pairs $(\VP,c)$ mapped to of each  $(\WP,d)$ by the correspondence just described, so \eqref{eqn:wlimstate} follows.
\end{proof}
\begin{proposition}\label{prop:1dim}
Let $\bzero_b:\BN\to\BZ$ be the constant function $\bzero_b(n)=0$ for $b\in\IndexSet$.
Then $H_\bzero(\HH)=\BC\Omega_N$.
\end{proposition}
\begin{proof}
It is sufficient to show that for any $A\in\SP_\bzero$,
\begin{equation}
\|{\rho}_N(A)\Omega_N-A\Omega_N\|_N=0.
\end{equation}
This follows from the fact that $(\VP,c)\in\Comp (A^*A)$ cannot have any pairs $(k,k')$ with $k\le\ell_A$ and $k'>\ell_A$, where $\ell_A$ is the length of the word $A$.
\end{proof}
The following is a partial
\begin{proposition}\label{prop:vac}
If $N<0$ and $\HH$ is an infinite dimensional real Hilbert space, then
 $\Gamma_N^\IndexSet(\HH)$ contains the projection onto the vacuum vector $\Omega_N$.  
\end{proposition}
\begin{proof}

By Proposition \ref{prop:wlim},
\begin{equation}
\wlim_{n\to\infty}\Phi_n(\omega_{b,i}^2)=1+\frac{N+1}{N^2}\NO_{b,i}.
\end{equation}
In particular, $\NO_{b,i}\in\Gamma_N^\IndexSet(\HH)$.

It is a consequence of the definitions of the operators $\NO_{b,i}$ that
\begin{equation}
\ker \NO_{b,i}=\bigoplus_{\bw_b(i)=0}H_{\bw}(\HH,\basis).
\end{equation}
If $P_{b,i}$ is the projection onto $\ker \NO_{b,i}$ and $P_{\Omega_N}$ is the projection onto the vacuum vector then by Proposition \ref{prop:1dim},
\begin{equation}
P_{\Omega_N}=\inf\{P_{b,i}:b\in\IndexSet, i\in\BN\},
\end{equation}
whence $P_{\Omega_N}\in \Gamma_N^\IndexSet(\HH)$.
\end{proof}

In \cite{BG}, Bo{\.z}ejko and Gu{\c{t}}{\u{a}} used a result analogous to Proposition \ref{prop:vac} to show that a von Neumann algebra under consideration was in fact the whole space of bounded operators.
However, they worked in a setting with a cyclic vacuum vector.
We do not yet know whether the vacuum vector is cyclic in our context.
Accordingly, the best that we can prove is the following.
\begin{proposition}
Let $\hat \FS_N^\IndexSet(\HH):=\overline{\Gamma_N^\IndexSet(\HH)\Omega_N}$, and define $\hat \Gamma_N^\IndexSet(\HH)=\left\{X|_{\hat \FS_N^\IndexSet(\HH)}: X\in \Gamma_N^\IndexSet(\HH)\right\}$.
Then $\hat \Gamma_N^\IndexSet(\HH)=\BO(\hat \FS_N^\IndexSet(\HH))$.
\end{proposition}
\begin{proof}
Suppose that $Q$ is a nonzero projection in $\hat \Gamma_N^\IndexSet(\HH)'$.
For $X\in \hat \Gamma_N^\IndexSet(\HH)$,
\begin{equation}
QX\Omega_N=X(Q\Omega_N)=X\Omega_N.
\end{equation}
In arriving at the second equality, we use the assumption that $Q$ commutes with $P_{\Omega_N}\in\hat \Gamma_N^\IndexSet(\HH)$, the projection onto the vacuum.
Since $\Omega_N$ is cyclic for the action of $\hat \Gamma_N^\IndexSet(\HH)$ on $ \hat \FS_N^\IndexSet(\HH)$, it follows that $Q=1$, whence $\hat \Gamma_N^\IndexSet(\HH)'=\BC$.
\end{proof}

\primarydivision{The $q_{ij}$-product of generalized Brownian motions}
\label{pri:qij}
In this \lprimarydivname, we present a generalization of {Gu{\c{t}}{\u{a}}'s $q$-product of noncommutative generalized Brownian motions.
\begin{definition}
For $\VP\in \PP(\infty)$, define the set of crossings of $\VP$ by
\begin{equation}
 \crossings(\VP)=\{((a_1,z_1),(a_2,z_2))\in \VP\times \VP :a_1<a_2<z_1<z_2\}.
\end{equation}
If $(\VP,c)\in \IPPi$, we also define $\crossings(\VP,c)=\crossings(\VP)$.
Suppose that for each $i\in \IndexSet$, a positive-definite function $\Pt_i:\PP(\infty)\to\BC$ is given, and that we have a (possibly infinite) matrix $Q=(q_{ij})_{i,j\in \IndexSet}$ with $q_{ij}=q_{ji}$ and $q_{ij}\in[-1,1]$.
Then we define the $Q$-product of the $\Pt_i$ to be the function on $\PP^\IndexSet(\infty)$ given by
\begin{equation}
 \left(*^{Q}_{bn \IndexSet} \Pt_b\right)(\VP,c):=\prod_{(p,p')\in\crossings(\VP)}q_{c(p),c(p')}\prod_{b\in \IndexSet} \Pt_b(c^{-1}(b)).
\end{equation}
\end{definition}

\begin{definition}
 A function $\Pt:\IPPi\to\BC$ is said to be multiplicative if for every $k,l,n\in\BN$ with $1\le k\le l\le n$ and any $\IndexSet$-colored pair partitions $\VP_1\in\IPPi(\{1,\ldots,k,l,\ldots,n\})$ and $\VP_2\in\IPPi(\{k+1,\ldots, l-1\})$, we have $\Pt(\VP_1\cup\VP_2)=\Pt(\VP_1)\cdot \Pt(\VP_2)$.
\end{definition}

\begin{proposition}
Suppose that $\Pt_b:\PP(\infty)\to\BC$ ($b\in \IndexSet$) are multiplicative positive definite functions such that $\Pt(\VP,c)=1$ whenever $\VP$ is the element of $\PPi$ with only one pair. Suppose also that for each $i,j\in \IndexSet$, some symmetric $Q=(q_{ij})_{i,j\in \IndexSet}$ with $q_{ij}\in[-1,1]$ is given.
Then $\left(*^{Q}_{a\in \IndexSet} \Pt_b\right)(\VP,c)$ is a positive definite function on $\PP^\IndexSet(\infty)$.
\end{proposition}
\begin{remark}
 In \cite{Guta}, the number of crossings between pairs of different colors was used instead of all crossings.
Of course, if we wish to impose this restriction in our framework, we can assume $q_{bb}=1$ for all $b\in\IndexSet$.
\end{remark}

The proof is essentially the same as the proof of positive definiteness of the $q$-product in \cite{Guta} but we present the argument again here for completeness.
\begin{proof}
 As a first step, we show that for each $\bn:\IndexSet \to\BN$, the kernel $k_\bn$ defined on $\BPP^\IndexSet(\bn,\bzero)$ by
\begin{equation}
 k_\bn(\bpd_1,\bpd_2)=\left(*^{Q}_{b\in \IndexSet} \Pt_b\right)(\bpd_1^*\cdot \bpd_2).
\end{equation}
is positive definite.
Using the definition of the $Q$ product,
\begin{equation}
 k_\bn(\bpd_1,\bpd_2)=\prod_{\substack{(p,p')\in\crossings(\bpd_1^*\cdot\bpd_2)\\p\in (\bpd_1^*\cdot\bpd_2)_b\\p'\in (\bpd_1^*\cdot\bpd_2)_{b'}}}q_{b,b'}\prod_{b\in \IndexSet} \Pt_b((\bpd_1^*\cdot\bpd_2)_b),
\end{equation}
where the subscript $b$ refers to the $b$-colored pair partition.
Since the $\Pt_b$ are positive definite and the pointwise product of positive definite kernels is positive definite, if we can show that
\begin{equation}
 k_\bn'(\bpd_1,\bpd_2)=\prod_{\substack{(p,p')\in\crossings(\bpd_1^*\cdot\bpd_2)\\p\in (\bpd_1^*\cdot\bpd_2)_b\\p'\in (\bpd_1^*\cdot\bpd_2)_{b'}}}q_{b,b'}
\end{equation}
then positive definiteness of $k_n$ will follow.
However, positive definiteness of $k_\bn'$ follows from positivity of the vacuum state on a $*$-algebra generated by annihilation operators $a_{b,i}$ for $i=1,\ldots,\bn(b)$ satisfying the commutation relation
\begin{equation}
 a_{b,i}a_{c,j}^*-q_{b,c}a_{c,j}^*a_{b,i}=\delta_{b,c}\delta_{i,j}.
\end{equation}
Positivity of that state has already been proven by Bo{\.z}ejko and Speicher in \cite{BS1994}.
 
For each $\bn$ denote the complex Hilbert space generated by the positive definite kernel $k_\bn$ by $V_n$ and let $\lambda_\bn:\BPP^\IndexSet(\bn,\bzero)\to V_n$ be the Gelfand map, i.e. $\left<\lambda_\bn(\bpd_1),\lambda_\bn(\bpd_2)\right>=k_\bn(\bpd_1,\bpd_2)$.
The natural action of the symmetric group $S(\bn)$ on $\BPP^\IndexSet(\bn,\bzero)$ preserves $k_\bn$, and thus gives rise to a unitary representation $U_\bn$ on $V_\bn$.
On $V:=\bigoplus_{\bn}V_\bn$, define the operators $j_b$ (for $a\in \IndexSet$) by $j_b\lambda_\bn(\bpd_1)=\lambda_{\bn+\delta_b}(\bpd_{b,0}\cdot\bpd_1)$.
By multiplicativity of $\Pt_b$ ($b\in \IndexSet$),
\begin{equation}
 k_\bn(\bpd_{b,0}\cdot \bpd_1,\bpd_{b,0}\cdot\bpd_2)=k_\bn(\bpd_1,\bpd_2),
\end{equation}
which shows that the definition of $j_b$ makes sense.
Since $j_b$ also satisfies the requisite intertwining property, we have a representation of the $*$-semigroup $\BPP^\IndexSet(\infty)$ on $V$ with respect to the extension of $\left(*^{Q}_{b\in \IndexSet} \Pt_b\right)$ to the broken pair partitions.
\end{proof}
As in the case of \cite{Guta}, we can use this construction to define new positive definite functions on pair partitions provided that our index set $\IndexSet$ is finite.
Assume that $\IndexSet$ is finite and $\Pt_b$ is a multiplicative positive definite function for each $b\in \IndexSet$.
On the Fock-like space $\FS_{\left(*^{Q}_{b\in \IndexSet} \Pt_b\right)}(\HK)$, we can define creation operators
\begin{equation}
 a^*(f):=\frac{1}{\sqrt{|\IndexSet|}}\sum_{b\in \IndexSet} a_b^*(f)
\end{equation}
for $f\in \HK$.
The restriction of the vacuum state to the $*$-algebra generated by the $a^*(f)$ is a Fock state, and we denote the associated positive definite function on $\BPP(\infty)$ by $\left(*^{Q}_{b\in \IndexSet} \Pt_b\right)^{(r)}$.
Explicitly, this function is given by
\begin{equation}
 \left(*^{Q}_{b\in \IndexSet} \Pt_b\right)^{(r)}(\VP)=\frac{1}{|\IndexSet|^{|\VP|}}\sum_{c:\VP\to I}\prod_{(p,p')\in\crossings(\VP)}q_{c(p),c(p')}\prod_{b\in \IndexSet} \Pt_b(c^{-1}(b)).
\end{equation}
In the case that the functions $\Pt_b$ are all the same, $\Pt_b=\Pt$ for all $b\in \IndexSet$, we write $\Pt_Q^{*\IndexSet}$ for $\left(*^{Q}_{b\in \IndexSet} \Pt\right)^{(r)}$, and in the case $\IndexSet=[n]:=\{1,\ldots, n\}$, we write $\Pt_Q^{*n}$.
\begin{remark}
One can ask the question of whether a central limit theorem may be found in this context, similar to the Central Limit Theorem of \cite{Guta}.
The Central Limit Theorem of Gu{\c{t}}{\u{a}} concerns the $q$-product of $n$ copies of a function $\Pt:\PP(\infty)\to\BC$.
By definition, the $q$-product is a function on $\PP^\IndexSet(\infty)$ for some set $\IndexSet$ of cardinality $n$.
However, we can define a function on (uncolored) pair partitions by taking a normalized sum over \textit{all} colorings.
The content of Gu{\c{t}}{\u{a}}'s Central Limit Theorem is that this normalized sum converges to the positive definite function arising from the algebra of $q$-commutation relations, $\Pt_q(\VP)=q^{\crossings(\VP)}$.

An analogous result in the $q_{ij}$ setting will at least require additional assumptions on the $q_{ij}$.
One can assume, for instance, that $q_{ij}\to q$ as $i,j\to\infty$.
In this case, the argument of Gu{\c{t}}{\u{a}} \cite{Guta} can be extended to show an analogous result.
We prefer to pursue a different line of inquiry, namely the case in which the $q_{ij}$ are periodic in the indices $i$ and $j$.
\end{remark}

\begin{theorem}[Central Limit Theorem]\label{thm:CLT}
 Let $Q\in M_N(\BR)$ be an $N\times N$ real symmetric matrix.
Let $Q_n$ be the $n\times n$ symmetric matrix with entries $\tilde q_{ij}$ where $\tilde q_{ij}=q_{\bar i\bar j}$ for $\bar i,\bar j$ such that $1\le \bar i,\bar j\le n$ and $i\equiv \bar i \pmod{N}$, $j\equiv\bar j\pmod{N}$.
Let $\Pt:\PP\to\BC$ be a positive definite multiplicative function such that $\Pt(\VP_1)=1$ where $\VP_1$ is the pair partition consisting of a single pair.
Then $\Pt_{Q_n}^{*n}$ converges pointwise to $\Pt_Q$, where
\begin{equation}
 \Pt_{Q}(\VP)=N^{-|\VP|}\sum_{d:\VP\to[N]}\prod_{(p,p')\in \crossings(\VP)}q_{d(p)d(p')}.
\end{equation}  
\end{theorem}
\begin{proof}
Fix a pair partition $\VP\in\PP(\infty)$.
For a function $c:\VP\to[N]$ denote by $P(c)$ the partition of $\VP$ such that two pairs $p$ and $p'$ are in the same block if and only if $c(p)=c(p')$.
Then
\begin{equation}\label{eqn:Qproddef}
\begin{split}
 \Pt_{Q_n}^{*n}(\VP)&= n^{-|\VP|}\sum_{c:\VP\to[N]}\prod_{(p,p')\in\crossings(\VP)}\tilde q_{c(p),c(p')}\prod_{b\in [n]}\Pt(c^{-1}(b))\\
&=\sum_{\pi\in \Pi(\VP)}n^{-|\VP|}\sum_{\substack{c:\VP\to[n]\\ P(c)=\pi}}\prod_{(p,p')\in\crossings(\VP)}\tilde q_{c(p),c(p')}\prod_{b\in [n]}\Pt(c^{-1}(b)),
\end{split}
\end{equation}
where $\Pi(\VP)$ is the set of all partitions of $\VP$.
We will consider the contribution of the various $\pi\in\Pi(\VP)$ to the sum as $n\to\infty$.

First consider the partition $\pi_1$ of $\VP$ into $|\VP|$ blocks of size $1$, corresponding (for fixed $n$) to injective functions $c:\VP\to[n]$.
For each such $c$ and $a\in [n]$, the pair partition $c^{-1}(b)$ is empty or a single pair, whence $\prod_{b\in [n]}\Pt(c^{-1}(b))=1$.
Furthermore, for each $c$,
\begin{equation}
\prod_{(p,p')\in\crossings(\VP)}\tilde q_{c(p),c(p')}=\prod_{(p,p')\in\crossings(\VP)}q_{\tilde c(p),\tilde c(p')},
\end{equation}
where $\tilde c: \VP\to[N]$ is the map such that $c(p)\equiv \tilde c(p)\pmod{N}$ for all $p\in \VP$.
If $M$ is the natural number such that $MN<n\le (M+1)N$, then for each function $d:\VP\to[N]$, the number $m_n$ of injective maps $c:\VP\to[n]$ such that $\tilde c=d$ is between $M(M-1)\cdots (M-|\VP|+1)$ and $(M+1)(M)\cdots (M-|\VP|)$.
In particular $m_n/n^{|\VP|}\to N^{-|\VP|}$ as $n\to\infty$. 
Thus, as $n\to\infty$ the contribution to the sum in \eqref{eqn:Qproddef} by the term corresponding to $P_1$ converges to
\begin{equation}
 \frac{1}{N^{|\VP|}}\sum_{\substack{d:\VP\to[N]}}\prod_{(p,p')\in\crossings(\VP)}\tilde q_{d(p),d(p')}=\Pt_{Q}(\VP).
\end{equation}

Now we will show that any other partition $\pi\ne\pi_1$ contributes $0$ to the sum in \eqref{eqn:Qproddef} in the limit as $n\to\infty$.
Such a partition $\pi$ has at most $|\VP|-1$ blocks.
For a given $n$, the number of maps $c:\VP\to[n]$ with $P(c)=\pi$ is  
\begin{equation}
n(n-1)\cdots(n-|P|+1)\le n(n-1)\cdots(n-|\VP|+2)<n^{|\VP|-1}.
\end{equation}
Thus, the contribution of the term indexed by $P$ is indeed $0$ in the limit.
\end{proof}
\bibliographystyle{amsalpha}
\bibliography{bibliography}
\end{document}